\documentclass[12pt]{amsart}
\usepackage{amsmath, amssymb, amsfonts}
\usepackage[all]{xypic}
\usepackage[pdftex]{graphicx}
\usepackage[usenames]{color}
\usepackage{stmaryrd}
\usepackage[utf8]{inputenc}
\usepackage{bbold}

\theoremstyle{plain}
\newtheorem{theorem}{Theorem}[section]
\newtheorem{lemma}[theorem]{Lemma}
\newtheorem{corollary}[theorem]{Corollary}
\newtheorem{prop}[theorem]{Proposition}

\newtheorem{conj}[theorem]{Conjecture}

\theoremstyle{remark}

\newtheorem{remark}[theorem]{Remark}

\newtheorem{example}[theorem]{Example}

\newtheorem*{note*}{Note}
\newtheorem*{remark*}{Remark}
\newtheorem*{example*}{Example}

\theoremstyle{definition}
\newtheorem*{definition*}{Definition}
\newtheorem*{hypothesis*}{Hypothesis}
\newtheorem*{assumption*}{Assumption}
\newtheorem{definition}[theorem]{Definition}

\newcommand{\Z}{\mathbb{Z}}
\newcommand{\R}{\mathbb{R}}
\newcommand{\Q}{\mathbb{Q}}
\newcommand{\C}{\mathbb{C}}
\newcommand{\N}{\mathbb{N}}

\newcommand{\Aut}{\mathrm{Aut}}
\newcommand{\Gal}{\mathrm{Gal}}
\newcommand{\Tr}{\mathrm{Tr}}
\newcommand{\GL}{\mathrm{GL}}

\newcommand{\ab}{\mathrm{ab}}

\newcommand{\Nrd}{\mathrm{Nrd}}

\newcommand{\Hom}{\mathrm{Hom}}

\newcommand{\res}{\mathrm{res}}
\newcommand{\quot}{\mathrm{quot}}
\newcommand{\infl}{\mathrm{infl}}
\newcommand{\id}{\mathrm{id}}

\newcommand{\Irr}{\mathrm{Irr}}
\newcommand{\ind}{\mathrm{ind}}
\newcommand{\Spec}{\mathrm{Spec}}
\newcommand{\PMod}{\mathrm{PMod}}
\newcommand{\perf}{\mathrm{perf}}
\newcommand{\Det}{\mathrm{Det}}
\newcommand{\ram}{\mathrm{ram}}
\newcommand{\ur}{\mathrm{ur}}
\newcommand{\aug}{\mathrm{aug}}
\newcommand{\Ver}{\mathrm{Ver}}
\newcommand{\pd}{\mathrm{pd}}
\newcommand{\Cone}{\mathrm{Cone}}
\newcommand{\Ext}{\mathrm{Ext}}

\renewcommand{\det}{\mathrm{det}}
\numberwithin{equation}{section}

\title[A main conjecture for local fields]{An equivariant Iwasawa main conjecture\\ for local fields}

\author{Andreas Nickel}
\address{Universit\"{a}t Duisburg--Essen\\
	Fakult\"{a}t f\"{u}r Mathematik\\
	Thea-Leymann-Str. 9\\
	45127 Essen\\
	Germany}
\email{andreas.nickel@uni-due.de}
\urladdr{https://www.uni-due.de/$\sim$hm0251/english.html}

\date{Version of 15th March 2018}
\begin{document}

\begin{abstract}
	Let $L/K$ be a finite Galois extension of $p$-adic fields and let
	$L_{\infty}$ be the unramified $\Z_p$-extension of $L$.
	Then $L_{\infty}/K$ is a one-dimensional $p$-adic Lie extension.
	In the spirit of the main conjectures of equivariant Iwasawa theory,
	we formulate a conjecture which relates the equivariant local
	epsilon constants attached to the finite Galois intermediate extensions
	$M/K$ of $L_{\infty}/K$
	to a natural arithmetic invariant arising from the
	\'etale cohomology of the constant sheaf $\Q_p/\Z_p$ on the spectrum
	of $L_{\infty}$. We give strong evidence of the conjecture
	including a full proof in the case that $L/K$ is at most
	tamely ramified.
\end{abstract}

\maketitle

\section{Introduction}
Let $E/F$ be a finite Galois extension of number fields with Galois
group $G$. If $E/F$ is tamely ramified, then the ring of integers
$\mathcal{O}_E$ in $E$ is projective as a module over the integral
group ring $\Z[G]$. The study of the Galois module structure of
$\mathcal{O}_E$ for tamely ramified extensions was systematically developed
by Fr\"ohlich (see \cite{MR717033} for a survey) 
and culminated in Taylor's proof \cite{MR608528}
of Fr\"ohlich's conjecture
that the class of $\mathcal{O}_E$ in the locally free class group
of $\Z[G]$ is determined by the Artin root numbers associated to the
irreducible complex symplectic characters of $G$.
Subsequently, Chinburg \cite{MR786352} formulated a generalization
of Fr\"ohlich's conjecture to the context of arbitrary 
finite Galois extensions $E/F$. This is often called
`Chinburg's $\Omega_2$-conjecture' and is in general still wide open.

Motivated by the requirement that the equivariant Tamagawa number conjecture
(as formulated by Burns and Flach \cite{MR1884523}) 
for the pair $(h^0(\Spec(E)), \Z[G])$ and its Kummer dual
$(h^0(\Spec(E))(1), \Z[G])$ should be compatible
with the functional equation of the associated equivariant $L$-functions,
Bley and Burns \cite{MR2005875} have formulated the `global equivariant
epsilon constant conjecture'. This conjecture asserts an equality
in the relative algebraic $K$-group $K_0(\Z[G], \R)$ between
an element constructed from epsilon constants and the sum of an
equivariant discriminant and certain terms coming from 
the \'etale cohomology of $\mathbb G_m$.
 Note that there is a natural
surjective morphism from $K_0(\Z[G], \R)$ to the locally free class group
of $\Z[G]$. The projection of the global equivariant
epsilon constant conjecture under this morphism indeed recovers Chinburg's
$\Omega_2$-conjecture. 
One advantage of the refinement of Bley and Burns is that
it naturally decomposes into `$p$-parts', where $p$ runs over all rational
primes.

Now fix a prime $p$ and let $L/K$ be a finite Galois extension of
$p$-adic fields with Galois group $G$. Breuning \cite{MR2078894}
defined an invariant $R_{L/K}$ in the relative algebraic $K$-group
$K_0(\Z_p[G], \Q_p^c)$, where $\Q_p^c$ is a fixed algebraic closure of $\Q_p$.
This invariant incorporates the equivariant local epsilon constant of
$L/K$ (i.e.~local Galois Gauss sums) and a natural 
arithmetic invariant arising from the \'etale
cohomology of the sheaf $\Z_p(1)$ on the spectrum of $L$.
His `local equivariant epsilon constant conjecture' then simply asserts
that $R_{L/K}$ vanishes. This fits into the very general framework
of local noncommutative Tamagawa number conjectures of
Fukaya and Kato \cite{MR2276851}.

The global and the local conjecture are in fact
closely related. Let $v$ be a $p$-adic place of $F$ and fix a place $w$
of $E$ above $v$. We write $F_v$ and $E_w$ for the completions of $F$ at
$v$ and $E$ at $w$, respectively. Then the $p$-part of the global conjecture
for $E/F$ is implied by the local conjectures for the extensions $E_w/F_v$,
where $v$ ranges over all $p$-adic places of $F$ 
(see \cite[Corollary 4.2]{MR2078894}).
If $p$ is odd, one actually knows that the $p$-part of the 
global conjecture for all Galois extensions
of number fields is equivalent to the local conjecture for all Galois
extensions of $p$-adic fields \cite[Theorem 4.3]{MR2078894}.

It therefore suffices to consider Breuning's conjecture.
The invariant $R_{L/K}$ is of the form
\[
	R_{L/K} = T_{L/K} + C_{L/K} + U_{L/K} - M_{L/K},
\]
where each term lies in $K_0(\Z_p[G], \Q_p^c)$. Here, the term $T_{L/K}$ is
the equivariant local epsilon constant, $U_{L/K}$ is the so-called
unramified term (see \S \ref{subsec:unramified-term}) and $M_{L/K}$
is a certain correction term.
We now briefly recall the definition of the cohomological term $C_{L/K}$.
Define a free $\Z_p[G]$-module $H_L := \bigoplus_{\sigma} \Z_p$,
where the sum ranges over all embeddings $\sigma: L \rightarrow \Q_p^c$.
Then
\begin{equation} \label{eqn:intro-complex}
	 K_L^{\bullet} := R\Gamma(L, \Z_p(1))[1] \oplus H_L[-1] 
\end{equation}
is a perfect complex of $\Z_p[G]$-modules which is acyclic outside
degrees $0$ and $1$. Moreover, we have natural isomorphisms
$H^1(K_L^{\bullet}) \simeq \Z_p \oplus H_L$ and
$H^0(K_L^{\bullet}) \simeq \widehat{L^{\times}}$, the $p$-completion
of $L^{\times}$. The valuation map, the $p$-adic logarithm and the various
embeddings $\sigma$ then induce a $\Q_p^c[G]$-isomorphism 
(see \cite[\S 2.4]{MR2078894})
\[
	\phi_L:\Q_p^c \otimes_{\Z_p} H^0(K_L^{\bullet}) \simeq
	\Q_p^c \otimes_{\Z_p} H^1(K_L^{\bullet}).
\]
These data can then be used to define $C_{L/K}$ as the `refined Euler characteristic'
(see \S \ref{subsec:refined-Euler})
\[
	C_{L/K} := \chi_{\Z_p[G], \Q_p^c}(K_L^{\bullet}, \phi_L^{-1})
	\in K_0(\Z_p[G], \Q_p^c).
\]
This definition only depends upon the
trivialization $\phi_L^{-1}$ and the class in 
$\Ext^2_{\Z_p[G]}(\Z_p, \widehat{L^{\times}})$ that is 
naturally determined by the
complex $R\Gamma(L, \Z_p(1))$. This is essentially the fundamental class
of local class field theory.

In this paper we approach the local equivariant epsilon constant conjecture
via Iwasawa theory. We formulate an equivariant
Iwasawa main conjecture, which might be seen as a local analogue of the
main conjecture of equivariant Iwasawa theory for totally real fields
proven by Ritter and Weiss \cite{MR2813337} and, independently, by
Kakde \cite{MR3091976} (under the assumption that Iwasawa's $\mu$-invariant
vanishes; see \cite{hybrid-EIMC} for results without this hypothesis). 

Every $p$-adic field $L$ has at least two $\Z_p$-extensions:
the cyclotomic and the unramified $\Z_p$-extension. 
It is more common in the literature to look at 
the cyclotomic $\Z_p$-extension,
but also the unramified $\Z_p$-extension is often considered
\cite{TV-mc, MR3194646, MR3273476, MR3323528}.
In order to explain why we believe that the unramified $\Z_p$-extension
bears interesting information in our case, we consider the following more
general situation. Let $V$ be a finite dimensional $\Q_p$-vector space
with a continuous action of $G_K$, the absolute Galois group of $K$.
Choose a $G_K$-stable $\Z_p$-lattice $T$ in $V$ and denote the quotient
$V/T$ by $A$. There are natural duals of $T$ and $V$ given by
$T^{\ast} := \Hom_{\Z_p}(T, \Z_p(1))$ and
$V^{\ast} := \Hom_{\Q_p}(V, \Q_p(1)) = \Q_p \otimes_{\Z_p} T^{\ast}$.
Let $K(A)/K$ be the extension defined by the kernel of
the homomorphism
$G_K \rightarrow \Aut(V)$. Then $K(A) = \bigcup_n K(A[p^n])$
is the field obtained by adjoining all $p$-power torsion points of $A$
and $K(A)/K$ is a compact $p$-adic Lie extension.
Iwasawa theory over a $p$-adic Lie extension $K_{\infty}$ of $K$
often behaves well when $K_{\infty}$ contains $K(A)$ as a subfield
(see \cite[\S 4.3]{MR1894938}, for instance; similarly
for number fields \cite{MR2217048}).
In the case considered in this article, the lattice $\Z_p(1)$ plays
the role of $T^{\ast}$. Thus we have $A = \Q_p/\Z_p$ and then clearly
$K(A) = K$ so that every choice of $K_{\infty}$ will contain $K(A)$.
Note that for all other Tate twists of $\Z_p$ and in many further interesting
cases as the $p$-adic Tate module of an elliptic curve, the requirement
$K(A) \subseteq K_{\infty}$ implies that $K_{\infty}$ contains the
cyclotomic $\Z_p$-extension.

Let us consider the unramified
$\Z_p$-extension $L_{\infty}$ of $L$. Then $L_{\infty}/K$ is an infinite
Galois extension and its Galois group $\mathcal{G}$ is a one-dimensional
$p$-adic Lie group. We let $\Lambda(\mathcal{G})$ denote the Iwasawa algebra
of $\mathcal{G}$ and let $\mathcal{Q}(\mathcal{G})$ be its total ring of fractions.
We also put $\mathcal{Q}^c(\mathcal{G}) := \Q_p^c \otimes_{\Q_p} \mathcal{Q}(\mathcal{G})$.

Now assume that $p$ is odd. Although we never need this assumption 
for our arguments, we have to impose it whenever we refer to
results of Ritter and Weiss, where it is always in force.
The local Galois
Gauss sums behave well under unramified twists (see Proposition
\ref{prop:Gauss-twists} below) 
and give rise to a homomorphism
$\tau_{L_{\infty}/K}$
on the ring of virtual $\Q_p^c$-valued characters of $\mathcal{G}$
with open kernel. This homomorphism takes values in 
$\mathcal{Q}^c(\Gamma_K)^{\times}$, 
where $\Gamma_K := \Gal(K_{\infty}/K) \simeq \Z_p$,
and plays the role
of $T_{L/K}$ above (in fact, the homomorphism
$\tau_{L_{\infty}/K}$ depends upon a choice of isomorphism 
$\C \simeq \C_p$,
but our conjecture does not; we will suppress this dependence in the introduction).

For $n \in \N$ let $L_n$ be the $n$-th layer of the $\Z_p$-extension $L_{\infty}/L$.
We define a complex of $\Lambda(\mathcal{G})$-modules
\[
	 K_{L_{\infty}}^{\bullet}  :=  R\Hom(R\Gamma(L_{\infty}, \Q_p / \Z_p), \Q_p / \Z_p)[-1]
	 \oplus H_{L_{\infty}}[-1],
\]
where $H_{L_{\infty}} := \varprojlim_n H_{L_n}$ is a free
$\Lambda(\mathcal{G})$-module of rank $[K:\Q_p]$, 
and show that this complex is indeed perfect.
To see the analogy with \eqref{eqn:intro-complex},
we observe that local Tate duality induces an isomorphism
\[
	 R\Gamma(L, \Z_p(1))[1] \simeq  R\Hom(R\Gamma(L, \Q_p / \Z_p), \Q_p / \Z_p)[-1]
\]
in the derived category of $\Z_p[G]$-modules. We construct a trivialization
$\phi_{\infty}^{-1}$ of the complex $K_{L_{\infty}}^{\bullet}$ which allows us
to define a refined Euler characteristic 
\[
	C_{L_{\infty}/K} \in K_0(\Lambda(\mathcal{G}), \mathcal{Q}^c(\mathcal{G})).
\]
In contrast to the isomorphism $\phi_L$ above, the map $\phi_{\infty}$
no longer incorporates the valuation map because $\Z_p$ becomes torsion when
considered as an Iwasawa module. For technical purposes, however,
we have to choose a compatible system of integral normal basis generators
along the unramified tower (mainly because we will refer to results
that require coefficient rings with finite residue field
so that we cannot pass to the completion of the ring of integers
in $F_{\infty}$ for a $p$-adic field $F$). 
The same choice will appear in our definition of
(a variant of) the unramified term $U_{L_{\infty}/K}'$. The main conjecture
will then not depend upon this choice. We will also define a certain correction term
$M_{L_{\infty}/K}$.

The main conjecture then asserts the following:
There exists (a unique)  
\[ 
	\zeta_{L_{\infty}/K} \in 
	K_1(\mathcal{Q}^c(\mathcal{G}))
\]
such that
	\[
	\partial(\zeta_{L_{\infty}/K}) = 
	-C_{L_{\infty}/K} - U'_{L_{\infty}/K} + M_{L_{\infty}/K}
	\]
	and
	\[
	\Det(\zeta_{L_{\infty}/K}) = \tau_{L_{\infty}/K}.
	\]
Here, $\partial:K_1(\mathcal{Q}^c(\mathcal{G})) \rightarrow
K_0(\Lambda(\mathcal{G}), \mathcal{Q}^c(\mathcal{G}))$ 
denotes the (surjective) connecting homomorphism of
the long exact sequence of relative $K$-theory
and $\Det$ is a homomorphism mapping $K_1(\mathcal{Q}^c(\mathcal{G}))$
to a certain $\Hom$-group (constructed by Ritter and Weiss \cite{MR2114937}).
The analogy to the main conjecture for totally real fields as formulated
by Ritter and Weiss is apparent. To make the analogy to Breuning's
conjecture clearer, let us assume only for the rest of this paragraph 
that $\mathcal{G}$ is abelian
or, more generally, that $\Det$ is an isomorphism. Then we may put
$T_{L_{\infty}/K} := \partial(\Det^{-1}(\tau_{L_{\infty}/K}))$
and 
\[
	R_{L_{\infty}/K} := T_{L_{\infty}/K} + C_{L_{\infty}/K} + U'_{L_{\infty}/K}
	-M_{L_{\infty}/K} \in  K_0(\Lambda(\mathcal{G}), \mathcal{Q}^c(\mathcal{G})).
\]
Then the main conjecture asserts that $R_{L_{\infty}/K}$ vanishes.

Our conjecture also fits into the framework of local noncommutative
Tamagawa number conjectures of Fukaya and Kato
\cite{MR2276851}. However, \cite[Conjecture 3.4.3]{MR2276851}
only asserts that `there exists a unique way to associate an
isomorphism' (called an $\varepsilon$-isomorphism) with certain 
properties for any pair
$(\Lambda,T)$ of certain adic rings $\Lambda$ and finitely generated 
projective $\Lambda$-modules $T$ endowed with a continuous action of
$G_{\Q_p}$. They do not explain how this isomorphism
can (at least conjecturally) be constructed in general.
In our situation this amounts to the definition of the trivialization
of the complex $K_{L_{\infty}}^{\bullet}$.
Therefore our conjecture makes the conjecture of Fukaya and Kato 
more precise in the situation $K = \Q_p$, $\Lambda = \Lambda(\mathcal{G})$
and $T = \Lambda^{\sharp}(1)$, where $\Lambda^{\sharp}$ denotes the free
$\Lambda$-module of rank $1$ upon which $\sigma \in G_{\Q_p}$ acts
as multiplication by $\overline{\sigma}^{-1}$; here $\overline{\sigma}$
denotes the image of $\sigma$ in $\mathcal{G}$.

Building on work of Fr\"ohlich, Bley, Burns, and Breuning 
we show that our conjecture
holds for tamely ramified extensions. 
If $\mathcal{G}$ is abelian, then this is the local analogue of
Wiles' result \cite{MR1053488} on the main conjecture for
totally real fields.
This allows us to deduce the conjecture
`over the maximal order' from its good functorial behaviour. 
Note that $p$ does not divide $[L_{\infty} : K_{\infty}]$ if $L/K$
is tamely ramified, and thus $\Lambda(\mathcal{G})$ is itself a maximal
order (over the classical Iwasawa algebra $\Z_p \llbracket T \rrbracket$)
in this case.

We give an important application of our results, which 
has no analogue at finite level: 
it suffices to prove the main conjecture after
localization at the height $1$ prime ideal $(p)$
of $\Z_p \llbracket T \rrbracket$. The cohomology groups of
the complex $K_{L_{\infty}}^{\bullet}$ then become free
(and thus perfect)
by a result of the author \cite{Swan-Iwasawa}
and so one does not need to take care of the associated 
extension class any longer. This application makes heavy use of
a result of Ritter and Weiss \cite{MR2205173}
on the image of $K_1(\Lambda(\mathcal{G}))$ under $\Det$.
Note that a similar reduction step appears in the proof of the main
conjecture for totally real fields.

In a forthcoming article we will show that our conjecture implies
Breuning's conjecture and also the 
equivariant local epsilon constant conjecture
for unramified twists of $\Z_p(1)$. 
Note that (with a few exceptions) these conjectures are known to hold
in exactly the same cases: for tamely ramified extensions \cite{MR2078894, MR3509721},
for certain weakly, but wildly ramified extensions \cite{MR3520002, MR3651025},
and if $L/\Q_p$ is an abelian extension \cite{MR2078894, MR3194646}
(see also \cite{MR2290586, MR2410782}).
So our work explains this analogy and provides a unifying approach
to these results. Moreover, it overcomes two major obstacles to proving
Breuning's conjecture: (i) the valuation map no longer appears
and so the trivialization of the complex is considerably easier and (ii)
one may reduce to a situation where the occurring complex has
perfect cohomology groups 
and so one does not need to take care of the extension 
class.\\

This article is organized as follows. 
In \S \ref{sec:alg-prelim} we review algebraic $K$-theory of $p$-adic
group rings and Iwasawa algebras.
In particular, we study how $K$-theory behaves
when one passes from group rings to Iwasawa algebras.
We introduce the determinant map
of Ritter--Weiss and compare it with the equivalent notion
of the reduced norm. In \S \ref{sec:Gauss-sums} we introduce
local Galois Gauss sums and study their behaviour under unramified
twists. We define the homomorphism $\tau_{L_{\infty}/K}$
and study its basic properties. The main part of \S \ref{sec:coh-term}
is devoted to the definition of the cohomological term $C_{L_{\infty}/K}$.
We introduce the complex $K_{L_{\infty}}^{\bullet}$ and show that it
is perfect. This in fact holds for more general $\Z_p$-extensions.
We then study (normal) integral basis generators and the behaviour
of the $p$-adic logarithm along the unramified tower.
Choosing a certain compatible system of normal integral basis generators,
we define a trivialization of the complex $K_{L_{\infty}}^{\bullet}$.
A similar choice will then appear in the definition of the unramified
term $U'_{L_{\infty}/K}$. 
We show that the sum $C_{L_{\infty}/K} + U'_{L_{\infty}/K}$
is well defined up to the image of an element 
$x \in K_1(\Q_p^c \otimes_{\Z_p} \Lambda(\mathcal{G}))$ such that 
$\Det(x)=1$. This will be sufficient for our purposes, but we point out
that it is conjectured that the map $\Det$ is injective. For instance,
this is true if $\mathcal{G}$ is abelian or, more generally, if
$p$ does not divide the order of the (finite) commutator
subgroup of $\mathcal{G}$ (this follows from 
\cite[Proposition 4.5]{MR3092262} as explained in
\cite[Remark 4.8]{hybrid-EIMC}). 
We also define the correction
term in this section. We formulate the main conjecture in 
\S \ref{sec:MC}. We show that it is well posed and study its functorial
properties. We also provide some first evidence including a result
that does not have an analogue at finite level.
In \S \ref{sec:max-orders} we prove our conjecture `over the maximal
order'. This includes a full proof of the conjecture for tamely ramified
extensions. As a corollary, we obtain an important reduction step
toward a full proof of the conjecture.

\subsection*{Acknowledgements}
The author acknowledges financial support provided by the 
Deutsche Forschungsgemeinschaft (DFG) 
within the Heisenberg programme (No.\, NI 1230/3-1).
I would like to thank Werner Bley for various enlightening discussions 
concerning local and global epsilon constant conjectures during my stay
at LMU Munich. 
I also thank Otmar Venjakob for his constructive suggestions on
an earlier version of this article.
Finally, I owe Henri Johnston a great debt of gratitude
for his interest in this project and the many valuable comments,
which helped to considerably improve this article.

\subsection*{Notation and conventions}
All rings are assumed to have an identity element and all modules are assumed
to be left modules unless otherwise stated.
Unadorned tensor products will always denote tensor products over $\Z$.
If $K$ is a field, we denote its absolute Galois group by $G_K$.
If $R$ is a ring, we write $M_{m \times n}(R)$ for the set of all 
$m \times n$ matrices with entries in $R$. We denote the group of invertible
matrices in $M_{n \times n}(R)$ by $\GL_n(R)$.
Moreover, we let $\zeta(R)$ denote the centre of the ring $R$.
If $M$ is an $R$-module we denote by $\pd_R(M)$ the projective dimension
of $M$ over $R$.

\section{Algebraic Preliminaries} \label{sec:alg-prelim}

\subsection{Derived categories}
Let $\Lambda$ be a noetherian ring and $\PMod(\Lambda)$ be the category of all finitely
generated projective $\Lambda$-modules. We write $\mathcal D(\Lambda)$ for the derived category
of $\Lambda$-modules and $\mathcal C^b(\PMod(\Lambda))$ for the category of bounded complexes
of finitely generated projective $\Lambda$-modules.
Recall that a complex of $\Lambda$-modules is called perfect if it is 
isomorphic in $\mathcal D(\Lambda)$
to an element of $\mathcal C^b(\PMod(\Lambda))$.
We denote the full triangulated subcategory of $\mathcal D(\Lambda)$
comprising perfect complexes by $\mathcal D^{\perf}(\Lambda)$.
If $M$ is a $\Lambda$-module and $n$ is an integer, we write $M[n]$ for the complex
\[
	\dots \longrightarrow 0 \longrightarrow 0 \longrightarrow M 
	\longrightarrow 0 \longrightarrow 0 \longrightarrow \dots.
\]
where $M$ is placed in degree $-n$. This is compatible with the usual shift
operator on cochain complexes.


\subsection{Relative Algebraic $K$-theory}\label{subsec:K-theory}
For further details and background on algebraic $K$-theory used in this section, 
we refer the reader to
\cite{MR892316} and \cite{MR0245634}.
Let $\Lambda$ be a noetherian ring.
We write $K_{0}(\Lambda)$ for the Grothendieck group of $\PMod(\Lambda)$ 
(see \cite[\S 38]{MR892316})
and $K_{1}(\Lambda)$ for the Whitehead group (see \cite[\S 40]{MR892316})
which is the abelianized infinite general linear group.
We denote the relative algebraic $K$-group corresponding to a ring
homomorphism $\Lambda \rightarrow \Lambda'$ by $K_0(\Lambda, \Lambda')$.
We recall that $K_{0}(\Lambda, \Lambda')$ is 
an abelian group with generators $[X,g,Y]$ where
$X$ and $Y$ are finitely generated projective $\Lambda$-modules
and $g:\Lambda' \otimes_{\Lambda} X \rightarrow \Lambda' \otimes_{\Lambda} Y$ 
is an isomorphism of $\Lambda'$-modules;
for a full description in terms of generators and relations, 
we refer the reader to \cite[p.\ 215]{MR0245634}.
Furthermore, there is a long exact sequence of relative $K$-theory
(see  \cite[Chapter 15]{MR0245634})
\begin{equation}\label{eqn:long-exact-seq}
K_{1}(\Lambda) \longrightarrow K_{1}(\Lambda') \xrightarrow{\partial_{\Lambda,\Lambda'}}
K_{0}(\Lambda,\Lambda') \longrightarrow K_{0}(\Lambda) 
\longrightarrow K_{0}(\Lambda').
\end{equation}

Let $R$ be a noetherian integral domain of characteristic $0$ with field of fractions $E$.
Let $A$ be a finite-dimensional semisimple $E$-algebra and let $\Lambda$ be an $R$-order in $A$.
For any field extension $F$ of $E$ we set $A_{F} := F \otimes_{E} A$.
Let $K_{0}(\Lambda, F) = K_0(\Lambda, A_F)$ denote the relative
algebraic $K$-group associated to the ring homomorphism 
$\Lambda \hookrightarrow A_{F}$. We then abbreviate 
the connecting homomorphism $\partial_{\Lambda,A_F}$
to $\partial_{\Lambda,F}$.
The reduced norm map
\[
\Nrd_{A}: A \longrightarrow \zeta(A)
\]
is defined componentwise on the Wedderburn decomposition of $A$ (see \cite[\S 9]{MR1972204})
and extends to matrix rings over $A$ in the obvious way; hence this induces
a map $K_{1}(A) \rightarrow \zeta(A)^{\times}$ which we also denote by $\Nrd_A$.

Let $\zeta(A) = \prod_{i} E_i$ be the decomposition of $\zeta(A)$ into a product
of fields. For any $x = (x_i)_i \in \zeta(A)$ we define an invertible
element $^{\ast}x = (^{\ast}x_i)_i \in \zeta(A)^{\times}$ by
$^{\ast}x_i := x_i$ if $x_i \not=0$ and $^{\ast}x_i = 1$ if $x_i = 0$.

\subsection{Refined Euler characteristics} \label{subsec:refined-Euler}
For any $C^{\bullet} \in \mathcal C^b (\PMod (\Lambda))$ we
define $\Lambda$-modules
\[
 C^{ev} := \bigoplus_{i \in \Z} C^{2i}, \quad C^{odd} := \bigoplus_{i \in \Z} C^{2i+1}.
\]
Similarly, we define $H^{ev}(C^{\bullet})$ and $H^{odd}(C^{\bullet})$ 
to be the direct sum over all even and odd degree
cohomology groups of $C^{\bullet}$, respectively.
A pair $(C^{\bullet},t)$
consisting of a complex $C^{\bullet} \in \mathcal D^{\perf}(\Lambda)$ and an
isomorphism $t: H^{odd}(C_F ^{\bullet}) \rightarrow H^{ev}(C_F^{\bullet})$ is called a
trivialized complex, where we write $C_F^{\bullet}$ for
$F \otimes^{\mathbb L}_R C^{\bullet}$.
We refer to $t$ as a trivialization of $C^{\bullet}$.

One defines the refined Euler characteristic 
$\chi_{\Lambda,F}(C^{\bullet}, t) \in K_0(\Lambda,F)$ of a trivialized complex as follows:
Choose a complex $P^{\bullet} \in \mathcal C^b(\PMod(\Lambda))$ which is
quasi-isomorphic to $C^{\bullet}$. Let $B^i(P_F ^{\bullet})$ and 
$Z^i(P_F^{\bullet})$ denote the $i$-th cobounderies and $i$-th cocycles of
$P_F ^{\bullet}$, respectively. 
For every $i \in \Z$ we have the obvious exact sequences
\[ 0 \longrightarrow {B^i(P_F^{\bullet})} \longrightarrow {Z^i(P_F^{\bullet})} \longrightarrow
 {H^i(P_F^{\bullet})} \longrightarrow 0,
\]
\[
   0 \longrightarrow {Z^i(P_F^{\bullet})} \longrightarrow {P_F^i} \longrightarrow 
   {B^{i+1}(P_F^{\bullet})} \longrightarrow 0. 
\]
If we choose splittings of the above sequences, we get an
isomorphism of $A_F$-modules
\[   \phi_t: P_F^{odd}  \simeq  \bigoplus_{i \in \Z} B^i(P_F^{\bullet}) \oplus
      H^{odd}(P_F^{\bullet})
      \simeq  \bigoplus_{i \in \Z} B^i(P_F^{\bullet})  \oplus H^{ev}(P_F^{\bullet})
      \simeq  P_F^{ev},
\]
where the second map is induced by $t$. Then the refined
Euler characteristic is defined to be
\[
	\chi_{\Lambda, A_F} (C^{\bullet}, t) = 
	\chi_{\Lambda, F} (C^{\bullet}, t) := 
	[P^{odd}, \phi_t, P^{ev}] \in K_0(\Lambda, F)
\]
which indeed is independent of all choices made in the
construction.
For further information concerning refined Euler characteristics
we refer the reader to \cite{MR2076565}.

\subsection{$p$-adic group rings}
Let $G$ be a finite group and $F$ a field of charactersitic $0$.
We write $\Irr_F(G)$ for the set of $F$-irreducible characters of $G$.
We fix an algebraic closure $F^c$ of $F$ and let $G_F := \Gal(F^c/F)$ denote
the absolute Galois group of $F$.

Fix a prime $p$ and set $\Irr(G) := \Irr_{\Q_p^c}(G)$. 
Then $G_{\Q_p}$ acts on each $\Q_p^c$-valued
character $\eta$ of $G$ and thereby on $\Irr(G)$ via ${}^{\sigma}\eta(g) = \sigma(\eta(g))$
for all $\sigma \in G_{\Q_p}$ and $g \in G$.
We fix a $\Q_p^c[G]$-module $V_{\eta}$ with character $\eta$.
Choosing a $\Q_p^c$-basis of $V_{\eta}$ yields a matrix representation
\[
 \pi_{\eta}: G \longrightarrow \GL_{\eta(1)} (\Q_p^c)
\]
with character $\eta$. We define a linear character
\begin{eqnarray*}
 \det_{\eta}: G & \longrightarrow & (\Q_p^c)^{\times}\\
 g & \mapsto & \det_{\Q_p^c}(\pi_{\eta}(g)) = \det_{\Q_p^c}(g \mid V_{\eta}).
\end{eqnarray*}
The Wedderburn decomposition of $\Q_p^c[G]$ is given by
\[
 \Q_p^c[G] = \bigoplus_{\eta \in \Irr(G)} \Q_p^c[G] e(\eta) \simeq \bigoplus_{\eta \in \Irr(G)} M_{\eta(1) \times \eta(1)}(\Q_p^c),
\]
where $e(\eta) := \eta(1) / |G| \sum_{g \in G} \eta(g^{-1}) g$ are primitive central 
idempotents and the isomorphism on the right
maps each $g \in G$ to the tuple $(\pi_{\eta}(g))_{\eta \in \Irr(G)}$.
In particular, we have an isomorphism
\[
 \zeta(\Q_p^c[G]) \simeq \bigoplus_{\eta \in \Irr(G)} \Q_p^c.
\]
The reduced norm of $x \in \Q_p^c[G]$ is then given by
$\Nrd_{\Q_p^c[G]}(x) = (\det_{\Q_p^c}(x \mid V_{\eta}))_{\eta \in \Irr(G)}$.
For every $g \in G$ we have in particular
\begin{equation} \label{eqn:Nrd(g)}
	\Nrd_{\Q_p^c[G]}(g) = (\det_{\eta}(g))_{\eta \in \Irr(G)}.
\end{equation}

By a well-known theorem of Swan (see \cite[Theorem (32.1)]{MR632548}) the 
map $K_{0}(\Z_{p}[G]) \rightarrow K_{0}(\Q_{p}^c[G])$ 
induced by extension of scalars is injective. 
Thus from \eqref{eqn:long-exact-seq} we obtain an exact sequence
\begin{equation}\label{eqn:group-ring-K-exact-seq}
K_{1}(\Z_{p}[G]) \longrightarrow K_{1}(\Q_{p}^c[G]) 
\stackrel{\partial_p}{\longrightarrow} K_{0}(\Z_{p}[G],\Q_{p}^c)
\longrightarrow 0,
\end{equation}
where we write $\partial_p$ for $\partial_{\Z_p[G], \Q_p^c}$.
If $H$ is a subgroup
of $G$, then there exist natural restriction maps $\res^G_H$ for all
$K$-groups in \eqref{eqn:group-ring-K-exact-seq}. If $H$ is a normal subgroup
of $G$, then there likewise exist natural quotient maps $\quot^G_{G/H}$ for all
$K$-groups in \eqref{eqn:group-ring-K-exact-seq}.
Moreover, the reduced norm map induces an isomorphism 
\begin{equation}
 \Nrd_{\Q_p^c[G]}: K_{1}(\Q_{p}^c[G]) \stackrel{\simeq}{\longrightarrow} \zeta(\Q_{p}^c[G])^{\times}
\end{equation}
by \cite[Theorem (45.3)]{MR892316}, and one has an equality
\begin{equation} \label{eqn:Nrd-padic-grouprings}
\Nrd_{\Q_p[G]}(K_{1}(\Z_{p}[G]))=\Nrd_{\Q_p[G]}(\Z_{p}[G]^{\times})
\end{equation}
as follows from  \cite[Theorem (40.31)]{MR892316}. 

We need the following generalization of Taylor's
fixed point theorem \cite{MR608528}
(see \cite[Theorem 10A]{MR717033})
due to Izychev and Venjakob \cite[Theorem 2.21]{MR2905563}.

\begin{theorem} \label{thm:fixed-points}
	Let $E$ be a tame (possibly infinite) Galois extension of $\Q_p$.
	Let $H$ be an open subgroup of $\Gal(E/\Q_p)$ that contains the inertia
	subgroup, and put $F := E^H$. Then
	\[
		\left(\Nrd_{E[G]}(\mathcal{O}_E[G]^{\times})\right)^H = 
		\Nrd_{F[G]}\left(\mathcal{O}_F[G]^{\times}\right).
	\]
\end{theorem}

Let $S$ be a ring extension of a ring $R$.
We denote the kernel of the natural map $K_1(R[G]) \rightarrow
K_1(S[G])$ by $SK_1(R[G], S)$. If $R$ is a domain
with field of fractions $K$, we put
$SK_1(R[G]) := SK_1(R[G], K)$.

\begin{lemma} \label{lem:SK1}
	Let $G$ be a finite group and let $F$ be a finite extension of $\Q_p$ with 
	ring of integers $\mathcal{O}$. Then the following holds.
	\begin{enumerate}
		\item 
		There is an exact sequence of abelian groups
		\[
		0 \longrightarrow SK_1(\Z_p[G], \mathcal{O}) \longrightarrow
		SK_1(\Z_p[G]) \longrightarrow SK_1(\mathcal{O}[G])
		\longrightarrow 
		\]
		\[
		\longrightarrow K_0(\Z_p[G], \mathcal{O}[G]) \longrightarrow K_0(\Z_p[G], \Q_p^c)
		\longrightarrow K_0(\mathcal{O}[G], \Q_p^c) \longrightarrow 0.
		\]
		\item
		$SK_1(\Z_p[G], \mathcal{O})$ is a finite $p$-group.
		\item
		If in addition the degree $[F_0:\Q_p]$ of the
		maximal unramified subfield $F_0$ in $F$ is prime to $p$, then
		$SK_1(\Z_p[G], \mathcal{O})$ vanishes.
	\end{enumerate}	 
\end{lemma}

\begin{proof}
	It follows from \cite[Theorem (45.3)]{MR892316} that the natural map
	$K_1(\Q_p[G]) \rightarrow K_1(\Q_p^c[G])$ is injective. Therefore
	$SK_1(\Z_p[G])$ identifies with the kernel of $K_1(\Z_p[G])
	\rightarrow K_1(\Q_p^c[G])$. A similar observation holds for
	$SK_1(\mathcal{O}[G])$. As $K_0(\Z_p[G]) \rightarrow K_0(\Q_p^c[G])$
	is injective by Swan's theorem, a forteriori the map
	$K_0(\Z_p[G]) \rightarrow K_0(\mathcal{O}[G])$ has to be injective.
	Considering the long exact sequences of relative $K$-theory
	\eqref{eqn:long-exact-seq} for the three occurring pairs,
	a diagram chase shows that we have (i).	
	Then (ii) follows as $SK_1(\Z_p[G])$ is a finite $p$-group by
	\cite[Theorem (46.9)]{MR892316}. The last claim is a consequence of
	\cite[Theorem 2.25]{MR2905563} which actually says that the third
	arrow in (i) is an isomorphism in this case.
\end{proof}

\subsection{Iwasawa algebras of one-dimensional $p$-adic Lie groups} \label{subsec:Iwasawa-algebras}

We assume for the rest of this section that $p$ is an odd prime.
Let $\mathcal{G}$ be a profinite group.
The complete group algebra of $\mathcal{G}$ over $\Z_p$ is 
\[
\Lambda(\mathcal{G}) := \Z_{p}\llbracket \mathcal{G} \rrbracket = \varprojlim \Z_{p}[\mathcal{G}/\mathcal{N}],
\]
where the inverse limit is taken over all open normal subgroups $\mathcal{N}$ of $\mathcal{G}$. Then $\Lambda(G)$ is a compact $\Z_p$-algebra and we denote
the kernel of the natural augmentation map
$\Lambda(\mathcal{G}) \twoheadrightarrow \Z_p$ by $\Delta(\mathcal{G})$.
If $M$ is a (left) $\Lambda(\mathcal{G})$-module we let
$M_{\mathcal{G}} := M / \Delta(\mathcal{G}) M$ be the module of coinvariants
of $M$. This is the maximal quotient module of $M$ with trivial
$\mathcal{G}$-action. 
Similarly, we denote the maximal submodule of $M$ upon which
$\mathcal{G}$ acts trivially by $M^{\mathcal{G}}$.

Now suppose that $\mathcal{G}$
contains a finite normal subgroup $H$ 
such that $\overline{\Gamma} := \mathcal{G}/H \simeq \Z_p$.
Then $\mathcal{G}$ may be written as a semi-direct product $\mathcal{G} = H \rtimes \Gamma$
where $\Gamma \leq \mathcal{G}$ and 
$\Gamma \simeq \overline{\Gamma} \simeq \Z_{p}$. 
In other words, $\mathcal{G}$ is  a one-dimensional $p$-adic Lie group.

If $F$ is a finite field extension of $\Q_{p}$  with ring of integers 
$\mathcal{O}=\mathcal{O}_{F}$, we put 
$\Lambda^{\mathcal{O}}(\mathcal{G}) := \mathcal{O} \otimes_{\Z_{p}} \Lambda(\mathcal{G}) = \mathcal{O}\llbracket \mathcal{G} \rrbracket$.
We fix a topological generator $\gamma$ of $\Gamma$ and put 
$\overline{\gamma} := \gamma \mod H$ which is a topological generator of $\overline{\Gamma}$.
Since any homomorphism $\Gamma \rightarrow \Aut(H)$ must have open kernel, 
we may choose a natural number $n$ such that $\gamma^{p^n}$ is central in $\mathcal{G}$; 
we fix such an $n$.
As $\Gamma_{0} := \Gamma^{p^n} \simeq \Z_{p}$, there is a ring isomorphism
$R:=\mathcal{O}\llbracket \Gamma_{0} \rrbracket \simeq \mathcal{O}\llbracket T \rrbracket$ induced by $\gamma^{p^n} \mapsto 1+T$
where $\mathcal{O}\llbracket T \rrbracket$ denotes the power series ring in one variable over $\mathcal{O}$.
If we view $\Lambda^{\mathcal{O}}(\mathcal{G})$ as an $R$-module 
(or indeed as a left $R[H]$-module), there is a decomposition
\begin{equation*} 
\Lambda^{\mathcal{O}}(\mathcal{G}) = \bigoplus_{i=0}^{p^n-1} R[H] \gamma^{i}.
\end{equation*}
Hence $\Lambda^{\mathcal{O}}(\mathcal{G})$ is finitely generated as an $R$-module
and is an $R$-order in the separable $E:=Quot(R)$-algebra
$\mathcal{Q}^{F} (\mathcal{G})$, the total ring of fractions of 
$\Lambda^{\mathcal{O}}(\mathcal{G})$, obtained
from $\Lambda^{\mathcal{O}}(\mathcal{G})$ by adjoining inverses of all central regular elements.
Note that $\mathcal{Q}^{F} (\mathcal{G}) =  E \otimes_{R} \Lambda^{\mathcal{O}}(\mathcal{G})$ 
and that by \cite[Lemma 1]{MR2114937} we have 
$\mathcal{Q}^{F} (\mathcal{G}) = F \otimes_{\Q_{p}} \mathcal{Q}(\mathcal{G})$,
where $\mathcal{Q}(\mathcal{G}) := \mathcal{Q}^{\Q_{p}}(\mathcal{G})$.

For any field $F$ of characteristic $0$ let $\Irr_{F}(\mathcal{G})$ 
be the set of 
$F$-irreducible characters of $\mathcal{G}$ with open kernel.
Fix a character $\chi \in \Irr_{\Q_{p}^{c}}(\mathcal{G})$ and let 
$\eta$ be an irreducible constituent of $\res^{\mathcal{G}}_{H} \chi$.
Then $\mathcal{G}$ acts on $\eta$ as $\eta^{g}(h) = \eta(g^{-1}hg)$
for $g \in \mathcal{G}$, $h \in H$, and following \cite[\S 2]{MR2114937} we set
\[
  St(\eta) := \{g \in \mathcal{G}: \eta^g = \eta \}, \quad 
  e_{\chi} := \sum_{\eta \mid \res^{\mathcal{G}}_{H} \chi} e(\eta).
\]
By \cite[Corollary to Proposition 6]{MR2114937} $e_{\chi}$ is a primitive central idempotent of
$\mathcal{Q}^{c}(\mathcal{G}) := \Q_{p}^{c} \otimes_{\Q_{p}} \mathcal{Q}(\mathcal{G})$.
In fact, every primitive central idempotent of $\mathcal{Q}^{c}(\mathcal{G})$ is of this form
and $e_{\chi} = e_{\chi'}$ if and only if $\chi = \chi' \otimes \rho$ 
for some character $\rho$ of $\mathcal{G}$ of type $W$
(i.e.~$\res^{\mathcal{G}}_{H} \rho = 1$).
The irreducible constituents of $\res^{\mathcal{G}}_{H} \chi$ are precisely 
the conjugates of $\eta$
under the action of $\mathcal{G}$, each occurring with the same 
multiplicity $z_{\chi}$ by \cite[Proposition 11.4]{MR632548}. 
By \cite[Lemma 4]{MR2114937} we have $z_{\chi}=1$ and thus we also have equalities
\begin{equation*} 
\res^{\mathcal{G}}_{H} \chi = \sum_{i=0}^{w_{\chi}-1} \eta^{\gamma^{i}},
\quad
e_{\chi} = \sum_{i=0}^{w_{\chi}-1} e(\eta^{\gamma^{i}}) = \frac{\chi(1)}{|H|w_{\chi}}\sum_{h \in H} \chi(h^{-1})h,
\end{equation*}
where $w_{\chi} := [\mathcal{G} : St(\eta)]$.
Note that  $\chi(1) = w_{\chi} \eta(1)$ and that
$w_{\chi}$ is a power of $p$ since $H$ is a subgroup of $St(\eta)$.

Let $V_{\chi}$ denote a realisation of $\chi$ over $\Q_p^c$.	
By \cite[Proposition 5]{MR2114937}, there exists a unique element
$\gamma_{\chi} \in \zeta(\mathcal{Q}^{c}(\mathcal{G})e_{\chi})$ such that $\gamma_{\chi}$
acts trivially on $V_{\chi}$ and $\gamma_{\chi}= gc$ where $g \in \mathcal{G}$ 
with $(g \bmod H) = \overline{\gamma}^{w_{\chi}}$
and $c \in (\Q_{p}^{c}[H]e_{\chi})^{\times}$. Moreover, $\gamma_{\chi}= gc=cg$.
By \cite[Proposition 5]{MR2114937}, the element $\gamma_{\chi}$
generates a procyclic $p$-subgroup $\Gamma_{\chi}$ 
of $(\mathcal{Q}^{c}(\mathcal{G})e_{\chi})^{\times}$ and induces an isomorphism
\begin{equation} \label{eqn:centre-gamma_chi}
  \zeta(\mathcal{Q}^c(\mathcal{G}) e_{\chi}) \simeq \mathcal{Q}^c(\Gamma_{\chi}).
\end{equation}

\subsection{$K$-theory of Iwasawa algebras} \label{subsec:K-of-Iwasawa-algebras}
We now specialze sequence \eqref{eqn:long-exact-seq} to the present situation.
If $F / \Q_p$ is a finite field extension, then
by \cite[Corollary 3.8]{MR3034286} we have an exact sequence
\begin{equation} \label{eqn:connecting-surj-F} 
 K_1(\Lambda(\mathcal{G})) \longrightarrow K_1(\mathcal{Q}^F(\mathcal{G}))
 \longrightarrow K_0(\Lambda(\mathcal{G}), \mathcal{Q}^F(\mathcal{G}))
 \longrightarrow 0
\end{equation}
and likewise an exact sequence
\begin{equation} \label{eqn:connecting-surj-c}
 K_1(\Lambda(\mathcal{G})) \longrightarrow K_1(\mathcal{Q}^c(\mathcal{G}))
 \longrightarrow K_0(\Lambda(\mathcal{G}), \mathcal{Q}^c(\mathcal{G}))
 \longrightarrow 0.
\end{equation}
As any $x \in K_1(\Q_p^c \otimes_{\Z_p} \Lambda(\mathcal{G}))$ 
actually lies in the image
of $K_1(F \otimes_{\Z_p} \Lambda(\mathcal{G}))$ 
for sufficiently large $F$, we deduce
from \cite[Theorem 3.7]{MR3034286} that the natural map
$K_1(\Q_p^c \otimes_{\Z_p} \Lambda(\mathcal{G})) \rightarrow
K_1(\mathcal{Q}^c(\mathcal{G}))$ is injective. An easy diagram chase 
now shows the following.

\begin{lemma} 
The natural map
\[
K_0(\Lambda(\mathcal{G}), \Q_p^c \otimes_{\Z_p} \Lambda(\mathcal{G}))
\rightarrow K_0(\Lambda(\mathcal{G}), \mathcal{Q}^c(\mathcal{G}))
\]
induced by extension of scalars is injective.
\end{lemma}

Following \cite[Proposition 6]{MR2114937}, we define a map
\[
  j_{\chi}: \zeta(\mathcal{Q}^{c} (\mathcal{G})) \twoheadrightarrow 
  \zeta(\mathcal{Q}^{c} (\mathcal{G})e_{\chi}) \simeq \mathcal{Q}^{c}(\Gamma_{\chi}) 
  \longrightarrow  \mathcal{Q}^{c}(\overline{\Gamma}),
\]
where the isomorphism is \eqref{eqn:centre-gamma_chi} and 
the last arrow is induced by mapping $\gamma_{\chi}$ to $\overline{\gamma}^{w_{\chi}}$.
It follows from op.\ cit.\ that $j_{\chi}$ is independent of the choice of 
$\overline{\gamma}$ and that 
for every matrix $\Theta \in M_{n \times n} (\mathcal{Q}^c(\mathcal{G}))$ we have
\begin{equation*} \label{eqn:jchi-det}
j_{\chi} (\Nrd_{\mathcal{Q}^c(\mathcal{G})}(\Theta)) = 
\det_{\mathcal{Q}^{c}(\overline{\Gamma})} (\Theta \mid \Hom_{\Q_p^c[H]}(V_{\chi},  \mathcal{Q}^{c}(\mathcal{G})^n)).
\end{equation*}
Here, $\Theta$ acts on $f \in \Hom_{\Q_p^c[H]}(V_{\chi},  \mathcal{Q}^{c}(\mathcal{G})^{n})$ 
via right multiplication,
and $\overline{\gamma}$ acts on the left via 
$(\overline{\gamma} f)(v) = \gamma \cdot f(\gamma^{-1} v)$ for all $v \in V_{\chi}$,
where we recall that $\gamma$ is the unique lift of $\overline{\gamma}$ to 
$\Gamma \leq \mathcal{G}$.
Hence the map
\begin{eqnarray*}
\Det(~)(\chi): K_{1}(\mathcal{Q}^c(\mathcal{G})) & \longrightarrow & 
  \mathcal{Q}^{c}(\overline{\Gamma})^{\times} \\
 {[P,\alpha]}& \mapsto & \mathrm{det}_{\mathcal{Q}^{c}(\overline{\Gamma})} (\alpha \mid \Hom_{\Q_p^c[H]}(V_{\chi}, P)),
\end{eqnarray*}
where $P$ is a projective $\mathcal{Q}^c(\mathcal{G})$-module and $\alpha$ a 
$\mathcal{Q}^c(\mathcal{G})$-automorphism of $P$, 
is just $j_{\chi} \circ \Nrd_{\mathcal{Q}^c(\mathcal{G})}$ (see \cite[\S 3, p.558]{MR2114937}).
If $\rho$ is a character of $\mathcal{G}$ of type $W$ (i.e.~$\res^{\mathcal{G}}_H \rho = 1$)
then we denote by
$\rho^{\sharp}$ the automorphism of the field $\mathcal{Q}^{c}(\overline{\Gamma})$ induced by
$\rho^{\sharp}(\overline{\gamma}) = \rho(\overline{\gamma}) \overline{\gamma}$. 
Moreover, we denote the additive group generated by all $\Q_{p}^{c}$-valued
characters of $\mathcal{G}$ with open kernel by $R_p(\mathcal{G})$. We let
$\Hom^{W}(R_{p}( \mathcal{G}), \mathcal{Q}^{c}(\overline{\Gamma})^{\times})$
be the group of all homomorphisms 
$f: R_p(\mathcal{G}) \rightarrow \mathcal{Q}^{c}(\overline{\Gamma})^{\times}$ satisfying
$f(\chi \otimes \rho) = \rho^{\sharp}(f(\chi))$ for all characters $\rho$ 
of type $W$. Finally, $\Hom^{\ast}(R_{p}( \mathcal{G}), \mathcal{Q}^{c}(\overline{\Gamma})^{\times})$
is the subgroup of $\Hom^{W}(R_{p}( \mathcal{G}), \mathcal{Q}^{c}(\overline{\Gamma})^{\times})$
of all homomorphisms $f$ that in addition satisfy
$f({}^{\sigma}\chi) = \sigma(f(\chi))$ for all Galois automorphisms 
$\sigma \in G_{\Q_{p}}$. If $A$ is a subring of
$\mathcal{Q}^{c}(\overline{\Gamma})$, we put
\[
\Hom^{W}(R_{p}( \mathcal{G}), A^{\times}) :=
\Hom(R_{p}( \mathcal{G}), A^{\times}) \cap
\Hom^{W}(R_{p}( \mathcal{G}), \mathcal{Q}^{c}(\overline{\Gamma})^{\times})
\]
and similarly with $\Hom^{\ast}$.

By \cite[Proof of Theorem 8]{MR2114937} we have a $G_{\Q_p}$-equivariant isomorphism
\begin{eqnarray*}
\zeta(\mathcal{Q}^c(\mathcal{G}))^{\times} & \simeq & 
\Hom^{W}(R_{p}(\mathcal{G}), \mathcal{Q}^{c}(\overline{\Gamma})^{\times})\\
x & \mapsto & [\chi \mapsto j_{\chi}(x)].
\end{eqnarray*}
By \cite[Theorem 8]{MR2114937} the map $\Theta \mapsto [\chi \mapsto \Det(\Theta)(\chi)]$
defines a homomorphism
\[
\Det: K_{1}(\mathcal{Q}^c(\mathcal{G})) \rightarrow \Hom^{W}(R_p(\mathcal{G}), \mathcal{Q}^{c}(\overline{\Gamma})^{\times})
\]
such that $\Det$ maps $K_{1}(\mathcal{Q}(\mathcal{G}))$ into
$\Hom^{\ast}(R_{p}( \mathcal{G}), \mathcal{Q}^{c}(\overline{\Gamma})^{\times})$,
and such that we obtain commutative triangles
\begin{equation} \label{eqn:Det_triangle-c}
\xymatrix{
& K_{1}(\mathcal{Q}^c(\mathcal{G})) \ar[dl]_{\Nrd_{\mathcal{Q}^c(\mathcal{G})}} \ar[dr]^{\Det} &\\
{\zeta(\mathcal{Q}^c(\mathcal{G}))^{\times}} \ar[rr]^-{\simeq} & & {\Hom^{W}(R_p( \mathcal{G}), \mathcal{Q}^{c}(\overline{\Gamma})^{\times})}}
\end{equation}
and
\begin{equation} \label{eqn:Det_triangle}
\xymatrix{
& K_{1}(\mathcal{Q}(\mathcal{G})) \ar[dl]_{\Nrd_{\mathcal{Q}(\mathcal{G})}} \ar[dr]^{\Det} &\\
{\zeta(\mathcal{Q}(\mathcal{G}))^{\times}} \ar[rr]^-{\simeq} & & {\Hom^{\ast}(R_p( \mathcal{G}), \mathcal{Q}^{c}(\overline{\Gamma})^{\times})}.}
\end{equation}

Let $\Z_p^c$ be the integral closure of $\Z_p$ in $\Q_p^c$ an put
$\Lambda^{c}(\overline{\Gamma}) := \Z_p^c \otimes_{\Z_p} \Lambda(\overline{\Gamma})$.
By \cite[Lemma 2]{MR2242618} the map $\Det$ restricts to a homomorphism
\begin{equation} \label{eqn:Det-of-Iwasawa}
  \Det: K_1(\Lambda(\mathcal{G})) \longrightarrow      
  \Hom^{\ast}(R_p( \mathcal{G}), \Lambda^{c}(\overline{\Gamma})^{\times}).
\end{equation}   
Let $\aug_{\overline{\Gamma}}: \Q_p^c \otimes_{\Z_p} 
\Lambda(\overline{\Gamma}) \twoheadrightarrow \Q_p^c$
be the natural augmentation map.
The following result will be useful when we like to check whether a given
homomorphism 
lies in the image of $K_1(\Lambda(\mathcal{G}))$ under $\Det$.

\begin{lemma} \label{lem:Det-criterion}
 Let $f,g \in \Hom^{W}(R_p( \mathcal{G}), (\Q_p^c \otimes_{\Z_p} \Lambda(\overline{\Gamma}))^{\times})$
 be two homomorphisms. Suppose that $\aug_{\overline{\Gamma}}(f(\chi)) = 
 \aug_{\overline{\Gamma}}(g(\chi))$ for all $\chi \in 
 \Irr_{\Q_p^c}(\mathcal{G})$.
 Then we have $f=g$. 
\end{lemma}

\begin{proof}
	Let $\chi \in \Irr_{\Q_p^c}(\mathcal{G})$ be a character.
	There is an isomorphism $\Lambda(\overline{\Gamma}) 
	\simeq \Z_p \llbracket T \rrbracket$, the ring of formal power
	series in one variable $T$ with coefficients in $\Z_p$, which
	maps $\overline{\gamma}$ to $1+T$.
	We identify $f(\chi)$ and $g(\chi)$ with the corresponding
	power series $f_{\chi}(T)$ and $g_{\chi}(T)$ in
	$\Q_p^c \otimes_{\Z_p} \Z_p\llbracket T \rrbracket$, respectively.
	We have to show that $h_{\chi}(T) := f_{\chi}(T) - g_{\chi}(T)$
	vanishes. The condition $\aug_{\overline{\Gamma}}(f(\chi)) = 
	\aug_{\overline{\Gamma}}(g(\chi))$ is equivalent to
	$h_{\chi}(0) = 0$. Now let $\rho$ be a character of type $W$.
	Then we have $\rho^{\sharp}(h_{\chi}(T)) = f_{\chi \otimes \rho}(T) - 
	g_{\chi \otimes \rho}(T) = h_{\chi \otimes \rho}(T)$, and so
	we obtain
	\[
		h_{\chi}(\rho(\overline{\gamma}) - 1) = h_{\chi \otimes \rho}(0) = 0.
	\]
	Now $h_{\chi}(T)$ vanishes by \cite[Corollary 7.4]{MR1421575}.
\end{proof}

\begin{example} \label{ex:Det(g)}
	For any $g \in \mathcal{G}$ we claim that
	the homomorphism $\Det(g)$ is given 
	on irreducible characters $\chi \in \Irr_{\Q_p^c}(\mathcal{G})$ by
	\begin{equation} \label{eqn:Det(g)}
		\chi \mapsto \det_{\chi}(g) \overline{g}^{\chi(1)},
	\end{equation}
	where $\overline{g} \in \overline{\Gamma}$ denotes the image of $g$
	under the canonical projection $\mathcal{G} \twoheadrightarrow
	\overline{\Gamma}$. We first note that
	$\det_{\chi \otimes \rho}(g) \overline{g}^{(\chi \otimes \rho)(1)} = \rho(g)^{\chi(1)} \det_ {\chi}(g) 
	\overline{g}^{\chi(1)} 
	= \rho^{\sharp}(\det_{\chi}(g) \overline{g}^{\chi(1)})$, and
	so \eqref{eqn:Det(g)} defines an element in
	$\Hom^{W}(R_p( \mathcal{G}), (\Q_p^c \otimes_{\Z_p} \Lambda(\overline{\Gamma}))^{\times})$. 
	The middle displayed formula on p.~2774 of 
	\cite[proof of Theorem 6.4]{MR2609173} shows that
	$\Det(g)(\chi)$ belongs to $\Q_p^c \otimes_{\Z_p} \Lambda(\overline{\Gamma})$. As the same is true for
	$\Det(g^{-1})(\chi)$, we see that 
	$\Det(g) \in \Hom^{W}(R_p( \mathcal{G}), (\Q_p^c \otimes_{\Z_p} \Lambda(\overline{\Gamma}))^{\times})$. 
	By Lemma \ref{lem:Det-criterion} it now suffices to show that
	$\aug_{\overline{\Gamma}}(\Det(g)(\chi)) = \det_{\chi}(g)$.
	We have $\Det(g)(\chi) = j_{\chi}(\Nrd_{\mathcal{Q}(\mathcal{G})}(g))$
	by triangle \eqref{eqn:Det_triangle}. 
	Choose a normal subgroup $\Gamma' \simeq \Z_p$ of $\mathcal{G}$
	which lies in the kernel of $\chi$. Put $G' := \mathcal{G} / 
	\Gamma'$ and view $\chi$ as a character of $G'$.
	Now \cite[(8)]{MR2609173} implies that
	$\aug_{\overline{\Gamma}}(j_{\chi}(\Nrd_{\mathcal{Q}(\mathcal{G})}(g)))$ equals
	the $\chi$-component of $\Nrd_{\Q_p[G']}(g')$,
	where $g'  := g \mod \Gamma'$. However, this $\chi$-component
	is $\det_{\chi}(g') = \det_{\chi}(g)$ by \eqref{eqn:Nrd(g)} as desired.
\end{example}

\begin{remark}\label{rem:Det-square}
	As we have observed in Example \ref{ex:Det(g)}, the proof of
	\cite[Theorem 6.4]{MR2609173} and in particular \cite[(8)]{MR2609173}
	show that we have a commutative square
	\begin{equation*} 
	\xymatrix{
		K_1(\Q_p^c \otimes_{\Z_p} \Lambda(\mathcal{G})) \ar[r] \ar[d]_{\Det} 
		& \varprojlim K_1(\Q_p^c[\mathcal{G}/ \mathcal{N}]) 
		\ar[d]_{\simeq}^{\left(\Nrd_{\Q_p^c[\mathcal{G}/ \mathcal{N}]}\right)_{\mathcal{N}}} \\
		\Hom^{W}(R_p( \mathcal{G}), (\Q_p^c \otimes_{\Z_p} \Lambda(\overline{\Gamma}))^{\times}) \ar[r] & 
		\varprojlim \zeta(\Q_p^c[\mathcal{G}/ \mathcal{N}])^{\times}
		\simeq \prod_{\chi \in \Irr_{\Q_p^c}(\mathcal{G})} (\Q_p^c)^{\times}
	}
	\end{equation*}
	where the inverse limits are taken over all open normal subgroups 
	$\mathcal{N}$ of $\mathcal{G}$.
	Now Lemma \ref{lem:Det-criterion} implies that the bottom map
	(which is given by $f \mapsto (\aug_{\overline \Gamma}(f(\chi)))_{\chi}$) is injective.
\end{remark}

\begin{prop} \label{prop:limit-relative-K0}
	Let $F$ be a finite extension of $\Q_p$ with 
	ring of integers $\mathcal{O}$. Then we have a canonical isomorphism
	\[
		K_0(\Lambda(\mathcal{G}), \Lambda^{\mathcal{O}}(\mathcal{G})) \simeq
		\varprojlim K_0(\Z_p[\mathcal{G} / \mathcal{N}], \mathcal{O}[\mathcal{G} / \mathcal{N}]),
	\]
	where the inverse limit runs over all open normal subgroup $\mathcal{N}$
	of $\mathcal{G}$.
\end{prop}

\begin{proof}
	Choose $n_0 \in \N$ such that $\Gamma^{p^{n_0}}$ is central in 
	$\mathcal{G}$. For $n \geq n_0$ we put $G_n := \mathcal{G} / \Gamma^{p^n}$.
	As the poset of subgroups $\Gamma^{p^n}$, $n \geq n_0$ is cofinal
	in the poset of all open normal subgroups of $\mathcal{G}$, we have to
	show that
	\[
	K_0(\Lambda(\mathcal{G}), \Lambda^{\mathcal{O}}(\mathcal{G})) \simeq
	\varprojlim_n K_0(\Z_p[G_n], \mathcal{O}[G_n]).
	\]
	As the kernel of the natural projection 
	$\Z_p[G_{n+1}] \twoheadrightarrow \Z_p[G_n]$ 
	is contained in the radical of $\Z_p[G_{n+1}]$ 
	by \cite[Proposition 5.26]{MR632548},
	we have a surjection $\Z_p[G_{n+1}]^{\times} \twoheadrightarrow
	\Z_p[G_{n}]^{\times}$ by \cite[Exercise 5.2]{MR632548}.
	We then likewise have $K_1(\Z_p[G_{n+1}]) \twoheadrightarrow
	K_1(\Z_p[G_{n}])$ by \cite[Theorem (40.31)]{MR892316}. Therefore
	the inverse system $K_1(\Z_p[G_{n}])$, $n \geq n_0$ satisfies the
	Mittag--Leffler condition. As $SK_1(\Z_p[G_n], \mathcal{O})$
	is finite for all $n \geq n_0$ by Lemma \ref{lem:SK1}, 
	taking inverse limits over the exact
	sequences of abelian groups
	\[
	0 \rightarrow SK_1(\Z_p[G_n], \mathcal{O}) \rightarrow
	K_1(\Z_p[G_n]) \rightarrow K_1(\mathcal{O}[G_n]) \rightarrow
	K_0(\Z_p[G_n], \mathcal{O}[G_n]) \rightarrow 0
	\]
	is exact. By \cite[Proposition 1.5.1]{MR2276851}
	(this requires $\mathcal{O}$ having a finite residue field) we have
	canonical isomorphisms
	\[
		K_1(\Lambda(\mathcal{G})) \simeq \varprojlim_n K_1(\Z_p[G_n]),
		\quad
		K_1(\Lambda^{\mathcal{O}}(\mathcal{G})) \simeq \varprojlim_n K_1(\mathcal{O}[G_n]).
	\]
	We thus obtain an exact sequence
	\[
	0 \rightarrow \varprojlim_n SK_1(\Z_p[G_n], \mathcal{O}) \rightarrow
	K_1(\Lambda(\mathcal{G})) \rightarrow 
	K_1(\Lambda^{\mathcal{O}}(\mathcal{G})) \rightarrow
	\varprojlim_n K_0(\Z_p[G_n], \mathcal{O}[G_n]) \rightarrow 0
	\]
	as desired.
\end{proof}

\begin{lemma} \label{lem:Nrd-SK1}
	Let $F$ be a finite extension of $\Q_p$ with 
	ring of integers $\mathcal{O}$. Define
	\[
	SK_1(\Lambda^{\mathcal{O}}(\mathcal{G})) :=
	\varprojlim SK_1(\mathcal{O}[\mathcal{G} / \mathcal{N}]),
	\]
	where the inverse limit is taken over all open normal subgroup $\mathcal{N}$
	of $\mathcal{G}$. Then we have an exact sequence
	\[
		0 \longrightarrow SK_1(\Lambda^{\mathcal{O}}(\mathcal{G})) \longrightarrow
		K_1(\Lambda^{\mathcal{O}}(\mathcal{G}))
		\xrightarrow{\Nrd}
		\zeta(\mathcal{Q}^F(\mathcal{G}))^{\times}.
	\]
\end{lemma}

\begin{proof}
	Let $x = (x_n)_n \in
	K_1(\Lambda^{\mathcal{O}}(\mathcal{G})) =
	\varprojlim_n K_1(\mathcal{O}[G_n])$. Then we have
	$\Nrd_{\mathcal{Q}^F(\mathcal{G})}(x) = 1$ if and only if
	the homomorphism
	\[
		\left[\chi \mapsto j_{\chi} (\Nrd_{\mathcal{Q}^F(\mathcal{G})}(x))\right]
		\in \Hom^W(R_p(\mathcal{G}), (\Q_p^c \otimes_{\Z_p} \Lambda(\overline{\Gamma}))^{\times})
	\]
	is trivial. Let $\chi \in \Irr_{\Q_p^c}(\mathcal{G})$ be 
	a character. Choose $n \in \N$ such that $\chi$ factors through
	$G_n$. As in Example \ref{ex:Det(g)} we have that
	$\aug_{\overline \Gamma} (j_{\chi} (\Nrd_{\mathcal{Q}^F(\mathcal{G})}(x)))$
	agrees with the $\chi$-component of $\Nrd_{F[G_n]}(x_n)$.
	Now Lemma \ref{lem:Det-criterion} implies the claim.
\end{proof}

\begin{remark} \label{rmk:Nrd-injective}
	As noted in \cite[Remark E]{MR2114937}, a conjecture of Suslin implies
	that the reduced norm
	$\Nrd_{\mathcal{Q}^F(\mathcal{G})}: K_1(\mathcal{Q}^F(\mathcal{G}))
	\rightarrow \zeta(\mathcal{Q}^F(\mathcal{G}))^{\times}$ is injective.
	This is true if $\mathcal{G}$ is abelian
	or, more generally, if
	$p$ does not divide the order of the commutator
	subgroup of $\mathcal{G}$ (this follows from 
	\cite[Proposition 4.5]{MR3092262} as explained in
	\cite[Remark 4.8]{hybrid-EIMC}). 
	Whenever this holds, Lemma \ref{lem:Nrd-SK1} shows that
	$SK_1(\Lambda^{\mathcal{O}}(\mathcal{G}))$ 
	identifies with the kernel of the 
	natural map
	$K_1(\Lambda^{\mathcal{O}}(\mathcal{G})) \rightarrow
	K_1(\mathcal{Q}^F(\mathcal{G}))$.
\end{remark}

\begin{remark}
	If $F = \Q_p$ and $\mathcal{G}$ is a pro-$p$-group, then
	$SK_1(\Lambda(\mathcal{G}))$ coincides with the kernel of the 
	natural map
	$K_1(\Lambda(\mathcal{G})) \rightarrow
	K_1(\Lambda^{\infty}(\mathcal{G}))$, where
	$\Lambda^{\infty}(\mathcal{G}) := 
	\varprojlim  \Q_p[\mathcal{G}/\mathcal{N}]$
	(see \cite[Corollary 3.2]{MR3093505}).
\end{remark}

\section{Galois Gauss sums} \label{sec:Gauss-sums}

\subsection{General Notation}
Fix a prime $p$.
For any $p$-adic field $K$ we denote its ring of integers by $\mathcal{O}_K$
and let $\pi_K \in \mathcal{O}_K$ be a uniformizer. Then $\mathfrak{p}_K := \pi_K \mathcal{O}_K$
is the unique maximal ideal in $\mathcal O_K$. 
We let $v_K: K^{\times} \rightarrow \Z$ be the associated normalized valuation,
i.e.\ $v_K(\pi_K) = 1$.
If $\mathfrak{a}$ is any ideal in $\mathcal{O}_K$, we let 
$N(\mathfrak{a}) = |\mathcal{O}_K / \mathfrak{a}|$ be its absolute norm.
In particular, $N(\mathfrak{p}_K)$ is the cardinality of the residue field of $K$.
We set $U_K^0 := \mathcal{O}_K^{\times}$ and $U_K^n := 1 + \mathfrak p_K^n$
for every positive integer $n$. We denote the absolute different of $K$ by
$\mathfrak{D}_K$ so that
\[
 \mathfrak{D}_K^{-1} = \left\{ x \in K \mid \Tr_{K/\Q_p}(x \mathcal{O}_K) \subseteq \Z_p \right\},
\]
where for any finite extension $K/F$ of local fields $\Tr_{K/F}: K \rightarrow F$
denotes the trace map. Similarly, we let $N_{K/F}: K^{\times} \rightarrow F^{\times}$
be the field theoretic norm map. We use the same notation for the norm
on ideals so that in particular $N_{K/F}(\mathfrak{p}_K) = \mathfrak{p}_F^{f_{K/F}}$,
where $f_{K/F}$ denotes the degree of the corresponding residue field extension.

Let $G_K^{\ab} := \Gal(K^{\ab}/K)$ be the Galois group over $K$ of the maximal abelian 
extension $K^{\ab}$ of $K$ and let
\[
  (-, K): K^{\times} \longrightarrow G_K^{\ab}
\]
be the local Artin map. Then we have commutative diagrams
\begin{equation} \label{eqn:LCT-norms}
 \xymatrix{
 K^{\times} \ar[rr]^{(-,K)} \ar[d]_{N_{K/F}} & & G_K^{\ab} \ar[d] 
 & & K^{\times} \ar[rr]^{(-,K)} & & G_K^{\ab}\\
 F^{\times} \ar[rr]^{(-,F)} & & G_F^{\ab} 
 & & F^{\times} \ar[rr]^{(-,F)} \ar[u] & & G_F^{\ab} \ar[u]_{\Ver_{K/F}}
 }
\end{equation}
where in the left diagram the vertical arrow on the right 
denotes the canonical map, and in the right diagram the vertical arrows
are the natural embedding and the transfer map
$\Ver_{K / F}: G_{F}^{\ab} \rightarrow G_K^{\ab}$.

\subsection{Abelian Galois Gauss sums}
Let $L/K$ be a finite Galois extension of $p$-adic fields with abelian Galois group $G$.
Then every $\chi \in \Irr_{\C}(G)$ may be viewed as a complex
character of $K^{\times}$ via the local Artin map
and the natural projection $G_K^{\ab} \twoheadrightarrow G$.
The conductor of $\chi$ is the ideal $\mathfrak{f}(\chi) = \mathfrak{p}_K^{m_{\chi}}$,
where $m_{\chi}$ is the smallest integer such that $\chi(U_K^{m_{\chi}}) = 1$.
Let $\psi_{p}$ be the composition of the following three maps:
\[
\psi_{p} : \Q_{p} \longrightarrow \Q_{p}/\Z_{p} \longrightarrow \Q/\Z \longrightarrow \C^{\times},
\]
where the first map is the canonical surjection, 
the second map is the canonical injection which maps $\Q_{p}/\Z_{p}$
onto the $p$-component of the divisible group $\Q/\Z$, 
and the third map is the exponential map $x \mapsto e^{2\pi ix}$.
Thus $\psi_{p}(\Z_{p})=1$ and for any $r \in \N$ we have $\psi_{p}(p^{-r})=\zeta_{p^{r}}$ 
where $\zeta_{p^{r}} = e^{\frac{2\pi ix}{p^{r}}}$
is a primitive $p^{r}$th root of unity. 
Define the standard additive character $\psi_{K}: K \longrightarrow \C^{\times}$ 
to be the composition 
$\psi_{p} \circ \mathrm{Tr}_{K/\Q_{p}}$.
Note that the codifferent $\mathfrak{D}_{K}^{-1}$ is the largest ideal of $K$
on which $\psi_{K}$ is trivial.

\begin{definition}\label{def:GGs-linear}
The local Galois Gauss sum $\tau_{K}(\chi)$ is defined to be the sum
\[
  \tau_{K}(\chi) = \sum_{u \in U_K^0 / U_K^{m_{\chi}}} 
  \chi(u c_{\chi}^{-1}) \psi_{K}(u c_{\chi}^{-1}) \in \Q^c,
\]
where $c_{\chi}$ is any generator of the ideal $\mathfrak{f}(\chi)\mathfrak{D}_{K}$ 
(the sum is easily shown to be independent of the choice of $c_{\chi}$).
\end{definition}

\begin{remark} \label{rem:special-GGs}
If $\chi$ is unramified (i.e. $m_{\chi} = 0$) then we may view
$\chi: K^{\times} / \mathcal{O}_K^{\times} \rightarrow \C^{\times}$ as a function on
the fractional ideals in $K$. Then the sum in Definition \ref{def:GGs-linear} 
reduces to one term and we have
$\tau_{K}(\chi)=\chi(\mathfrak{D}_{K}^{-1})$. 
If $\chi$ is ramified (i.e. $m_{\chi} > 0$) then the sum runs over all
$u \in U_K^0 / U_K^{m_{\chi}} = \mathcal{O}_K^{\times} / 1 + \mathfrak{f}(\chi)$.
\end{remark}

\begin{remark}
 One knows that $|\tau_K(\chi)| = \sqrt{N(\mathfrak{f}(\chi))}$
 (see \cite[Chapter II, Proposition 2.2]{MR0447187}, for instance).
 In particular, $\tau_K(\chi)$ is non-zero.
\end{remark}

The following result is well known (see the proof of \cite[Chapter III, Lemma 6.1]{MR717033},
for instance). We give a proof for convenience of the reader.

\begin{prop}\label{prop:Gauss-twists-abelian}
Let $\chi, \rho \in \Irr_{\C}(G)$ be two irreducible characters of $G$.
If $\rho$ is unramified then 
\[
  \tau_{K}(\chi \otimes \rho) = \rho((\mathfrak{f}(\chi)\mathfrak{D}_{K})^{-1}) \tau_{K}(\chi)
  = \rho(c_{\chi}^{-1}) \tau_K(\chi).
\] 
\end{prop}

\begin{proof}
As $\rho$ is unramified, we have $m_{\chi \otimes \rho} = m_{\chi}$
and thus $\mathfrak{f}(\chi \otimes \rho)=\mathfrak{f}(\chi)$.
Let $c_{\chi}$ be a generator of 
$\mathfrak{f}(\chi \otimes \rho)\mathfrak{D}_{K}= 
\mathfrak{f}(\chi)\mathfrak{D}_{K}$.
Hence
\begin{eqnarray*}
\tau_{K}(\chi \otimes \rho) 
&=& \sum_{u \in U_K^0/U_K^{m_{\chi}}} (\chi \otimes \rho)(u 
	c_{\chi}^{-1}) \psi_{K}(u c_{\chi}^{-1}) \\
&=& \rho(c_{\chi}^{-1}) \sum_{u \in U_K^0/U_K^{m_{\chi}}} 
	\chi(u c_{\chi}^{-1}) \psi_{K}(u c_{\chi}^{-1}) \\
&=& \rho(c_{\chi}^{-1}) \tau_{K}(\chi)\\
&=& \rho((\mathfrak{f}(\chi)\mathfrak{D}_{K})^{-1}) \tau_{K}(\chi),
\end{eqnarray*}
where the second equality uses the fact that $\rho(uc_{\chi}^{-1})=\rho(c_{\chi}^{-1})$ 
for all $u \in \mathcal{O}_{K}^{\times}$.
\end{proof}

\subsection{General Galois Gauss sums}
Now let $L/K$ be an arbitrary finite Galois extension of $p$-adic fields
with Galois group $G$. We write $R(G)$ for the ring of virtual $\C$-valued characters
of $G$. 
If $\chi$ is a character of $G$, then $\deg(\chi) := \chi(1)$ is called the degree of $\chi$.
This uniquely extends to a homomorphism
\[
 \deg: R(G) \longrightarrow \Z.
\]
Let $A$ be an abelian group.
A family of homomorphisms
$f_{L^H}:R(H) \rightarrow A$, where $H$ runs through all subgroups of $G$,
is called inductive in degree $0$ if
$f_K(\ind_H^G \chi) = f_{L^H}(\chi)$ for every subgroup $H$ of $G$ and every
$\chi \in R(H)$ of degree $0$. Such a family is uniquely determined
by its values on linear characters (see \cite[Chapter III, remark after Lemma 1.1]{MR717033}; the argument also appears in the proof
of Proposition \ref{prop:Gauss-twists} below, in particular see \eqref{eqn:Brauer-GGs}).

The following definition is in fact well defined (see \cite[Chapter II, \S 4]{MR0447187}).

\begin{definition}
 There is a unique family of homomorphisms
 \begin{eqnarray*}
  \tau_{L^H}: R(H) & \longrightarrow & (\Q^c)^{\times}\\
  \chi & \mapsto & \tau_{L^H}(\chi)
 \end{eqnarray*}
 such that $\tau_{L^H}(\chi)$ is the abelian Galois Gauss sum defined 
 in Definition \ref{def:GGs-linear}
 for every linear character $\chi$ of $H$,
 and such that the family is inductive in degree $0$.
 We call $\tau_K(\chi)$ the local Galois Gauss sum of $\chi$.
\end{definition}

If $\chi$ is a complex valued character of $G$, we let $\mathfrak{f}(\chi)$ be
the Artin conductor of $\chi$. As $\mathfrak{f}(\chi + \chi') = \mathfrak{f}(\chi)
\cdot \mathfrak{f}(\chi')$ for any two characters $\chi$ and $\chi'$, 
there is a unique way to define $\mathfrak{f}(\chi)$ for any virtual character $\chi \in R(G)$
such that $\mathfrak{f}$ is a homomorphism on $R(G)$ with values in the fractional
ideals of $K$.
We now prove the following generalization of Proposition \ref{prop:Gauss-twists-abelian} which might be also well known to experts.

\begin{prop} \label{prop:Gauss-twists}
 Let $\chi, \rho \in \Irr_{\C}(G)$ be two irreducible characters of $G$.
If $\rho$ is unramified (and thus linear) then 
\[
  \tau_{K}(\chi \otimes \rho) = \rho((\mathfrak{f}(\chi)\mathfrak{D}_{K}^ {\chi(1)})^{-1}) \tau_{K}(\chi)
  = \rho(c_{\chi}^{-1}) \tau_K(\chi),
\] 
where $c_{\chi}$ is any generator of the ideal $\mathfrak{f}(\chi) \mathfrak{D}_K^{\chi(1)}$.
\end{prop}

\begin{proof}
 For any group $G$ we write $\mathbb{1}_G$ for the trivial character.
 By a strengthened version of Brauer's induction theorem (see \cite[Exercise 10.6]{MR0450380})
 there are subgroup $U$ of $G$ and linear characters $\lambda_U$ of $U$ such that
 \begin{equation} \label{eqn:Brauer-chi}
  \chi - \chi(1) \mathbb 1_G = \sum_U z_U \, \ind_U^G(\lambda_U - \mathbb 1_U),
 \end{equation}
  where the $z_U$ are suitable integers. 
  As Galois Gauss sums are inductive in degree $0$ and 
  $\tau_K(\mathbb 1_G) = \tau_{L^U}(\mathbb 1_U) = 1$ for all
  $U$, equation \eqref{eqn:Brauer-chi} implies that
  \begin{equation} \label{eqn:Brauer-GGs}
   \tau_K(\chi) = \prod_U \tau_{L^U}(\lambda_U)^{z_U}.
  \end{equation}
  By \cite[Corollary 10.20]{MR632548} we likewise have
  \begin{equation} \label{eqn:Brauer-chirho}
   \chi \otimes \rho - \chi(1) \rho = \sum_U z_U \, \ind_U^G((\lambda_U \otimes \rho_U) - \rho_U ),
  \end{equation}
  where we put $\rho_U := \res^G_U \rho$. Note that $\rho_U = \rho \circ N_{L^U/K}$
  by local class field theory (use the left diagram \eqref{eqn:LCT-norms}).
  For the Artin conductor we compute
  \begin{eqnarray*}
   \mathfrak{f}(\chi) & = & \mathfrak{f}(\chi - \chi(1) \mathbb 1_G)\\
   & = & \prod_U \mathfrak{f}(\ind_U^G(\lambda_U - \mathbb 1_U))^{z_U}\\
   & = & \prod_U N_{L^U/K}(\mathfrak{f}(\lambda_U - \mathbb 1_U))^{z_U}\\
   & = & \prod_U N_{L^U/K}(\mathfrak{f}(\lambda_U))^{z_U}.
  \end{eqnarray*}
  Here, the first and the last equality follow from the fact that the conductor of
  the trivial character is trivial, whereas the other two equalities
  follow from the fundamental properties of the Artin conductor (see 
  \cite[Chapter II, \S 1]{MR0447187}). We thus obtain an equality
  \begin{equation} \label{eqn:rho-equality}
   \rho(\mathfrak{f}(\chi)) = \prod_U \rho_U(\mathfrak{f}(\lambda_U))^{z_U}.
  \end{equation}
  For the Galois Gauss sums we then have
  \begin{eqnarray*}
   \tau_K(\chi \otimes \rho)\tau_K(\rho)^{- \chi(1)} 
   & = & \prod_U \tau_{L^U}(\lambda_U \otimes \rho_U)^{z_U} \tau_{L^U}(\rho_U)^{-z_U}\\
   & = & \prod_U \rho_U(\mathfrak{f}(\lambda_U))^{-z_U} \tau_{L^U}(\lambda_U)^{z_U}\\
   & = & \rho(\mathfrak{f}(\chi))^{-1} \tau_K(\chi).
  \end{eqnarray*}
  As Galois Gauss sums are inductive in degree $0$, the first equality follows
  from \eqref{eqn:Brauer-chirho}. The second is Proposition \ref{prop:Gauss-twists-abelian}
  and the equality $\tau_{L^U}(\rho_U) = \rho_U(\mathfrak{D}_{L^U}^{-1})$
  (see Remark \ref{rem:special-GGs}), whereas the third is \eqref{eqn:Brauer-GGs}
  and \eqref{eqn:rho-equality}. As $\tau_{K}(\rho) = \rho(\mathfrak{D}_{K}^{-1})$
  we are done.
\end{proof}

If $K/F$ is a field extension and $\chi$ is a character of $G_K$, then we also
write $\ind_K^F(\chi)$ for the induced character $\ind_{G_K}^{G_F}(\chi)$.
By \cite[Remark 3, p.\ 109]{MR717033} one has
\[
 \tau_{\Q_p}(\ind_K^{\Q_p}(\chi)) = \tau_K(\chi) \tau_{\Q_p}(\ind_K^{\Q_p}(\mathbb 1))^{\chi(1)}.
\]
In fact, this is easily deduced from inductivity in degree $0$. Now
Proposition \ref{prop:Gauss-twists} obviously implies the following.

\begin{corollary} \label{cor:Gauss-twists}
 Let $\chi, \rho \in \Irr_{\C}(G)$ be two irreducible characters of $G$.
 If $\rho$ is unramified then 
 \[
  \tau_{\Q_p}(\ind_K^{\Q_p}(\chi \otimes \rho)) = \rho(c_{\chi}^{-1}) \tau_{\Q_p}(\ind_K^{\Q_p}(\chi)).
 \]
\end{corollary}

Let $\kappa$ be the $p$-adic cyclotomic character
\[
    \kappa: G_{\Q} \longrightarrow \Z_{p}^{\times},
\]
defined by $\omega(\zeta) = \zeta^{\kappa(\omega)}$ for every 
$\omega \in G_{\Q}$
and every $p$-power root of unity $\zeta$.

\begin{theorem} \label{thm:Galois-action}
 Let $\chi$ be a character of $G_K$ with open kernel. 
 Then for every $\omega \in G_{\Q}$ one has
 \[
  \omega^{-1}(\tau_K(\omega \circ \chi)) = \tau_K(\chi) \cdot \det_{\chi}(\kappa(\omega)).
 \]
\end{theorem}
\begin{proof}
 This is \cite[Chapter II, Theorem 5.1]{MR0447187}, for instance.
\end{proof}

\subsection{Galois Gauss sums in unramified $\Z_p$-extensions} \label{sec:GGs-Z_p-extensions}
In this subsection $p$ is assumed to be odd.
For a $p$-adic field $K$ we let $K_{\infty}$ be the unique unramified
$\Z_p$-extension. Then $\Gamma_K := \Gal(K_{\infty}/K) \simeq \Z_p$ is topologically
generated by the Frobenius automorphism $\phi_K \in \Gamma_K$.

Now let $L/K$ be a finite Galois extension of $p$-adic fields with Galois group $G$.
Then $L_{\infty} / K$ is a $p$-adic Lie extension of dimension $1$ and we put
$\mathcal{G} := \Gal(L_{\infty}/K)$ and $H := \Gal(L_{\infty}/K_{\infty})$. 
By the argument given in \cite[\S 1]{MR2114937} the short exact sequence
\begin{equation*} 
	1 \longrightarrow H \longrightarrow \mathcal{G} \longrightarrow \Gamma_K
	\longrightarrow 1
\end{equation*}
splits. We may therefore write $\mathcal{G}$ as a
semi-direct product $\mathcal{G} \simeq H \rtimes \Gamma$, where
$\Gamma \simeq \Gamma_K \simeq \Z_p$. 
Note that $\Gamma_K$ now plays the same role as $\overline{\Gamma}$ in
\S \ref{subsec:Iwasawa-algebras} and \S \ref{subsec:K-of-Iwasawa-algebras}.

The maximal abelian extension $\Q_p^{\ab}$ of $\Q_p$ is the compositum of the maximal
unramified extension $\Q_p^{\ur}$ of $\Q_p$ and the totally ramified abelian
extension $\Q_p^{\ram} := \Q_p(\zeta_{p^{\infty}})$ which is obtained by adjoining
all $p$-power roots of unity. For $\omega \in G_{\Q_p}$ we define
$\omega^{\ur} \in G_{\Q_p}^{\ab}$ by declaring 
$\omega^{\ur}|_{\Q_p^{\ur}} = \omega|_{\Q_p^{\ur}}$ and 
$\omega^{\ur}|_{\Q_p^{\ram}} = \id$. Similarly, we define
$\omega^{\ram} \in G_{\Q_p}^{\ab}$ by declaring 
$\omega^{\ram}|_{\Q_p^{\ram}} = \omega|_{\Q_p^{\ram}}$ and 
$\omega^{\ram}|_{\Q_p^{\ur}} = \id$.

Each $\omega \in G_{\Q_p}$ acts on the finite set of left cosets 
$G_{\Q_p} / G_K$ by left multiplication and we let
$\epsilon_{K/\Q_p}(\omega) \in \left\{\pm 1 \right\}$ be the signature
of this permutation.
Then \cite[Chapter II, Proposition 3.2]{MR0447187}
states that for any character $\chi$ of $G_K$ with open kernel one has
\begin{equation} \label{eqn:Ver-chi-relation}
\det_{\ind_K^{\Q_p}(\chi)}(\omega) = 
\epsilon_{K/\Q_p}(\omega)^{\chi(1)} \cdot
\det_{\chi}(\Ver_{K/\Q_p}(\omega)).
\end{equation}

\begin{theorem} \label{thm:GGs-in-Z_p}
 Choose an isomorphism $j: \C \simeq \C_p$. 
 \begin{enumerate}
  \item 
 The map 
 \begin{eqnarray*}
  \tau_{L_{\infty}/K}^{(j)}: R_p(\mathcal{G}) & \longrightarrow & (\Q_p^c \otimes_{\Z_p} \Lambda(\Gamma_K))^{\times}\\
  \chi & \mapsto & \phi_K^{-v_K(c_{\chi})} j\left(\tau_{\Q_p}(\ind_K^{\Q_p}(j^{-1} \circ \chi))\right)
 \end{eqnarray*}
 belongs to $\Hom^W(R_p(\mathcal{G}), (\Q_p^c \otimes_{\Z_p} \Lambda(\Gamma_K))^{\times})$, where 
 $c_{\chi}$ is a generator of $\mathfrak{f}(\chi)\mathfrak{D}_K^{\deg(\chi)}$.
 \item
 If $j': \C \simeq \C_p$ is a second choice of isomorphism, then
 \[
   \tau_{L_{\infty}/K}^{(j)} \cdot \left(\tau_{L_{\infty}/K}^{(j')}\right)^{-1}
   \in \Det(K_1(\Lambda(\mathcal{G}))).
 \]
 \item
 For every $\omega \in G_{\Q_p}$ we have
 \[
  \omega\left(\tau_{L_{\infty}/K}^{(j)}(\omega^{-1} \circ \chi)\right) = 
  \tau_{L_{\infty}/K}^{(j)}(\chi) \cdot \det_{\ind_K^{\Q_p}(\chi)}(\omega^{\ram}).
 \]
 \end{enumerate}
\end{theorem}

\begin{proof}
 Let $\rho$ be a character of type $W$. Then $\rho$ factors through $\Gamma_K$
 and thus it is unramified. We compute
 \begin{eqnarray*}
  \tau_{L_{\infty}/K}^{(j)}(\chi \otimes \rho)
  & = & \phi_K^{-v_K(c_{\chi \otimes \rho})} j\left(\tau_{\Q_p}(\ind_K^{\Q_p}(j^{-1} \circ (\chi \otimes \rho)))\right)\\
  & = & \phi_K^{-v_K(c_{\chi})} \rho(c_{\chi}^{-1}) j\left(\tau_{\Q_p}(\ind_K^{\Q_p}(j^{-1} \circ \chi))\right)\\
  & = & \rho^{\sharp}(\phi_K^{-v_K(c_{\chi})}) j\left(\tau_{\Q_p}(\ind_K^{\Q_p}(j^{-1} \circ \chi))\right)\\
  & = & \rho^{\sharp}\left(\tau_{L_{\infty}/K}^{(j)}(\chi)\right),
 \end{eqnarray*}
 where the second equality follows from Corollary \ref{cor:Gauss-twists}.
 The third equality holds as $\rho(\phi_K) = \rho(\pi_K)$ by local class field theory. This proves (i). For (ii) we write 
 $j'|_{\Q^c} = j|_{\Q^c} \circ \omega$ with $\omega \in G_{\Q}$.
 Then Theorem \ref{thm:Galois-action} implies that
 \begin{eqnarray*}
 \tau_{L_{\infty}/K}^{(j)}(\chi) \cdot
  \left(\tau_{L_{\infty}/K}^{(j')}(\chi)\right)^{-1} & = & j\left(\det_{\ind_K^{\Q_p}(j^{-1} \circ \chi)}(\kappa(\omega)) \right) \\
  & = & \det_{\ind_K^{\Q_p}(\chi)}(\tilde \omega),
 \end{eqnarray*}
 where $\tilde \omega := (\kappa(\omega), \Q_p) \in G_{\Q_p}^{\ab}$. 
 Let $F := L_{\infty} \cap K^{\ab}$ and choose any $\tilde g \in \mathcal{G}$
 such that $\tilde g|_F = \Ver_{K/\Q_p}(\tilde \omega)|_F$.
 Then we have that 
 $\det_{\chi}(\tilde g) = \det_{\chi}(\Ver_{K/\Q_p}(\tilde \omega))$
 and
 \begin{eqnarray*}
	 \tilde g|_{K_{\infty}} & = & \Ver_{K/\Q_p}(\tilde \omega)|_{K_{\infty}}\\
	 & = & \Ver_{K/\Q_p}((\kappa(\omega), \Q_p))|_{K_{\infty}}\\
	 & = & (\kappa(\omega), K)|_{K_{\infty}}\\
	 & = & \id_{K_{\infty}}
 \end{eqnarray*}
 by local class field theory (use the right diagram \eqref{eqn:LCT-norms}).
 It now follows from Example \ref{ex:Det(g)}
 that
 \[
	 \left[ \chi \mapsto \det_{\chi}(\Ver_{K/\Q_p}(\tilde \omega)) \right]
	 = \Det(\tilde g) \in \Det(K_1(\Lambda(\mathcal{G}))).
 \] 
 As $\Hom_{\Q_p^c[H]}(V_{\chi}, \mathcal{Q}^c(\mathcal{G}))$ is a
 $\mathcal{Q}^c(\Gamma_K)$-vector space of dimension
 $\chi(1)$ (see the proof of \cite[Proposition 6]{MR2114937}), 
 we also have that
 \[
	 \left[\chi \mapsto \epsilon_{K/\Q_p}(\tilde \omega)^{\chi(1)} \right]
	 = \Det(\epsilon_{K/\Q_p}(\tilde \omega)) 
	 \in \Det(K_1(\Lambda(\mathcal{G}))).
 \]
 Now \eqref{eqn:Ver-chi-relation} implies (ii). 
 Finally, (iii) is also deduced from
 Theorem \ref{thm:Galois-action} once we note that
 $\omega^{\ram} = (\kappa(\omega), \Q_p)^{-1}$ for every
 $\omega \in G_{\Q_p}^{\ab}$.
\end{proof}

\subsection{Functorialities}
Let $N$ be a finite normal subgroup of $\mathcal{G}$ and let $\mathcal{H}$ be an open subgroup of $\mathcal{G}$.
There are canonical maps (see \cite[\S 3]{MR2114937})
\begin{eqnarray*}
	\quot^{\mathcal{G}}_{\mathcal{G}/N}: & \Hom^{W}(R_{p}(\mathcal{G}), \mathcal{Q}^{c}(\Gamma_{K})^{\times}) \longrightarrow &
	\Hom^{W}(R_{p}(\mathcal{G}/N), \mathcal{Q}^{c}(\Gamma_{K})^{\times}),\\
	\res^{\mathcal{G}}_{\mathcal{H}}: & \Hom^{W}(R_{p}(\mathcal{G}), \mathcal{Q}^{c}(\Gamma_{K})^{\times}) \longrightarrow &
	\Hom^{W}(R_{p}(\mathcal{H}), \mathcal{Q}^{c}(\Gamma_{K'})^{\times}),
\end{eqnarray*}
where $K' := L_{\infty}^{\mathcal{H}}$; here for $f \in \Hom^{W}(R_{p}(\mathcal{G}), \mathcal{Q}^{c}(\Gamma_{K})^{\times})$
we have $(\quot^{\mathcal{G}}_{\mathcal{G}/N} f)(\chi) = f(\infl^{\mathcal{G}}_{\mathcal{G}/N} \chi)$ and
$(\res^{\mathcal{G}}_{\mathcal{H}} f)(\chi') = f(\ind^{\mathcal{G}}_{\mathcal{H}} \chi')$ for $\chi \in R_{p}(\mathcal{G}/N)$
and $\chi' \in R_{p}(\mathcal{H})$. Note that we view
$\mathcal{Q}^{c}(\Gamma_{K'})$ as a subfield of $\mathcal{Q}^{c}(\Gamma_{K})$
via $\phi_{K'} \mapsto \phi_K^{f_{K'/K}}$.

\begin{prop} \label{prop:functoriality-Gauss}
	Choose an isomorphism $j: \C \simeq \C_p$. Then the following statements
	hold.
	\begin{enumerate}
		\item 
		Let $N$ be a finite normal subgroup of $\mathcal{G}$ and put
		$L_{\infty}' := L_{\infty}^N$. Then
		\[
			\quot^{\mathcal{G}}_{\mathcal{G}/N}
			\left(\tau_{L_{\infty}/K}^{(j)}\right) = 
			\tau_{L_{\infty}'/K}^{(j)}.
		\]
		\item
		Let $\mathcal{H}$ be an open subgroup of $\mathcal{G}$ and put
		$K' := L_{\infty}^{\mathcal{H}}$. Then
		\[
		\res^{\mathcal{G}}_{\mathcal{H}}
		\left(\tau_{L_{\infty}/K}^{(j)}\right) = 
		\tau_{L_{\infty}/K'}^{(j)}.
		\]
	\end{enumerate}
\end{prop}

\begin{proof}
	Part (i) is easy so that we only prove part (ii).
	Let $\chi' \in R_p(\mathcal{H})$. 
	We have to show that
	$\tau_{L_{\infty}/K}^{(j)}(\ind_{\mathcal{H}}^{\mathcal{G}}(\chi')) 
	= \tau_{L_{\infty}/K'}^{(j)}(\chi')$.
	We clearly have
	$\ind^{\Q_p}_K (j^{-1} \circ \ind_{\mathcal{H}}^{\mathcal{G}}(\chi'))
	= \ind^{\Q_p}_{K'} (j^{-1} \circ \chi')$ so that it suffices to show
	that 
	\begin{equation} \label{eqn:phi-equality}
		\phi_K^{-v_K\left(c_{\ind_{\mathcal{H}}^{\mathcal{G}}(\chi')}\right)} 
		= \phi_{K'}^{- v_{K'}\left(c_{\chi'}\right)}.
	\end{equation}
	For this we compute (see \cite[p.~23]{MR0447187} for the first equality)
	\begin{eqnarray*}
		\mathfrak{f}(\ind_{\mathcal{H}}^{\mathcal{G}}(\chi'))
		\mathfrak D_K^{\deg(\ind_{\mathcal{H}}^{\mathcal{G}}(\chi'))}
		& = & N_{K'/K}\left(\mathfrak f(\chi') 
		\mathfrak D_{K'/K}^{\deg(\chi')}\right)
		\cdot \mathfrak D_K^{\deg(\ind_{\mathcal{H}}^{\mathcal{G}}(\chi'))}\\
		& = & N_{K'/K}\left(\mathfrak f(\chi') 
		\mathfrak D_{K'/K}^{\deg(\chi')}
		\mathfrak D_K^{\deg(\chi')}\right)\\
		& = & N_{K'/K}\left(\mathfrak f(\chi') 
		\mathfrak D_{K'}^{\deg(\chi')}\right).
	\end{eqnarray*}
	As $v_K \circ N_{K'/K} = f_{K'/K} \cdot v_{K'}$ it follows that
	$v_K(c_{\ind_{\mathcal{H}}^{\mathcal{G}}(\chi')}) = 
	f_{K'/K} \cdot v_{K'}\left(c_{\chi'}\right)$.
	Since $\phi_K^{f_{K'/K}} = \phi_{K'}$ we get \eqref{eqn:phi-equality}.
\end{proof}

\section{The cohomological and the unramified term} \label{sec:coh-term}

\subsection{Galois cohomology} \label{subsec:Galois-coh}
 If $F$ is a field and $M$ is a topological
 $G_F$-module, we write $R\Gamma(F,M)$ for the complex of continuous 
 cochains of $G_F$ with coefficients in $M$ and $H^i(F,M)$ for its
 cohomology in degree $i$. Similarly, we write $H_i(F,M)$
 for the $i$-th homology group of $G_F$ with coefficients in $M$. 
 If $F$ is an algebraic extension
 of $\Q_p$ and $M$ is a discrete or compact $G_F$-module, then for $r \in \Z$
 we denote the $r$-th Tate twist of $M$ by $M(r)$. For any abelian group
 $A$ we write $\widehat A$ for its $p$-completion, that is
 $\widehat A = \varprojlim_n A / p^n A$.
 
 Now let $L/K$ be a finite Galois extension of $p$-adic fields with
 Galois group $G$. We recall that
 \[
	 C_L^{\bullet} := R\Gamma(L, \Z_p(1))[1] \in \mathcal{D}(\Z_p[G])
 \]
 is a perfect complex of $\Z_p[G]$-modules which is acyclic outside degrees
 $0$ and $1$ and that there are canonical isomorphisms of
 $\Z_p[G]$-modules
 \begin{equation} \label{eqn:coh-of-C_L}
	 H^0(C_L^{\bullet}) \simeq \widehat{L^{\times}}, \quad
	 H^1(C_L^{\bullet}) \simeq \Z_p.
 \end{equation}
 We note that local Tate duality induces an isomorphism
 \begin{equation} \label{eqn:Tate-duality}
	 C_L^{\bullet} \simeq R\Gamma(L, \Q_p / \Z_p)^{\vee}[-1]
 \end{equation}
 in $\mathcal{D}(\Z_p[G])$, where we write $(C^{\bullet})^{\vee}$ for $R\Hom(C^{\bullet}, \Q_p / \Z_p)$
 for any complex $C^{\bullet}$. 
 
 Now let $L_{\infty}$ be an arbitrary $\Z_p$-extension of $L$
 with Galois group $\Gamma_L$ and for each $n \in \N$ 
 let $L_n$ be its $n$-th layer. We assume that
 $L_{\infty} / K$ is again a Galois extension with Galois group
 $\mathcal{G} := \Gal(L_{\infty} / K)$.
 We let $X_{L_{\infty}}$ denote the Galois group
 over $L_{\infty}$ of the maximal abelian pro-$p$-extension
 of $L_{\infty}$. We put
 \[
 Y_{L_{\infty}} := \Delta(G_K)_{G_{L_{\infty}}}
 = \Z_p \widehat \otimes_{\Lambda(G_{L_{\infty}})} \Delta(G_K)
 \]
 and observe that $\pd_{\Lambda(\mathcal{G})}(Y_{L_{\infty}}) \leq 1$
 by \cite[Theorem 7.4.2]{MR2392026}. As $H_1(L_{\infty}, \Z_p)$
 canonically identifies with $X_{L_{\infty}}$, taking 
 $G_{L_{\infty}}$-coinvariants of the obvious short exact sequence
 \begin{equation} \label{eqn:aug-ses}
 0 \longrightarrow \Delta(G_K) \longrightarrow \Lambda(G_K) \longrightarrow
 \Z_p \longrightarrow 0
 \end{equation}
 yields an exact sequence
 \begin{equation} \label{eqn:4term-sequence}
	0 \longrightarrow X_{L_{\infty}} \longrightarrow Y_{L_{\infty}}
	\longrightarrow \Lambda(\mathcal{G}) \longrightarrow \Z_p \longrightarrow 0
 \end{equation}
 of $\Lambda(\mathcal{G})$-modules (this should be compared to the 
 sequence constructed by Ritter and Weiss \cite[\S 1]{MR1935024}).
 The middle arrow thus defines a perfect complex
 of $\Lambda(\mathcal{G})$-modules
 \[
	 C_{L_{\infty}}^{\bullet}: \dots \longrightarrow 0 \longrightarrow
	 Y_{L_{\infty}} \longrightarrow \Lambda(\mathcal{G})
	 \longrightarrow 0 \longrightarrow \dots,
 \]
 where we place $Y_{L_{\infty}}$ in degree $0$. This complex
 obviously is acyclic outside degrees $0$ and $1$ and we have
 isomorphisms
 \begin{equation} \label{eqn:coh-at-infty}
	 H^0(C_{L_{\infty}}^{\bullet}) \simeq X_{L_{\infty}}, \quad
	 H^1(C_{L_{\infty}}^{\bullet}) \simeq \Z_p.
 \end{equation}
 The following is a variant of \cite[Corollary 4.16]{MR1894938}
 and \cite[Theorem 2.4]{MR3072281}.
 
 \begin{prop} \label{prop:complex-infinity}
 	With the above notation, we have isomorphisms
 	\begin{eqnarray*}	 	
	 	C_{L_{\infty}}^{\bullet} & \simeq &
	 	R\Gamma(L_{\infty}, \Q_p / \Z_p)^{\vee}[-1] \\
	 	&\simeq & \varprojlim_n C_{L_{n}}^{\bullet} 
 	\end{eqnarray*}
 	in $\mathcal{D}(\Lambda(\mathcal{G}))$.
 \end{prop}
 
 \begin{proof}
 	Tate duality \eqref{eqn:Tate-duality} 
 	and \cite[Chapter III, Lemma 1.16]{MR559531}
 	imply that we have isomorphisms
 	\begin{eqnarray*}
 		\varprojlim_n C_{L_{n}}^{\bullet} 
 		& \simeq & \varprojlim_n R\Gamma(L_n, \Q_p / \Z_p)^{\vee}[-1] \\
 		& \simeq & (\varinjlim_n R\Gamma(L_n, \Q_p / \Z_p))^{\vee}[-1] \\
 		& \simeq & R\Gamma(L_{\infty}, \Q_p / \Z_p)^{\vee}[-1]
 	\end{eqnarray*}
 	in $\mathcal{D}(\Lambda(\mathcal{G}))$. This gives the second isomorphism
 	of the theorem. In particular, the complex
 	$R\Gamma(L_{\infty}, \Q_p / \Z_p)^{\vee}$ is acyclic outside degrees
 	$-1$ and $0$ (see also \cite[Theorem 7.1.8(i)]{MR2392026}).
 	
 	For any compact right $\Lambda(G_{L_{\infty}})$-module $M$ and
 	any discrete left $\Lambda(G_{L_{\infty}})$-module $N$ (considered
 	as complexes in degree $0$) there is an isomorphism
 	\[
	 	M \widehat \otimes^{\mathbb L}_{\Lambda(G_{L_{\infty}})}
	 	N^{\vee} \simeq
	 	R\Hom_{\Lambda(G_{L_{\infty}})}(M,N)^{\vee}
 	\]
 	in $\mathcal{D}(\Lambda(\mathcal{G}))$ by
 	\cite[Corollary 5.2.9]{MR2392026}. We note that
 	$R\Gamma(L_{\infty}, N) \simeq R\Hom_{\Lambda(G_{L_{\infty}})}(\Z_p,N)$
 	and so specializing $M = \Z_p$ and $N = \Q_p/\Z_p$ yields an isomorphism
 	\begin{equation} \label{eqn:adjointness}
 	\Z_p \widehat \otimes^{\mathbb L}_{\Lambda(G_{L_{\infty}})}
 	\Z_p \simeq
 	R\Gamma(L_{\infty},\Q_p / \Z_p)^{\vee}
 	\end{equation}
 	in $\mathcal{D}(\Lambda(\mathcal{G}))$. Therefore 
 	$\Z_p \widehat \otimes^{\mathbb L}_{\Lambda(G_{L_{\infty}})} \Z_p$
 	is also acyclic outside degrees $-1$ and $0$. We now apply
 	the functor $\Z_p \widehat \otimes_{\Lambda(G_{L_{\infty}})} -$
 	to sequence \eqref{eqn:aug-ses} and obtain a long exact sequence
 	in homology which coincides with \eqref{eqn:4term-sequence}.
 	In particular, we derive from this that
 	\[
	 	H_i(L_{\infty}, \Delta(G_K)) = H_i(L_{\infty}, \Lambda(G_K)) 
	 	= 0 \quad \mbox{for all } i \geq 1.
 	\]
 	Hence the exact triangle
 	\[
	 	\Z_p \widehat \otimes^{\mathbb L}_{\Lambda(G_{L_{\infty}})} \Delta(G_K)
	 	\longrightarrow
	 	\Z_p \widehat \otimes^{\mathbb L}_{\Lambda(G_{L_{\infty}})} \Lambda(G_K)
	 	\longrightarrow
	 	\Z_p \widehat \otimes^{\mathbb L}_{\Lambda(G_{L_{\infty}})} \Z_p
 	\]
 	implies that the complex $C_{L_{\infty}}^{\bullet}[1]$ is isomorphic
 	to $\Z_p \widehat \otimes^{\mathbb L}_{\Lambda(G_{L_{\infty}})} \Z_p$ in
 	$\mathcal{D}(\Lambda(\mathcal{G}))$. The result follows from this
 	and \eqref{eqn:adjointness}.
 \end{proof}
 
 We now specialize to the case, where $L_{\infty}$ is the unramified
 $\Z_p$-extension of $L$.
 We put 
 $U^1(L_{\infty}) := \varprojlim_n U^1_{L_n}$ where the 
 transition maps are given by the norm maps.
 
 \begin{corollary} \label{cor:coh-at-infty-I}
 	The complex $C_{L_{\infty}}^{\bullet}$ is a perfect complex of
 	$\Lambda(\mathcal{G})$-modules which is acyclic outside degrees $0$
 	and $1$. If $L_{\infty}$ is the unramified $\Z_p$-extension of $L$,
 	then we have canonical isomorphisms of $\Lambda(\mathcal{G})$-modules
 	\[
	 	H^0(C_{L_{\infty}}^{\bullet}) \simeq U^1(L_{\infty}) 
		 	\simeq X_{L_\infty}, \quad
		H^1(C_{L_{\infty}}^{\bullet}) \simeq \Z_p.
 	\]
 	In particular, $H^i(C_{L_{\infty}}^{\bullet})$ has no non-trivial
 	finite submodule for each $i \in \Z$.
 \end{corollary}
 \begin{proof}
 After taking $p$-completions, the 
 valuation map $L_n^{\times} \twoheadrightarrow \Z$ induces
 an exact sequence of $\Z_p[\Gal(L_n/K)]$-modules
 \[
	0 \longrightarrow U^1_{L_n} \longrightarrow \widehat{L_n^{\times}}
	\longrightarrow \Z_p \longrightarrow 0. 
 \]
 Taking inverse limits, where the transition maps on the left and
 in the middle are the norm maps and on the right are multiplication
 by $p$, induces an isomorphism
 $U^1(L_{\infty}) \simeq \varprojlim_n \widehat{L_n^{\times}}$
 (see also \cite[Theorem 11.2.4(iii)]{MR2392026}).
 Moreover, we have $\varprojlim_n \widehat{L_n^{\times}} \simeq
 X_{L_{\infty}}$ by local class field theory.
 We also note that $U^1(L_{\infty})$ has no non-trivial finite submodule
 by \cite[Theorem 11.2.4(ii)]{MR2392026}.
 Now Proposition \ref{prop:complex-infinity} and
 \eqref{eqn:coh-at-infty} imply the result.
 \end{proof}
 
 \subsection{Modified Galois cohomology}
 We write $\Sigma(L)$ for the set of all embeddings $L \rightarrow \Q_p^c$
 fixing $\Q_p$ and define
 \[
	 H_L := \bigoplus_{\sigma \in \Sigma(L)} \Z_p.
 \]
 If $L/K$ is an extension of $p$-adic fields, then the restriction map
 $\Sigma(L) \rightarrow \Sigma(K)$, $\sigma \mapsto \sigma |_K$
 induces an epimorphism $H_L \twoheadrightarrow H_K$.
 If $L/K$ is a Galois extension with Galois group $G$, then $H_L$ is a free
 $\Z_p[G]$-module of rank $[K:\Q_p]$. More precisely, if we choose a lift
 $\hat \tau \in \Sigma(L)$ for every $\tau \in \Sigma(K)$, then the set 
 $\left\{\hat \tau \mid \tau \in \Sigma(K)\right\}$ constitutes a
 $\Z_p[G]$-basis of $H_L$. Following Breuning \cite[\S 2.4]{MR2078894}
 we define a perfect complex
 \[
	 K_L^{\bullet} := R\Gamma(L, \Z_p(1))[1] \oplus H_L[-1] 
	 = C_L^{\bullet} \oplus H_L[-1]
 \]
 in $\mathcal{D}(\Z_p[G])$. 
 Let $L_{\infty}$ be an arbitrary $\Z_p$-extension of $L$ such that 
 $L_{\infty} / K$ is a Galois extension with Galois group $\mathcal{G}$.
 We put
 $H_{L_{\infty}} := \varprojlim_n H_{L_n}$ which is a free
 $\Lambda(\mathcal{G})$-module of rank $[K:\Q_p]$.
 We define a complex of $\Lambda(\mathcal{G})$-modules
 \begin{eqnarray*}
 K_{L_{\infty}}^{\bullet} & := & R\Gamma(L_{\infty}, \Q_p / \Z_p)^{\vee}[-1]
  \oplus H_{L_{\infty}}[-1] \\
 & \simeq & C_{L_{\infty}}^{\bullet} \oplus H_{L_{\infty}}[-1],
 \end{eqnarray*}
 where the isomorphism in $\mathcal{D}(\Lambda(\mathcal{G}))$ has been
 established in Proposition \ref{prop:complex-infinity}.
 The following is immediate from Corollary \ref{cor:coh-at-infty-I}.
 
 \begin{corollary} \label{cor:coh-at-infty-II}
 	The complex $K_{L_{\infty}}^{\bullet}$ is a perfect complex of
 	$\Lambda(\mathcal{G})$-modules which is acyclic outside degrees $0$
 	and $1$. If $L_{\infty}$ is the unramified $\Z_p$-extension of $L$,
 	then we have canonical isomorphisms of $\Lambda(\mathcal{G})$-modules
 	\[
 	H^0(K_{L_{\infty}}^{\bullet}) \simeq U^1(L_{\infty}) 
 	\simeq X_{L_\infty}, \quad
 	H^1(K_{L_{\infty}}^{\bullet}) \simeq \Z_p \oplus H_{L_{\infty}}.
 	\]
 	In particular, $H^i(K_{L_{\infty}}^{\bullet})$ has no non-trivial
 	finite submodule for each $i \in \Z$.
 \end{corollary}
 
 \subsection{Normal bases in unramified $\Z_p$-extensions} \label{sec:normal-bases}
 If $L_{\infty}/L$ is the unramified $\Z_p$-extension, 
 the ring of integers $\mathcal{O}_{L_n}$
 is free (of rank $1$) as a module over the group ring 
 $\mathcal{O}_{L}[\Gal(L_n/L)]$ for each $n$. Thus we may choose a 
 normal integral basis generator
 $b_n \in \mathcal{O}_{L_n}$, that is
 $\mathcal{O}_{L_n} = \mathcal{O}_{L}[\Gal(L_n/L)] \cdot b_n$. 
 In fact, more is true.
 
 \begin{lemma} \label{lem:limits-IBG}
 	There exists $b \in \varprojlim_n \mathcal{O}_{\Q_{p,n}}$
 	such that for every $p$-adic field $F$ we have
 	$\mathcal{O}_{F_n} = \mathcal{O}_F[\Gal(F_n/F)] \cdot b_n$ 
 	if we write $b = (b_n)_n \in \varprojlim_n \mathcal{O}_{F_{n}}$.
 \end{lemma}
 
 \begin{proof}
 	Let $\overline{F_n}$ and $\overline{F}$ denote the residue fields of
 	$F_n$ and $F$, respectively. Let $b_n \in \mathcal{O}_{F_n}$ and 
 	write $\overline{b_n}$ for its image in $\overline{F_n}$.
 	Then $\mathcal{O}_{F_n} = \mathcal{O}_F[\Gal(F_n/F)] \cdot b_n$
 	if and only if $\Tr_{\overline{F_n} / \overline{F}}(\overline{b_n}) \not=0$
 	by \cite[Propositions 2.2 and 5.1]{MR3411126} 
 	(see also \cite[Remark 2.3]{MR3411126}).
 	Since $F_{n+1}/F_n$ is unramified, the trace maps 
 	$\Tr_{F_{n+1}/F_n}: \mathcal{O}_{F_{n+1}} \rightarrow \mathcal{O}_{F_n}$
 	are surjective. Therefore the canonical map
 	$\varprojlim_n \mathcal{O}_{F_n} \rightarrow \mathcal{O}_F$ is also 
 	surjective. Let $b = (\beta_n)_n \in \varprojlim_n \mathcal{O}_{\Q_{p,n}}$
 	be a pre-image of $1 \in \Z_p$ (i.e.\ such that $\beta_0=1$).
 	Then $\Tr_{\overline{\Q_{p,n}} / \overline{\Q_p}}(\overline{\beta_n}) = \overline{\beta_0}
 	=\overline{1}$ is non-zero for all $n \geq 0$.
 	Now let $m \geq 0$ be such that  $F \cap \Q_{p, \infty} = \Q_{p,m}$.
 	Then we have $b = (b_n)_n \in \varprojlim_n \mathcal{O}_{F_{n}}$, where
 	$b_n = \beta_{n+m}$. It follows that
 	\[
 	\Tr_{\overline{F_n} / \overline{F}}(\overline{b_n}) = 
 	\Tr_{\overline{\Q_{p,n+m}} / \overline{\Q_{p,m}}}(\overline{\beta_{n+m}}) 
 	= \overline{\beta_m} \not= 0
 	\]
 	as desired.
 \end{proof}
 
 As before let $L/K$ be a finite Galois extension of $p$-adic fields
 with Galois group $G$ and put $G_n := \Gal(L_n/K)$, $n \geq 0$.
 Then we have $ \varprojlim_n G_n \simeq \mathcal{G}
 \simeq H \rtimes \Gamma$ and we let $L'$ be the fixed field under $\Gamma$.
 Then $L'_{\infty}$ identifies with $L_{\infty}$ and we may suppose
 that $L = L'_{n_0}$ for some integer $n_0 \geq 0$. 
 Note that now $G_n$ may be written as $G_n \simeq H \rtimes \Gal(L_n/L')$,
 where $\Gal(L_n/L')$ is a cyclic group of order $p^{n+n_0}$.
 By Lemma \ref{lem:limits-IBG} we have
 $\varprojlim_n \mathcal{O}_{L_n} \simeq \varprojlim_n \mathcal{O}_{L'_n}
 \simeq \Lambda^{\mathcal{O}_{L'}}(\Gamma)$
 as $\Lambda(\Gamma)$-modules. In particular,
 $\varprojlim_n \mathcal{O}_{L_n}$ is a free
 $\Lambda(\Gamma)$-module of rank $[L':\Q_p]$.
 
 \begin{prop} \label{prop:normal-bases-tower}
 	There is an $a = (a_n)_n \in \varprojlim_n \mathcal{O}_{L_n}$
 	with the following properties:
 	\begin{enumerate}
 		\item 
 		each $a_n$ generates a normal basis for $L_n/K$;
 		\item
 		the $\Lambda^{\mathcal{O}_K}(\mathcal{G})$-linear map
 		\begin{eqnarray*}
 			\Lambda^{\mathcal{O}_K}(\mathcal{G}) & \longrightarrow & 
 			\varprojlim_n \mathcal{O}_{L_n}\\
 			1 & \mapsto & a
 		\end{eqnarray*}
 		is injective and its cokernel is a finitely generated
 		$\Lambda(\Gamma)$-torsion module whose $\lambda$-invariant vanishes.
 	\end{enumerate}
 \end{prop}
 
 \begin{proof}
 	Choose $b = (b_n)_n \in \varprojlim_n \mathcal{O}_{L'_{n}}$
 	as in Lemma \ref{lem:limits-IBG} with $F = L'$. Let $c \in
 	\mathcal{O}_L$ generate a normal basis for $L/K$
 	and put $a_n' := \Tr_{L/L'}(c) \cdot b_n \in \mathcal{O}_{L'_n}$
 	and $a_n := a'_{n+n_0} \in \mathcal{O}_{L_n}$. 
 	Then $a := (a_n)_n$ belongs to $\varprojlim_n \mathcal{O}_{L_n}$.
 	In order to verify the
 	first property we have to show that $a_n'$ generates a normal
 	basis for $L'_n/K$ for $n \geq n_0$. For this let $x \in L_n'$
 	be arbitrary. Let $\gamma_n$ be a generator of the cyclic group
 	$\Gal(L'_n/L')$. Then we may write 
 	$x = \sum_{i=0}^{p^n-1} y_i \gamma_n^i(b_n)$ with $y_i \in L'$,
 	$0 \leq i < p^n$. As $L' \subseteq L$ we may likewise write
 	$y_i = \sum_{g \in G} z_{i,g} g(c)$ with $z_{i,g} \in K$,
 	$g \in G$ for each $i$. Since $y_i$ is invariant under
 	$\Gal(L/L')$ we find that $z_{i,g} = z_{i,g'g}$
 	whenever $g' \in \Gal(L/L')$, and so 
 	we may write
 	$y_i = \sum_{h \in H} z_{i,h} h(\Tr_{L/L'}(c))$. 
 	As $\Tr_{L/L'}(c)$ is invariant under $\Gal(L_n'/L')$
 	and $b_n$ is invariant under $H$ by Lemma \ref{lem:limits-IBG}, we
 	find that
 	\[
	 	x = \sum_{i=0}^{p^n-1} \sum_{h \in H} z_{i,h} h \gamma_n^i(a_n')
 	\]
 	as desired. 
 	Moreover, if $x \in \mathcal{O}_{L_n'}$ then we may
 	choose each $y_i \in \mathcal{O}_{L'}$. If $p^m$ is the index
 	of $\mathcal{O}_K[G] \cdot c$ in $\mathcal{O}_L$ then
 	$p^m z_{i,h} \in \mathcal{O}_K$ for all $h \in H$, $0 \leq i < p^n$.
 	Let $C_n$ be the cokernel of the injection 
 	$\mathcal{O}_K[G_n] \rightarrow \mathcal{O}_{L_n}$ that maps $1$
 	to $a_n$. Then $p^m C_n = 0$ for all $n \geq 0$.
 	For each $n \geq 0$ we now have a commutative diagram
 	\[ \xymatrix{
	 	0 \ar[r] & \mathcal{O}_K[G_{n+1}] \ar[r] \ar[d] & 
		 	\mathcal{O}_{L_{n+1}} \ar[d]^{\Tr_{L_{n+1 / L_n}}} \ar[r]
		 	& C_{n+1} \ar[d] \ar[r] & 0\\
		0 \ar[r] & \mathcal{O}_K[G_{n}] \ar[r] & 
		\mathcal{O}_{L_{n}} \ar[r] & C_n \ar[r] & 0
 	}\]
 	with exact rows and surjective vertical maps. Taking inverse limits
 	is therefore exact and we obtain an exact sequence of
 	$\Lambda^{\mathcal{O}_K}(\mathcal{G})$-modules
 	\[
 	0 \longrightarrow \Lambda^{\mathcal{O}_K}(\mathcal{G}) \longrightarrow 
 	\varprojlim_n \mathcal{O}_{L_n} \longrightarrow C_{\infty}
 	\longrightarrow 0,
 	\]
 	where $C_{\infty} := \varprojlim_n C_n$. As
 	$\Lambda^{\mathcal{O}_K}(\mathcal{G})$ and 
 	$\varprojlim_n \mathcal{O}_{L_n}$ are free $\Lambda(\Gamma)$-modules
 	of the same (finite) rank, the cokernel $C_{\infty}$ is a finitely
 	generated $\Lambda(\Gamma)$-torsion module.
 	Since $p^m$ annihilates $C_{\infty}$, its $\lambda$-invariant
 	vanishes. Thus the second property holds as well.
 \end{proof}

 \subsection{The logarithm in unramified $\Z_p$-extensions}
 For any $p$-adic field $L$ we let $\log_L: L^{\times} \rightarrow L$
 denote the $p$-adic logarithm, normalized as usual such that $\log_L(p) = 0$.
 If $L_{\infty}/L$ is the unramified $\Z_p$-extension of $L$
 with $n$-th layer $L_n$, then we simply write $\log_n$ for $\log_{L_n}$.
 
 \begin{prop} \label{prop:log-limits}
 	The maps $\log_n: U_{L_n}^1 \rightarrow L_n$ induce a 
 	well defined injective map
 	\[
 	\log_{\infty}: U^1(L_{\infty}) \longrightarrow
 	\Q_p \otimes_{\Z_p} \varprojlim_n \mathcal{O}_{L_n}
 	\]
 	and an isomorphism of $\Q_p \otimes_{\Z_p} \Lambda(\mathcal{G})$-modules
 	\[
 	\log_{\infty}: \Q_p \otimes_{\Z_p} U^1(L_{\infty}) 
 	\stackrel{\simeq}{\longrightarrow}
 	\Q_p \otimes_{\Z_p} \varprojlim_n \mathcal{O}_{L_n}.
 	\]
 \end{prop}
 
 \begin{proof}
 	If $p$ is odd, the $p$-adic logarithm
 	induces an isomorphism $U_{\Q_p}^1 \simeq p \Z_p$.
 	For $p=2$ we have $U_{\Q_2}^2 \simeq 4 \Z_2$. As for every
 	$u \in U^1_{\Q_2}$ we have $u^2 \in U^2_{\Q_2}$, it follows
 	that $\log_{\Q_2}(u) = 2^{-1}\log_{\Q_2}(u^2)$ belongs to $2 \Z_2$.
 	For any $p$ and each $n \geq 0$ we therefore have 
 	a commutative square
 	\[\xymatrix{
 		U_{L_n}^1
 		\ar[rr]^{\log_n} \ar[d]_{N_{L_{n}/\Q_p}} & & 
 		p \mathfrak{D}_{L_n}^{-1} \ar[d]^{\Tr_{L_{n}/\Q_p}} \\
 		U_{\Q_p}^1 \ar[rr]^{\log_{\Q_p}} & & 
 		p \Z_p.
 	}\]
 	Since $L_n/L$ is unramified, we have $\mathfrak{D}_{L_n}
 	= \mathfrak{D}_{L} \mathcal{O}_{L_n}$. Therefore the denominators
 	of $\log_n(U^1_{L_n})$ are bounded independently of $n$ and thus
 	$\log_{\infty}$ is well defined. 
 	
 	The kernel of $\log_n$ consists of the 
 	$p$-power roots of unity $\mu_p(L_n)$
 	in $L_n$. As $U^1(L_{\infty})$ contains no non-trivial elements
 	of finite order by \cite[Theorem 11.2.4(ii)]{MR2392026}
 	(see also Corollary \ref{cor:coh-at-infty-I}), 
 	the map $\log_{\infty}$ is injective.
 	
 	Let $C_n$ be the cokernel of the map $\log_n$
 	in the above diagram. As $L_n / L$ is unramified, we may choose
 	an integer $m \geq 1$, independent of $n$,
 	such that $\log_n$ induces an isomorphism
 	$U_{L_n}^m \simeq \mathfrak{p}_{L_n}^m$ for all $n$. 
 	We thus have exact sequences
 	\[
 	0 \longrightarrow \mu_p(L_n) \longrightarrow 	U_{L_n}^1 / U_{L_n}^m
 	\longrightarrow p \mathfrak{D}_{L_n}^{-1} / \mathfrak{p}_{L_n}^m
 	\longrightarrow C_n \longrightarrow 0
 	\]
 	for all $n \geq 0$. 
	Choose a natural number $N$ such that $p^{N+1} 
	\mathfrak{D}_{L}^{-1} / \mathfrak{p}_{L}^m$ vanishes.
	Then $p^N$ annihilates $p \mathfrak{D}_{L_n}^{-1} / \mathfrak{p}_{L_n}^m$
	and thus $C_n$ for each $n \geq 0$.
	Therefore $\varprojlim_n C_n$ is a finitely generated 
	Iwasawa torsion module
	with vanishing $\lambda$-invariant and hence
	$\Q_p \otimes_{\Z_p} \varprojlim_n C_n = 0$ as desired.
 \end{proof}
 
 \subsection{Embeddings in unramified $\Z_p$-extensions} \label{subsec:embeddings}
 As before let $L/K$ be a finite Galois extension of $p$-adic fields
 with Galois group $G$. The various embeddings of $L$
 into $\Q_p^c$ induce a $\Q_p^c[G]$-isomorphism
 \begin{eqnarray*}
 \rho_L: \Q_p^c \otimes_{\Q_p} L & \longrightarrow & \Q_p^c \otimes_{\Z_p} H_L
	 = \bigoplus_{\sigma \in \Sigma(L)} \Q_p^c\\
	z \otimes l & \mapsto & (z \sigma(l))_{\sigma \in \Sigma(L)}.
 \end{eqnarray*}
 We now study the behaviour of the maps $\rho_{L_n}$ along the 
 unramified tower. To lighten notation we simply write $\rho_n$
 for $\rho_{L_n}$. For any $\tau \in \Sigma(K)$ we choose a lift
 $\hat \tau: L_{\infty} \hookrightarrow \Q_p^c$. We define
 \[
	 K_{\tau} := \tau(K), \quad L_{\tau} := \hat \tau(L), \quad
	 L_{\tau,n} = \hat \tau(L_n),
 \]
 where $0 \leq n \leq \infty$, and note that these definitions do not
 depend on the particular choice of $\hat \tau$ because the fields
 $L$ and $L_n$, $0 \leq n \leq \infty$ are all Galois over $K$.
 Recall that
 \[
	 \Q_p^c \otimes_{\Z_p} H_{L_n} \simeq \bigoplus_{\tau \in \Sigma(K)} \Q_p^c[G_n].
 \]
 We let $\rho_{\tau,n}$ be the 
 composition of $\rho_n$ and the projection onto the $\tau$-component,
 that is $\rho_{\tau,n}(x) = \sum_{g \in G_n} \hat{\tau}g(x) g^{-1}$
 for every $x \in L_n$. It is clear that
 \[
	 \rho_{\tau,n}(\mathcal{O}_{L_n}) \subseteq \mathcal{O}_{L_{\tau,n}}[G_n].
 \]
 However, we will need a slightly finer result.
 For this let $\phi \in \Gamma \leq \mathcal{G}$ be the unique
 element such that $\phi$ maps to $\phi_K$ under the natural
 projection $\mathcal G \twoheadrightarrow \mathcal{G}/H \simeq \Gamma_K$.
 As in \S \ref{sec:normal-bases} we let $L'$ 
 be the fixed field $L_{\infty}^{\Gamma}$.
 Then $L_{\infty} = L'_{\infty}$ and we may suppose that
 $L = L'_{n_0}$ for some integer $n_0 \geq 0$.
 
%
%
 
 Fix an integer $n \geq 0$ and an embedding $\tau \in \Sigma(K)$.
 Let $E_{\tau}$ be either $L_{\tau,m}$ for some $n \leq m \leq \infty$
 or the completion of $L_{\tau,\infty}$. We let $\phi$
 act on $E_{\tau}$ via $\hat \tau \phi \hat \tau^{-1}$.
 Then we have
 \begin{equation} \label{eqn:E_tau-invariance}
	 E_{\tau}^{\phi^{p^n} = 1} = E_{\tau}^{\Gamma^{p^n}} = L_{\tau,n}'
	 := \hat{\tau}(L_n').
 \end{equation}
 We point out that this may depend upon the choice of $\hat{\tau}$
 for small $n$.
 Let $\mathcal{O}_{E_{\tau}}$ be the ring of integers in $E_{\tau}$.
 Then $\phi \otimes 1$ acts on the coefficients of
 \[
	 \mathcal{O}_{E_{\tau}}[G_n] = 
		 \mathcal{O}_{E_{\tau}} \otimes_{\Z_p} \Z_p[G_n],
 \]
 and $1 \otimes \phi$ acts via right multiplication
 by $\phi |_{L_n} \in G_n$. Inspired by \cite[\S 2]{MR3194646}, we define
 \[
	 \mathcal{O}_{E_{\tau}}[G_n]_{\varphi} :=
	 \left\{y \in \mathcal{O}_{E_{\tau}}[G_n] \mid (\phi \otimes 1)y
	 = y (1 \otimes \phi) \right\}
 \]
 which is easily seen to be an $\mathcal O_{L_{\tau}'}[G_n]$-submodule
 of $\mathcal{O}_{E_{\tau}}[G_n]$. As $1 \otimes \phi^{p^{n+n_0}}$ acts trivially
 on $\mathcal{O}_{E_{\tau}}[G_n]$, equation \eqref{eqn:E_tau-invariance}
 implies that in fact
 \[
	 \mathcal{O}_{E_{\tau}}[G_n]_{\varphi} = 
	 \mathcal{O}_{L_{\tau,n}}[G_n]_{\varphi}.
 \]
 
 \begin{lemma} \label{lem:rho-image}
 	For every integer $n \geq 0$ and every $\tau \in \Sigma(K)$ we have
 	\[
	 	\rho_{\tau,n}(\mathcal{O}_{L_n}) \subseteq
	 	 \mathcal{O}_{L_{\tau,n}}[G_n]_{\varphi}.
 	\]
 \end{lemma}
 
 \begin{proof}
	For $x \in \mathcal{O}_{L_n}$ we compute
	\begin{eqnarray*}
		(\phi \otimes 1) \rho_{\tau,n}(x)
		& = & (\phi \otimes 1) \sum_{g \in G_n} \hat{\tau}g(x) g^{-1}\\
		& = & \sum_{g \in G_n} \hat{\tau} \phi g(x) g^{-1}\\
		& = & \sum_{g \in G_n} \hat{\tau}g(x) g^{-1} \phi\\
		& = & \rho_{\tau,n}(x) (1 \otimes \phi)
	\end{eqnarray*}
	as desired.
 \end{proof}
 
 We now consider the $\mathcal O_{L_{\tau}'}[G_n]$-module
 $\mathcal{O}_{L_{\tau,n}}[G_n]_{\varphi}$ in more detail.
 Let us put $\Gamma_n := \Gamma / (\Gamma)^{p^n} \simeq \Gal(L_n'/L')$.
 Recall that $L_{\tau,n} = L'_{\tau, n + n_0}$ so that
 $\mathcal{O}_{L_{\tau,n}}$ is a 
 $\mathcal O_{L_{\tau}'}[\Gamma_{n+n_0}]$-module in a natural way.
 
 \begin{prop} \label{prop:phi-iso}
	 For every integer $n \geq 0$ and every $\tau \in \Sigma(K)$
	 there is a natural isomorphism of $\mathcal O_{L_{\tau}'}[G_n]$-modules \[
		 \delta_{\tau,n}: \mathcal{O}_{L_{\tau,n}}[G_n]_{\varphi} 
		 \stackrel{\simeq}{\longrightarrow} 
		 \mathcal O_{L_{\tau}'}[G_n] 
		 \otimes_{\mathcal O_{L_{\tau}'}[\Gamma_{n+n_0}]} \mathcal{O}_{L_{\tau,n}}.
	 \]
 \end{prop}
 
 \begin{proof}
 	Let $y  = \sum_{g \in G_n} y_g g \in \mathcal{O}_{L_{\tau,n}}[G_n]$
 	be arbitrary. Then we have
 	\begin{equation} \label{eqn:y-phi-invariant}
 	y \in \mathcal{O}_{L_{\tau,n}}[G_n]_{\varphi}
 	\Longleftrightarrow \phi(y_g) = y_{g \phi_n^{-1}} \quad
 	\forall g \in G_n,
 	\end{equation}
 	where we set $\phi_n := \phi |_{L_n} \in \Gamma_{n+n_0} \leq G_n$.
 	Suppose that \eqref{eqn:y-phi-invariant} holds for $y$.
 	Let $C \subseteq G_n$ be a set of left coset representatives
 	of $G_n / \Gamma_{n+n_0}$. Define a map
 	\begin{eqnarray*}
	 	\delta_{\tau,n}:
	 	\mathcal{O}_{L_{\tau,n}}[G_n]_{\varphi} & \longrightarrow &
	 	\mathcal O_{L_{\tau}'}[G_n] 
	 	\otimes_{\mathcal O_{L_{\tau}'}[\Gamma_{n+n_0}]} \mathcal{O}_{L_{\tau,n}}\\
	 	y & \mapsto & \sum_{c \in C} c \otimes y_c,
 	\end{eqnarray*}
 	where unadorned tensor products denote tensor products over
 	$\mathcal O_{L_{\tau}'}[\Gamma_{n+n_0}]$ in this proof. 
 	This map does actually not depend on the choice of $C$.
 	Let $C'$ be a second choice of left coset representatives.
 	Then for each $c' \in C'$ there is a unique $c \in C$ and an
 	integer $j$ such that $c' = c \phi_n^j$. We thus have
 	\[
	 	c' \otimes y_{c'} = c \phi_n^j \otimes \phi^{-j}(y_{c})
	 	= c \otimes y_c
	 \]
	 by \eqref{eqn:y-phi-invariant} as desired. We now show that
	 $\delta_{\tau,n}$ is $G_n$-equivariant. For this let
	 $g'\in G_n$ and $c \in C$ be arbitrary. Then there is a
	 unique $\tilde c \in C$ and an integer $j$ such that
	 $g'c = \tilde c \phi_n^j$. We compute
	 \begin{eqnarray*}
		 g'(c \otimes y_c) & = & \tilde c \phi_n^j \otimes y_c
			 = \tilde c \otimes \phi^j(y_c)\\
			 & = & \tilde c \otimes y_{c \phi_n^{-j}}
			 = \tilde c \otimes y_{(g')^{-1} \tilde c},
	 \end{eqnarray*}
	 where the third equality is \eqref{eqn:y-phi-invariant}.
	 As $y_{(g')^{-1} \tilde c}$ is the coefficient at $\tilde c$
	 of $g' y$, we see that indeed $\delta_{\tau,n}(g'y) = 
	 g' \delta_{\tau,n}(y)$.
	 Finally, it is easily checked that
	 \begin{eqnarray*}
	 	\mathcal O_{L_{\tau}'}[G_n] 
	 	\otimes_{\mathcal O_{L_{\tau}'}[\Gamma_{n+n_0}]} \mathcal{O}_{L_{\tau,n}}
	 	 & \longrightarrow & \mathcal{O}_{L_{\tau,n}}[G_n]_{\varphi}
	 	\\
	 	\sum_{c \in C} c \otimes z_c & \mapsto &
	 	\sum_{c \in C} \sum_{i=0}^{p^{n+n_0}-1} \phi^{-i}(z_c) c \phi_n^i
	 \end{eqnarray*}
	 is an inverse of $\delta_{\tau,n}$.
 \end{proof}
 
 \begin{corollary} \label{cor:phi-limits}
	 For every integer $n \geq 0$ and every $\tau \in \Sigma(K)$
	 the $\mathcal O_{L_{\tau}'}[G_n]$-module $\mathcal{O}_{L_{\tau,n}}[G_n]_{\varphi}$ is free of rank $1$.
	 In fact, any choice of $b \in \varprojlim_n \mathcal{O}_{\Q_{p,n}}$
	 as in Lemma \ref{lem:limits-IBG}
	 defines (non-canonical) isomorphisms
	 $\beta_{\tau,n}: \mathcal{O}_{L_{\tau,n}}[G_n]_{\varphi} \simeq 
	 \mathcal O_{L_{\tau}'}[G_n]$ such that the following two diagrams commute
	 for all $n \geq 0$ and all $\tau, \tau' \in \Sigma(K)$:
	 \[\xymatrix{
	 	\mathcal{O}_{L_{\tau,n+1}}[G_{n+1}]_{\varphi}
	 	 \ar[rr]^-{\beta_{\tau,n+1}} \ar[d] & & 
	 	 \mathcal O_{L_{\tau}'}[G_{n+1}] \ar[d] \\
	 	\mathcal{O}_{L_{\tau,n}}[G_n]_{\varphi} \ar[rr]^-{\beta_{\tau,n}} & & 
	 	\mathcal O_{L_{\tau}'}[G_n],
	 }\]
	 where the vertical arrows are induced by the canonical projection
	 $G_{n+1} \twoheadrightarrow G_n$, and
	 \[\xymatrix{
	 	\mathcal{O}_{L_{\tau,n}}[G_{n}]_{\varphi}
	 	\ar[rr]^-{\beta_{\tau,n}} \ar[d] & & 
	 	\mathcal O_{L_{\tau}'}[G_{n}] \ar[d] \\
	 	\mathcal{O}_{L_{\tau',n}}[G_n]_{\varphi} \ar[rr]^-{\beta_{\tau',n}} & & 
	 	\mathcal O_{L_{\tau'}'}[G_n],
	 }\]
	 where the vertical arrows are induced by applying 
	 $\hat{\tau}' \circ \hat{\tau}^{-1}$ on the coefficients.
 \end{corollary}
 
 \begin{proof}
 	Choose $b  \in \varprojlim_n \mathcal{O}_{\Q_{p,n}}$ as in Lemma
 	\ref{lem:limits-IBG} and write $b = (b_n)_n \in
 	\varprojlim_n \mathcal{O}_{L'_{n}}$. For each $\tau \in \Sigma(K)$ we put 
 	$b_{\tau} := (b_{\tau,n})_n \in \varprojlim_n \mathcal{O}_{L'_{\tau,n}}$,
 	where $b_{\tau,n} := \hat{\tau}(b_n)$. Then we have for each
 	$\tau \in \Sigma(K)$ and each $n \geq 0$ that
 	$\mathcal{O}_{L_{\tau,n}'} = \mathcal{O}_{L_{\tau}'}[\Gamma_n] \cdot
 	b_{\tau,n}$. This induces an isomorphism of
 	$\mathcal{O}_{L_{\tau}'}[\Gamma_{n+n_0}]$-modules
 	\[
	 	B_{\tau,n}: \mathcal{O}_{L_{\tau,n}} 
	 	\simeq \mathcal{O}_{L_{\tau}'}[\Gamma_{n+n_0}]
 	\]
 	which maps $b_{\tau,n+n_0}$ to $1$.
	We let
	\[
	\beta_{\tau,n}: \mathcal{O}_{L_{\tau,n}}[G_n]_{\varphi}  \longrightarrow 
	\mathcal O_{L_{\tau}'}[G_n]
	\]
	be the map which is the composition of $\delta_{\tau,n}$ and
	$1 \otimes B_{\tau,n}$. Then $\beta_{\tau,n}$ is an
	isomorphism of $\mathcal O_{L_{\tau}'}[G_n]$-modules
	by Proposition \ref{prop:phi-iso}.
	
	With the above choices, the first diagram
	of the corollary commutes because 
	\[\xymatrix{
		\mathcal{O}_{L_{\tau,n+1}}
		\ar[rr]^-{B_{\tau,n+1}} \ar[d]_{\Tr_{L_{\tau,n+1}/L_{\tau,n}}} & & 
		\mathcal O_{L_{\tau}'}[\Gamma_{n+n_0+1}] \ar[d]^{\mathrm{pr}} \\
		\mathcal{O}_{L_{\tau,n}} \ar[rr]^-{B_{\tau,n}} & & 
		\mathcal O_{L_{\tau}'}[\Gamma_{n+n_0}]
	}\]
	commutes by construction, and the diagram
	\[\xymatrix{
		\mathcal{O}_{L_{\tau,n+1}}[G_{n+1}]_{\varphi}
		\ar[rr]^-{\delta_{\tau,n+1}} \ar[d]_{\mathrm{pr}} & & 
		\mathcal O_{L_{\tau}'}[G_{n+1}] \otimes_{\mathcal O_{L_{\tau}'}[\Gamma_{n+n_0+1}]} \mathcal{O}_{L_{\tau,n+1}} 
		\ar[d]^{\mathrm{pr} \otimes \Tr_{L_{\tau,n+1}/L_{\tau,n}}} \\
		\mathcal{O}_{L_{\tau,n}}[G_{n}]_{\varphi}
		\ar[rr]^-{\delta_{\tau,n}} & & 
		\mathcal O_{L_{\tau}'}[G_{n}] \otimes_{\mathcal O_{L_{\tau}'}[\Gamma_{n+n_0}]} \mathcal{O}_{L_{\tau,n}} 
	}\]
	also commutes, where the maps $\mathrm{pr}$ are induced by the natural
	projection maps $G_{n+1} \twoheadrightarrow G_n$
	and $\Gamma_{n+n_0+1} \twoheadrightarrow \Gamma_{n+n_0}$, respectively. 
	Finally, the second diagram commutes
	as we have $b_{\tau',n} = \hat{\tau}' \hat{\tau}^{-1}(b_{\tau,n})$
	again by construction.
 \end{proof}
 
 \begin{corollary}
 	Let $\tau \in \Sigma(K)$. Then $\Lambda_{\tau}(\mathcal{G})_{\varphi} :=
 	\varprojlim_n \mathcal{O}_{L_{\tau,n}}[G_{n}]_{\varphi}$ is a free
 	$\Lambda^{\mathcal{O}_{L_{\tau}'}}(\mathcal{G})$-module of rank $1$.
 \end{corollary}
 
 For each $n \in \N$ we have by Lemma \ref{lem:rho-image} that
 $\rho_n$ induces an injective map
 \[
	 \rho_{n}: \mathcal{O}_{L_n} \longrightarrow 
	 \bigoplus_{\tau \in \Sigma(K)} \mathcal{O}_{L_{\tau,n}}[G_n]_{\varphi}.
 \]
 Taking projective limits yields
 an embedding
 \[
	 \rho_{\infty}: \varprojlim_n \mathcal{O}_{L_n} \longrightarrow
	 \bigoplus_{\tau \in \Sigma(K)} \Lambda_{\tau}(\mathcal{G})_{\varphi}.
 \]
 As we have shown above, each choice of 
 $b \in \varprojlim_n \mathcal{O}_{\Q_{p,n}}$ as in Lemma
 \ref{lem:limits-IBG} defines an isomorphism of $\Lambda(\mathcal{G})$-modules
 \[
	 \beta_{\infty}: \bigoplus_{\tau \in \Sigma(K)} \Lambda_{\tau}(\mathcal{G})_{\varphi}
	 \stackrel{\simeq}{\longrightarrow} \bigoplus_{\tau \in \Sigma(K)}
	 \Lambda^{\mathcal{O}_{L_{\tau}'}}(\mathcal{G}).
 \]
 The composite map $\beta_{\infty} \circ \rho_{\infty}$ 
 induces an isomorphism
 of $\Q_p^c \otimes_{\Z_p} \Lambda(\mathcal{G})$-modules
 \begin{equation} \label{eqn:alpha_infty}
	 \alpha_{\infty}:
	 \Q_p^c \otimes_{\Z_p} \varprojlim_n  \mathcal{O}_{L_n}
	 \stackrel{\simeq}{\longrightarrow} \Q_p^c \otimes_{\Z_p} H_{L_{\infty}}.
 \end{equation}
 
 To see this, it suffices to note that $\alpha_{\infty}$ is the
 composite of the following $\Q_p^c \otimes_{\Z_p} \Lambda(\Gamma)$-module
 isomorphisms:
 \[
	\Q_p^c \otimes_{\Z_p} \varprojlim_n \mathcal{O}_{L_n}
	\simeq \Q_p^c \otimes_{\Z_p} \Lambda^{\mathcal{O}_{L'}} (\Gamma)
	\simeq \bigoplus_{\sigma' \in \Sigma(L')} 
		\Q_p^c \otimes_{\Z_p} \Lambda(\Gamma)
	\simeq \Q_p^c \otimes_{\Z_p} H_{L'_{\infty}}
	= \Q_p^c \otimes_{\Z_p} H_{L_{\infty}}.
 \] 
 
 The map $\alpha_{\infty}$ depends on the choices of 
 $b \in \varprojlim_n \mathcal{O}_{\Q_{p,n}}$
 and of the lifts $\hat{\tau}$
 of $\tau \in \Sigma(K)$. 
 
 \begin{lemma} \label{lem:dependence-b-tau}
 	Let $\tilde b \in \varprojlim_n \mathcal{O}_{\Q_{p,n}}$ be a second choice
 	of system of normal integral basis generators as in Lemma
 	\ref{lem:limits-IBG}. 
 	Let $\hat{\tilde \tau}: L_{\infty}
 	\hookrightarrow \Q_p^c$ be lifts of $\tau \in \Sigma(K)$. 	
 	These choices define an isomorphism of 
 	$\Q_p^c \otimes_{\Z_p} \Lambda(\mathcal{G})$-modules
 	$\tilde\alpha_{\infty}$ as in \eqref{eqn:alpha_infty} above.
 	Then 
 	\[
	 	\left[\Q_p^c \otimes_{\Z_p} H_{L_{\infty}}, \tilde\alpha_{\infty} \circ
	 	\alpha_{\infty}^{-1} \right] \in 
	 	K_1(\Q_p^c\otimes_{\Z_p} \Lambda(\mathcal{G}))
 	\]
 	maps to zero in $K_0(\Lambda(\mathcal{G}),\Q_p^c\otimes_{\Z_p} \Lambda(\mathcal{G}))$.
 \end{lemma}
 
 \begin{proof}
 	Fix an integer $n \geq 0$ and let $\tau \in \Sigma(K)$.
 	The inverse of the isomorphism $\beta_{\tau,n}$ is given by 
 	the $\mathcal O_{L_{\tau}'}[G_n]$-linear map
 	\begin{eqnarray*}
 		\gamma_{\tau,n}: \mathcal O_{L_{\tau}'}[G_n] & \longrightarrow &
 		\mathcal{O}_{L_{\tau,n}}[G_n]_{\varphi}\\
 		1 & \mapsto & \sum_{i=0}^{p^{n+n_0}-1} \phi^{-i}(b_{\tau,n+n_0}) \phi_n^i,
 	\end{eqnarray*}
 	where we have used the notation of the 
 	proof of Corollary \ref{cor:phi-limits}.
 	Then we have for sufficiently large $m \geq n$ that
 	\[
	 	\gamma_{\tau,n}(1) 
	 	\in \mathcal{O}_{\Q_{p,m}}[\Gamma_{n + n_0}]^{\times}
 	\]
 	by \cite[Proposition 4.3, p.\ 30]{MR717033}.
 	Let $\omega \in G_{\Q_p}$ be arbitrary. We may write
 	$\omega \hat\tau = \hat{\tau}' \omega_{\tau}$ for some
 	$\tau' \in \Sigma(K)$ and some
 	$\omega_{\tau} \in G_K$. We then have an equality
 	\[
	 	\omega(\gamma_{\tau,n}(1)) = 
	 	\gamma_{\tau',n}(1) \phi_n^{z_{\tau}(\omega)},
 	\]
 	where $z_{\tau}(\omega) \in \Z$ is an integer such that
 	$\omega_{\tau}|_{K_{n+n_0}} = \phi_K^{z_{\tau}(\omega)}|_{K_{n+n_0}}$.
 	We let $z(\omega) := \sum_{\tau \in \Sigma(K)} z_{\tau}(\omega)$ so that
 	$\Ver_{K/{\Q_p}}(\omega) |_{K_{n+n_0}} = \phi_K^{z(\omega)} |_{K_{n+n_0}}$.
 	We define
 	\begin{equation} \label{eqn:nu_n}
	 	\nu_n := \prod_{\tau \in \Sigma(K)} \gamma_{\tau,n}(1)
	 	\in \mathcal{O}_{\Q_{p,m}}[\Gamma_{n + n_0}]^{\times}
 	\end{equation}
 	so that we have an equality
 	\begin{equation} \label{eqn:Galois-on-nu}
 	\omega(\nu_n) = \nu_n \cdot \phi_n^{z(\omega)}.
 	\end{equation} 
 	Replacing $b$ by $\tilde b$ and $\hat \tau$ by
 	$\hat{\tilde \tau}$ for each $\tau \in \Sigma(K)$
 	we obtain in a similar way for each $n \geq 0$ an element
 	$\tilde \nu_n \in \mathcal{O}_{\Q_{p,m}}[\Gamma_{n + n_0}]^{\times}$
 	such that \eqref{eqn:Galois-on-nu} holds with $\nu_n$ replaced by
 	$\tilde \nu_n$.
 	It now follows from \eqref{eqn:Galois-on-nu} that
 	$\tilde \nu_n^{-1} \cdot \nu_n$ is invariant under the action of $G_{\Q_p}$ 
 	and thus
 	\begin{equation} \label{eqn:different-nus}
	 	\tilde \nu_n^{-1} \cdot \nu_n \in
	 	\Z_p[\Gamma_{n + n_0}]^{\times}
 	\end{equation}
 	for each $n$. It follows that 
 	\[
	 	\varprojlim_n (\tilde \nu_n^{-1} \cdot \nu_n) \in
	 	\Lambda(\Gamma)^{\times} \subseteq \Lambda(\mathcal{G})^{\times}
 	\]
 	is a pre-image of $\left[\Q_p^c \otimes_{\Z_p} H_{L_{\infty}},
 	\tilde\alpha_{\infty} \circ \alpha_{\infty}^{-1} \right]$ under
 	the composite map
 	\[
 	\Lambda(\mathcal{G})^{\times} \longrightarrow K_1(\Lambda(\mathcal{G}))
 	\longrightarrow K_1(\Q_p^c \otimes_{\Z_p} \Lambda(\mathcal{G})).
 	\]
 	The long exact sequence of relative $K$-theory now implies the claim.
 \end{proof}
 
 \subsection{The unramified term} \label{subsec:unramified-term}
 
 Let us denote the ring of integers in the maximal tamely ramified extension
 of $\Q_p$ in $\Q_p^c$ by $\mathcal{O}_p^t$. For a finite group $G$
 we let $\iota$ be the scalar extension map
 $K_0(\Z_p[G], \Q_p^c) \rightarrow K_0(\mathcal{O}_p^t[G], \Q_p^c)$.
 
 Now let $L/K$ be a Galois extension of $p$-adic fields with Galois group $G$.
 Then by \cite[Proposition 2.12]{MR2078894} there exists a unique
 $U_{L/K} \in K_0(\Z_p[G], \Q_p^c)$ satisfying the following two properties
 ($U_{L/K}$ is called the \emph{unramified term} attached to $L/K$).
 \begin{enumerate}
 	\item 
 	$\iota(U_{L/K}) = 0$.
 	\item
 	If $u = (u_{\chi})_{\chi \in \Irr(G)} \in 
 	\prod_{\chi \in \Irr(G)} (\Q_p^c)^{\times}$ is any pre-image of $U_{L/K}$
 	under $\partial_p \circ \Nrd_{\Q_p^c[G]}^{-1}$, then
 	$\omega(u_{\omega^{-1} \circ \chi}) = u_{\chi} 
 	\det_{\ind^{\Q_p}_K \chi}(\omega^{\ur})$ for every $\omega \in G_{\Q_p}$.
 \end{enumerate}
  
 Now recall the setting and notation of subsection \ref{subsec:embeddings}.
 We again assume that $\phi \in \Gamma$ maps to $\phi_K$ under $\mathcal{G}
 \twoheadrightarrow \Gamma_K$.
 We let $u_n$ be a pre-image of $U_{L_n/K}$ as in (ii) above, which by (i)
 actually belongs to $\Nrd_{\Q_p^c[G_n]}(\mathcal{O}_p^t[G_n]^{\times})$.
 Recall the definition \eqref{eqn:nu_n} of $\nu_n$.
 We define
 \[
	 u_n' := u_n \cdot \Nrd_{\Q_p^c[G_n]}(\nu_n)^{-1} \in
	 \Nrd_{\Q_p^c[G_n]}(\mathcal{O}_p^t[G_n]^{\times}).
 \]
 We write $u_n' = (u_{n,\chi}')_{\chi \in \Irr(G_n)}$
 and let $\omega \in G_{\Q_p}$. Then (ii) above,
 the Galois action \eqref{eqn:Galois-on-nu} on $\nu_n$ 
 and \eqref{eqn:Det(g)} imply the first equality of
 \begin{eqnarray*}
	 \omega(u_{n,\omega^{-1} \circ \chi}') & = & u_{n,\chi}' 
     \det_{\ind^{\Q_p}_K \chi}(\omega^{\ur})
     \det_{\chi}(\phi_n^{z(\omega)})^{-1}\\
     & = & u_{n,\chi}' \epsilon_{K/\Q_p}(\omega^{\ur})^{\chi(1)}
     \det_{\chi}(\Ver_{K/\Q_p}(\omega^{\ur})
     \phi_n^{-z(\omega)}).
 \end{eqnarray*}
 The second equality is \eqref{eqn:Ver-chi-relation}.
 It is clear from the definition of $z(\omega)$ that the map
 $G_{\Q_p} \rightarrow \Gamma_{n+n_0}'$, $\omega \mapsto
 \phi_n^{z(\omega)}$ is actually a group homomorphism that only depends
 upon $\omega^{\ur}$. Moreover, the restriction of
 $\Ver_{K/\Q_p}(\omega^{\ur}) \phi_n^{-z(\omega)}$ to $K_{n+n_0}$
 is trivial. It follows that there is a positive integer $k$,
 independent of $n$, such that $u_n'$ is invariant under
 $\omega^k$ for all $\omega \in G_{\Q_p}$ (for instance we may
 take $k = 2|H|$). The following result is now implied by
 Theorem \ref{thm:fixed-points}.
 
 \begin{lemma} \label{lem:un'}
 	There is a finite unramified extension $F$ of $\Q_p$ such that
 	$u_n'$ belongs to $\Nrd_{F[G_n]}(\mathcal{O}_F[G_n]^{\times})$
 	for all $n$.
 \end{lemma}
 
 We now define a variant of the unramified term by
 \[
	 U_{L_n/K}' := \partial_p \circ \Nrd_{\Q_p^c[G]}^{-1}(u_n')
	 \in K_0(\Z_p[G_n], \Q_p^c)
 \]
 and note that this only depends upon $L_n/K$ 
 by \eqref{eqn:different-nus}. 
 It follows from \cite[Lemma 2.13]{MR2078894} and the definition of
 $\nu_n$ that we have $\quot^{G_{n+1}}_{G_n}(U_{L_{n+1}/K}')
 = U_{L_n/K}'$ for all $n$.
 Moreover, $U_{L_n/K}'$ maps to zero 
 in $K_0(\mathcal O_F[G_n], \Q_p^c)$ by Lemma \ref{lem:un'}
 and thus has a pre-image $\widehat U_{L_n/K}'$ 
 in $K_0(\Z_p[G_n], \mathcal{O}_F[G_n])$ by Lemma \ref{lem:SK1}.
 As $SK_1(\mathcal{O}_F[G_n])$ is finite for all $n$,
 we may choose these pre-images such that
 $\quot^{G_{n+1}}_{G_n}(\widehat U_{L_{n+1}/K}')
 = \widehat U_{L_n/K}'$ for all $n$. Via Proposition
 \ref{prop:limit-relative-K0} we now define
 \[
	 \widehat U_{L_{\infty}/K}' := \varprojlim_n \widehat U_{L_n/K}'
	 \in K_0(\Lambda(\mathcal{G}), \Lambda^{\mathcal{O}_F}(\mathcal{G}))
 \]
 which is well-defined up to an element in the image of
 $SK_1(\Lambda^{\mathcal{O}_F}(\mathcal{G}))$.
 We let $U_{L_{\infty}/K}'$ be the image of $\widehat U_{L_{\infty}/K}'$
 under the natural map
 $K_0(\Lambda(\mathcal{G}), \Lambda^{\mathcal{O}_F}(\mathcal{G}))
 \rightarrow K_0(\Lambda(\mathcal{G}), \Q_p^c \otimes_{\Z_p} \Lambda(\mathcal{G}))$.
 The following is now an immediate consequence of Lemma \ref{lem:Nrd-SK1}.
 
 \begin{lemma} \label{lem:unramified-term-wdef}
 	The element $U_{L_{\infty}/K}' \in K_0(\Lambda(\mathcal{G}), \Q_p^c \otimes_{\Z_p} \Lambda(\mathcal{G}))$ is well-defined up to
 	the image of
 	an element $x \in K_1(\Q_p^c \otimes_{\Z_p} \Lambda(\mathcal{G}))$
 	such that $\Nrd_{\mathcal Q^c(\mathcal{G})}(x) = 1$.
 \end{lemma}
 
 \begin{remark}
 	As follows from Remark \ref{rmk:Nrd-injective}, the map
 	$\Nrd_{\mathcal Q^c(\mathcal{G})}$ conjecturally is injective
 	on $K_1(\mathcal Q^c(\mathcal{G}))$. We have already observed that
 	the natural map 
 	$K_1(\Q_p^c \otimes_{\Z_p} \Lambda(\mathcal{G})) \rightarrow
 	K_1(\mathcal{Q}^c(\mathcal{G}))$ is injective.
 	Thus the element $U_{L_{\infty}/K}'$ is at least conjecturally
 	well-defined.
 \end{remark}
 
 \subsection{Definition of the cohomological term}
 By Proposition \ref{prop:log-limits} and \eqref{eqn:alpha_infty}
 the composite map $\phi_{\infty} := \beta_{\infty} \circ \rho_{\infty}
 \circ \log_{\infty}$ induces and isomorphism of 
 $\Q_p^c \otimes_{\Z_p} \Lambda(\mathcal{G})$-modules
 \begin{equation} \label{eqn:phi-infty}
	 \phi_{\infty}: \Q_p^c \otimes_{\Z_p} U^1(L_{\infty})
	 \stackrel{\simeq}{\longrightarrow} \Q_p^c \otimes_{\Z_p} H_{L_{\infty}}.
 \end{equation}
 By Corollary \ref{cor:coh-at-infty-II} it likewise induces an isomorphism
 of $\mathcal{Q}^c(\mathcal{G})$-modules
 \[
	 \phi_{\infty}: \mathcal{Q}^c(\mathcal{G}) \otimes_{\Lambda(\mathcal{G})}
	 H^0(K^{\bullet}_{L_{\infty}})
	 \stackrel{\simeq}{\longrightarrow} 
	 \mathcal{Q}^c(\mathcal{G}) \otimes_{\Lambda(\mathcal{G})}
	 H^1(K^{\bullet}_{L_{\infty}})
 \]
 
 \begin{definition}
 	Let $L/K$ be a finite Galois extension of $p$-adic fields.
 	Let $L_{\infty}$ be the unramified $\Z_p$-extension of $L$ and let
 	$\mathcal{G} = \Gal(L_{\infty} / K)$. We define the \emph{cohomological term}
 	attached to the extension $L_{\infty}/K$ to be
 	\[
	 	C_{L_{\infty}/K} := 
	 	\chi_{\Lambda(\mathcal{G}), \mathcal{Q}^c(\mathcal{G})}(K_{L_{\infty}}^{\bullet}, \phi_{\infty}^{-1})
	 	\in K_0(\Lambda(\mathcal{G}), \mathcal{Q}^c(\mathcal{G})).
 	\]
 \end{definition}
 
 Now Lemmas \ref{lem:dependence-b-tau} and \ref{lem:unramified-term-wdef}
 and the definition of $U_{L_{\infty}/K}'$ imply the following.
 
 \begin{lemma} \label{lem:well-defined}
 	The element
 	\[
	 	C_{L_{\infty}/K} + U_{L_{\infty}/K}' \in K_0(\Lambda(\mathcal{G}), \mathcal{Q}^c(\mathcal{G}))
 	\]
 	is well-defined up to the image of
 	an element $x \in K_1(\Q_p^c \otimes_{\Z_p} \Lambda(\mathcal{G}))$
 	such that $\Nrd_{\mathcal Q^c(\mathcal{G})}(x) = 1$.
 \end{lemma}
 
 The following result should be compared with
 Theorem \ref{thm:GGs-in-Z_p} (iii).
 
 \begin{theorem} \label{thm:Galois-on-coh-term}
 	Let $\xi \in K_1(\mathcal{Q}^c(\mathcal{G}))$ be any pre-image
 	of $C_{L_{\infty}/K} + U_{L_{\infty}/K}'$. Then for every
 	$\omega \in G_{\Q_p}$ we have
 	\[
	 	\omega \left(\Det(\xi)(\omega^{-1} \circ \chi)\right) = 
	 	\Det(\xi)(\chi) \cdot \det_{\ind_K^{\Q_p}(\chi)}(\omega^{\ram})^{-1}.
 	\]
 \end{theorem}
 
 \begin{proof}
	Let us choose an
	$a = (a_n)_n \in \varprojlim_n \mathcal{O}_{L_n}$
	as in Proposition \ref{prop:normal-bases-tower}.
	Then each $a_n$
	generates a normal basis for $L_n/K$. Let $\mathcal{L} \subseteq
	\varprojlim_n \mathcal{O}_{L_n}$ be the full 
	$\Lambda^{\mathcal{O}_K}(\mathcal{G})$-sublattice generated by $a$.
	We put $X := \log_{\infty}^{-1}(\mathcal{L})$ which is contained
	in $\Q_p \otimes_{\Z_p} U^1(L_{\infty})$ by Proposition
	\ref{prop:log-limits}. We may and do choose an $a$ such 
	that $X$ is actually a $\Lambda(\mathcal{G})$-submodule of 
	$U^1(L_{\infty})$. We consider the map of complexes
	\[
		\lambda: X \oplus H_{L_{\infty}}[-1] \longrightarrow
		K_{L_{\infty}}^{\bullet}
	\]
	which on cohomology induces the natural embeddings
	$X \hookrightarrow U^1(L_{\infty}) \simeq H^0(K_{L_{\infty}}^{\bullet})$
	and $H_{L_{\infty}} \hookrightarrow \Z_p \oplus H_{L_{\infty}}
	\simeq H^1(K_{L_{\infty}}^{\bullet})$.
	The cone of $\lambda$ is a perfect complex of
	$\Lambda(\mathcal{G})$-modules whose cohomology groups
	are torsion as $\Lambda(\Gamma)$-modules. It therefore defines an
	element
	\[
		\chi_{\Lambda(\mathcal{G}), \mathcal{Q}(\mathcal{G})}
		(\Cone(\lambda), 0)
		\in K_0(\Lambda(\mathcal{G}), \mathcal{Q}(\mathcal{G}))
	\]
	which has a pre-image in $K_1(\mathcal{Q}(\mathcal{G}))$.
	As the image of $K_1(\mathcal{Q}(\mathcal{G}))$ under $\Det$
	is Galois-invariant and we have equalities
	\begin{eqnarray}
		C_{L_{\infty}/K} & = & [H_{L_{\infty}}, \phi_{\infty}^{-1}, X]
		+ \chi_{\Lambda(\mathcal{G}), \mathcal{Q}(\mathcal{G})}
		(\Cone(\lambda), 0) \nonumber \\
		& = & - [\mathcal{L}, \alpha_{\infty}, H_{L_{\infty}}]
		+ \chi_{\Lambda(\mathcal{G}), \mathcal{Q}(\mathcal{G})}
		(\Cone(\lambda), 0), \label{eqn:coh-term-description}
	\end{eqnarray}
	we may replace $\xi$ by a pre-image $\xi'$ of
	\[
		U'_{L_{\infty}/K} - [\mathcal{L}, \alpha_{\infty}, H_{L_{\infty}}]
		\in K_0(\Lambda(\mathcal{G}), \Q_p^c \otimes_{\Z_p} \Lambda(\mathcal{G})).
	\]
	Then $\Det(\xi')$ belongs to $\Hom^W(R_p(\mathcal{G}), (\Q_p^c 
	\otimes_{\Z_p} \Lambda(\Gamma_K))^{\times})$ and thus
	by Lemma \ref{lem:Det-criterion} it suffices to show that
	\[
	\omega \left(\aug_{\Gamma_K}(\Det(\xi')(\omega^{-1} \circ \chi))\right) = 
	\aug_{\Gamma_K}(\Det(\xi')(\chi)) \cdot \det_{\ind_K^{\Q_p}(\chi)}(\omega^{\ram})^{-1}
	\]
	for all $\chi \in \Irr_{\Q_p^c}(\mathcal{G})$.
	As in Example \ref{ex:Det(g)}, one can now use \cite[(8)]{MR2609173}
	to deduce this from the following results on finite level.
	Fix a character $\chi \in \Irr_{\Q_p^c}(\mathcal{G})$ and choose
	a sufficiently large $n$ such that $\chi$ factors through $G_n$.
	By Proposition \ref{prop:normal-bases-tower} the embedding
	$\mathcal{L} \hookrightarrow \varprojlim_n \mathcal{O}_{L_n}$
	yields an embedding
	\[
		\mathcal{L}_n := \mathcal{L}_{\Gamma_n} \hookrightarrow
		\mathcal{O}_{L_n}
	\]
	by taking $\Gamma_n$-coinvariants. Choose 
	$\xi_n'  = (\xi'_{n,\psi})_{\psi}\in 
	\prod_{\psi \in \Irr(G_n)}(\Q_p^c)^{\times}$
	such that
	\begin{eqnarray*}
		\partial_p(\Nrd_{\Q_p^c[G_n]}^{-1}(\xi_n')) & = &
			U_{L_n/K}' - [\mathcal{L}_n, \alpha_n, H_{L_n}] \\
			& = & U_{L_n/K} - [\mathcal{L}_n, \rho_n, H_{L_n}]
	\end{eqnarray*}	
	where $\alpha_n$ is induced by $\alpha_{\infty}$.
	Now the Galois-action on a pre-image of the unramified term
	(see \S \ref{subsec:unramified-term} (ii)) and
	\cite[Lemma 2.8]{MR2078894} imply that
	\[
		\omega(\xi'_{n, \omega^{-1} \circ \psi}) = \xi'_{n, \psi} \cdot
		\det_{\ind_K^{\Q_p}(\psi)}(\omega^{\ram})^{-1}
	\]
	for every $\omega \in G_{\Q_p}$ and every $\psi \in \Irr(G_n)$.
	Taking $\psi = \chi$ completes the proof.
 \end{proof}
 
 \subsection{The correction term}
 Let $I$ be the inertia subgroup of $\mathcal{G}$ and let
 $\sigma \in \mathcal{G}$ be a lift of the Frobenius automorphism
 in $\mathcal{G}/I$. Then $I$ is a finite normal subgroup of $\mathcal{G}$
 so that $e_I := |I|^{-1} \sum_{i \in I} i$ is a central idempotent in
 $\mathcal{Q}(\mathcal{G})$. Recall the notation introduced at the end
 of \S \ref{subsec:K-theory} and denote the cardinality of the residue 
 field of $K$ by $q_K$. We let $ m_{L_{\infty}/K}
 \in K_1(\mathcal{Q}(\mathcal{G}))$ be the image of
 \[
 \frac{^{\ast}((1 - \sigma q_K^{-1}) e_I)}
 {^{\ast}((1 - \sigma^{-1}) e_I)} \in
 \zeta(\mathcal{Q}(\mathcal{G}))^{\times}
 \]
 under the canonical maps $\zeta(\mathcal{Q}(\mathcal{G}))^{\times}
 \hookrightarrow \mathcal{Q}(\mathcal{G})^{\times} \rightarrow
 K_1(\mathcal{Q}(\mathcal{G}))$. We put
 \[
 M_{L_{\infty}/K} := 
 \partial_{\Lambda(\mathcal{G}), \mathcal{Q}(\mathcal{G})}(m_{L_{\infty}/K}) \in
 K_0(\Lambda(\mathcal{G}), \mathcal{Q}(\mathcal{G}))
 \]
 and call $M_{L_{\infty}/K}$ the \emph{correction term}.
 
 \subsection{Functorialities}
 Let $N$ be a finite normal subgroup of $\mathcal{G}$ and let $\mathcal{H}$ be an open subgroup of $\mathcal{G}$.
 There are canonical maps
 \begin{eqnarray*}
 	\quot^{\mathcal{G}}_{\mathcal{G}/N}: & K_{0}(\Lambda(\mathcal{G}), \mathcal{Q}^c(\mathcal{G})) \longrightarrow &
 	K_{0}(\Lambda(\mathcal{G}/N), \mathcal{Q}^c(\mathcal{G}/N)),\\
 	\res^{\mathcal{G}}_{\mathcal{H}}: & K_{0}(\Lambda(\mathcal{G}), \mathcal{Q}^c(\mathcal{G})) \longrightarrow &
 	K_{0}(\Lambda(\mathcal{H}), \mathcal{Q}^c(\mathcal{H}))
 \end{eqnarray*}
 induced from scalar extension along $\Lambda(\mathcal{G}) \longrightarrow \Lambda(\mathcal{G}/N)$ and restriction of scalars
 along $\Lambda(\mathcal{H}) \hookrightarrow \Lambda(\mathcal{G})$.
 Likewise, there are restriction and quotient maps between the
 respective $K_1$-groups.
 
 \begin{prop} \label{prop:functoriality-coh}
	The following statements hold.
 	\begin{enumerate}
 		\item 
 		Let $N$ be a finite normal subgroup of $\mathcal{G}$ and put
 		$L_{\infty}' := L_{\infty}^N$. Then we have
 		\[
 		\quot^{\mathcal{G}}_{\mathcal{G}/N}(M_{L_{\infty}/K})
 		= M_{L_{\infty}'/K},
 		\]
 		and up to an element
 		$x \in K_1(\Q_p^c \otimes_{\Z_p} \Lambda(\mathcal{G}/N))$
 		such that $\Nrd_{\mathcal{Q}^c(\mathcal{G}/N)}(x) = 1$
 		we have an equality
 		\[
	 		\quot^{\mathcal{G}}_{\mathcal{G}/N}(C_{L_{\infty}/K} + 
	 		U_{L_{\infty}/K}') = 
	 		C_{L_{\infty}'/K} + U_{L_{\infty}'/K}'.
 		\]
 		\item
 		Let $\mathcal{H}$ be an open subgroup of $\mathcal{G}$
 		and put $K' := L_{\infty}^{\mathcal{H}}$.
 		Then we have
 		\[
 		\res^{\mathcal{G}}_{\mathcal{H}}(M_{L_{\infty}/K})
 		= M_{L_{\infty}/K'},
 		\]
 		and up to an element
 		$x \in K_1(\Q_p^c \otimes_{\Z_p} \Lambda(\mathcal{H}))$
 		such that $\Nrd_{\mathcal{Q}^c(\mathcal{H})}(x) = 1$
 		we have an equality
 		\[
 		\res^{\mathcal{G}}_{\mathcal{H}}(C_{L_{\infty}/K} + 
 		U_{L_{\infty}/K}') = 
 		C_{L_{\infty}/K'} + U_{L_{\infty}/K'}'.
 		\]
 	\end{enumerate}
 \end{prop}	
 
 \begin{proof}
 	A straightforward calculation shows that indeed
 	$\quot^{\mathcal{G}}_{\mathcal{G}/N}(m_{L_{\infty}/K})
 	= m_{L_{\infty}'/K}$ and
 	$\res^{\mathcal{G}}_{\mathcal{H}}(m_{L_{\infty}/K})
 	= m_{L_{\infty}/K'}$. For the sum of the cohomological
 	and the unramified term, the result follows as in
 	\cite[\S 2.4 and \S 2.5]{MR2078894}. The main ingredient is
 	that in case (i) we have an isomorphism
 	\[
	 	\Lambda(\mathcal{G}/N) \otimes^{\mathbb L}_{\Lambda(\mathcal{G})}
	 	K_{L_{\infty}}^{\bullet} \simeq K_{L_{\infty}'}^{\bullet}
 	\]
 	in $\mathcal{D}(\Lambda(\mathcal{G}/N))$ by
 	\cite[Proposition 1.6.5]{MR2276851}.
 	We leave the details to the reader.
 \end{proof}

 \section{The main conjecture} \label{sec:MC}
 
 \subsection{Statement of the main conjecture}
 As before let $L/K$ be a finite Galois extension of $p$-adic
 local fields and let $L_{\infty}$ be the unramified $\Z_p$-extension
 of $L$. Then $L_{\infty}/K$ is a one-dimensional $p$-adic Lie extension
 with Galois group $\mathcal{G}$. Choose an isomorphism $j: \C \simeq \C_p$.
 Lemma \ref{lem:well-defined} implies that the following conjecture is
 well-posed.
 
 \begin{conj} \label{conj:main-conjecture}
 	There exists  $\zeta^{(j)}_{L_{\infty}/K} \in 
 	K_1(\mathcal{Q}^c(\mathcal{G}))$ such that
	\[
		\partial_{\Lambda(\mathcal{G}), \mathcal{Q}^c(\mathcal{G})}(\zeta^{(j)}_{L_{\infty}/K}) = 
		-C_{L_{\infty}/K} - U'_{L_{\infty}/K} + M_{L_{\infty}/K}
	\]
	and
	\[
		\Det(\zeta^{(j)}_{L_{\infty}/K}) = \tau^{(j)}_{L_{\infty}/K}.
	\]
 \end{conj}
 
 The following observation is immediate from 
 Theorem \ref{thm:GGs-in-Z_p} (ii).
 
 \begin{lemma}
 	Conjecture \ref{conj:main-conjecture} does not depend on 
 	the choice of $j: \C \simeq \C_p$.
 \end{lemma}
 
 \begin{remark}
 	It follows from sequence \eqref{eqn:connecting-surj-c}
 	that $-C_{L_{\infty}/K} - U'_{L_{\infty}/K} + M_{L_{\infty}/K}$
 	always has a pre-image in $K_1(\mathcal{Q}^c (\mathcal{G}))$.
 \end{remark}
 
 \begin{remark}
 	It is expected that the reduced norm $\Nrd_{\mathcal{Q}^c(\mathcal G)}:
 	K_1(\mathcal Q^c (\mathcal G)) \rightarrow 
 	\zeta(\mathcal Q^c(\mathcal G))^{\times}$ is injective.
 	If this is true, then the element $\zeta^{(j)}_{L_{\infty}/K} \in 
 	K_1(\mathcal{Q}^c(\mathcal{G}))$ is unique (if it exists).
 \end{remark}
 
 \begin{remark}
 	In subsequent work we will show that Conjecture \ref{conj:main-conjecture}
 	for $L_{\infty} / K$ implies the equivariant local
 	$\varepsilon$-constant conjecture of Breuning 
 	\cite[Conjecture 3.2]{MR2078894} and, more generally, for
 	unramified twists of $\Z_p(1)$.
 \end{remark}
 
 \begin{remark}
 	Suppose that $\mathcal{G}$ is abelian. Then $\Det$ induces an
 	isomorphism
 	\[
	 	\Det: K_1(\mathcal Q^c(\mathcal G)) \simeq 
	 	\Hom^{W}(R_p( \mathcal{G}), \mathcal{Q}^{c}(\Gamma_K)^{\times}).
 	\]
 	This follows from triangle \eqref{eqn:Det_triangle-c}
 	and Wedderburn's theorem (see the proof of \cite[Lemma 5a]{MR1935024}).
 	We therefore may define
 	\[
	 	T_{L_{\infty}/K} := \partial_{\Lambda(\mathcal{G}), \mathcal{Q}^c(\mathcal{G})}(\Det^{-1}(\tau^{(j)}_{L_{\infty}/K}))
	 	\in K_0(\Lambda(\mathcal{G}), \mathcal{Q}^c(\mathcal{G}))
 	\]
 	which indeed does not depend on $j$ by Theorem \ref{thm:GGs-in-Z_p} (ii).
 	We put
 	\[
	 	R_{L_{\infty}/K} := T_{L_{\infty}/K} + C_{L_{\infty}/K} + U'_{L_{\infty}/K}
	 	- M_{L_{\infty}/K}.
 	\]
 	Then Conjecture \ref{conj:main-conjecture} asserts that
 	$R_{L_{\infty}/K}$ vanishes,
 	and the analogy to Breuning's conjecture 
 	\cite[Conjecture 3.2]{MR2078894} becomes more apparent.
 \end{remark}
 
 \subsection{Functorialities}
 The following result is immediate from Propositions
 \ref{prop:functoriality-Gauss} and \ref{prop:functoriality-coh}.
 
 \begin{prop}
 	Suppose that Conjecture \ref{conj:main-conjecture} holds
 	for the extension $L_{\infty}/K$.
 	\begin{enumerate}
 		\item 
 		Let $N$ be a finite normal subgroup of $\mathcal{G}$
 		and put $L_{\infty}' := L_{\infty}^N$.
 		Then Conjecture \ref{conj:main-conjecture} holds
 		for the extension $L_{\infty}'/K$.
 		\item 
 		Let $\mathcal{H}$ be an open subgroup of $\mathcal{G}$
 		and put $K' := L_{\infty}^{\mathcal{H}}$.
 		Then Conjecture \ref{conj:main-conjecture} holds
 		for the extension $L_{\infty}/K'$.
 	\end{enumerate}
 \end{prop}
 
 \begin{remark}
 	Suppose that $\mathcal{G}$ is abelian. Then Propositions
 	\ref{prop:functoriality-Gauss} and \ref{prop:functoriality-coh}
 	indeed show that
 	\[
	 	\quot^{\mathcal{G}}_{\mathcal{G}/N}(R_{L_{\infty}/K})
	 	= R_{L_{\infty}'/K} \quad \mathrm{and} \quad
	 	\res^{\mathcal{G}}_{\mathcal{H}}(R_{L_{\infty}/K})
	 	= R_{L_{\infty}/K'}.
 	\]
 \end{remark}
 
 \subsection{First evidence}
  We first consider the Galois action on the occurring objects.
  By definition $m_{L_{\infty}/K} \in K_1(\mathcal{Q} (\mathcal{G}))$
  is a pre-image of $M_{L_{\infty}/K}$ and $\Det(m_{L_{\infty}/K})$
  is Galois-invariant. Now Theorem \ref{thm:GGs-in-Z_p} (iii)
  and Theorem \ref{thm:Galois-on-coh-term}
  imply the following analogue of \cite[Proposition 3.4]{MR2078894}.
  
  \begin{prop} \label{prop:mc-Galois}
  	For every $x_{L_{\infty}/K} \in K_1(\mathcal{Q}^c (\mathcal{G}))$
  	such that
  	\[
  	\partial_{\Lambda(\mathcal{G}), \mathcal{Q}^c(\mathcal{G})}(x_{L_{\infty}/K}) = 
  	-C_{L_{\infty}/K} - U'_{L_{\infty}/K} + M_{L_{\infty}/K}
  	\]
  	we have
  	\[
  	\Det(x_{L_{\infty}/K})^{-1} \cdot \tau^{(j)}_{L_{\infty}/K}
  	\in \Hom^{\ast}(R_p( \mathcal{G}), \mathcal{Q}^{c}(\Gamma_K)^{\times}).
  	\]
  \end{prop}
  
  We now prove the following strengthening of Proposition 
  \ref{prop:mc-Galois} which has no analogue at finite level.
  
  \begin{prop} \label{prop:mc-Galois+}
  	For every $x_{L_{\infty}/K} \in K_1(\mathcal{Q}^c (\mathcal{G}))$
  	such that
  	\[
  	\partial_{\Lambda(\mathcal{G}), \mathcal{Q}^c(\mathcal{G})}(x_{L_{\infty}/K}) = 
  	-C_{L_{\infty}/K} - U'_{L_{\infty}/K} + M_{L_{\infty}/K}
  	\]
  	we have
  	\[
  	\Det(x_{L_{\infty}/K})^{-1} \cdot \tau^{(j)}_{L_{\infty}/K}
  	\in \Hom^{\ast}(R_p( \mathcal{G}), (\Q_p^c \otimes_{\Z_p} \Lambda(\Gamma_K))^{\times}).
  	\]
  \end{prop}
  
  \begin{proof}
  	We know that $ \tau^{(j)}_{L_{\infty}/K}$ belongs to
  	$\Hom^{W}(R_p( \mathcal{G}), (\Q_p^c \otimes_{\Z_p} \Lambda(\Gamma_K))^{\times})$ by Theorem \ref{thm:GGs-in-Z_p} (i).
  	By Proposition \ref{prop:mc-Galois} it 
  	therefore suffices to show that
  	$x_{L_{\infty}/K}$ lies in the image of
  	$K_1(\Q_p^c \otimes_{\Z_p} \Lambda(\mathcal{G}))$. This is true
  	if and only if 
  	$-C_{L_{\infty}/K} - U'_{L_{\infty}/K} + M_{L_{\infty}/K}$
  	maps to zero under the canonical scalar extension map
  	\[
	  s_p^c: K_0(\Lambda(\mathcal{G}), \mathcal{Q}^c(\mathcal{G}))
	  	\longrightarrow
  	K_0(\Q_p^c \otimes_{\Z_p} \Lambda(\mathcal{G}), \mathcal{Q}^c(\mathcal{G})).
  	\]
  	As $U'_{L_{\infty}/K}$ lies in 
  	$K_0(\Lambda(\mathcal{G}), \Q_p^c \otimes_{\Z_p} \Lambda(\mathcal{G}))$,
  	we clearly have $s_p^c(U'_{L_{\infty}/K}) = 0$.
  	The following computation then finishes the proof:
  	\begin{eqnarray*}
  	s_p^c(C_{L_{\infty}/K}) & = &
	  	\chi_{\Q_p^c \otimes_{\Z_p} \Lambda(\mathcal{G}),
	  		\mathcal{Q}^c(\mathcal{G})}\left(\Q_p^c \otimes_{\Z_p} U^1(L_{\infty}) \oplus \Q_p^c \otimes_{\Z_p} H_{L_{\infty}}[-1] 
	  	\oplus \Q_p^c[-1], \phi_{\infty}^{-1}\right) \\
	  	& = & \chi_{\Q_p^c \otimes_{\Z_p} \Lambda(\mathcal{G}),
	  		\mathcal{Q}^c(\mathcal{G})}(\Q_p^c[-1], 0)\\
	  	& = & - \partial_{\Q_p^c \otimes_{\Z_p} \Lambda(\mathcal{G}),
	  		\mathcal{Q}^c(\mathcal{G})}(^{\ast}((1-\sigma^{-1})e_I))\\
	  	& = & s_p^c(M_{L_{\infty}/K}).
  	\end{eqnarray*}
  	Here, the first equality follows from the definition of the 
  	cohomological term and the fact that the cohomology of
  	$\Q_p^c \otimes^{\mathbb L}_{\Z_p} C_{L_{\infty}/K}$ is perfect.
  	The second equality is a consequence of \eqref{eqn:phi-infty}. The third
  	equality results from the short exact sequence
  	\[
  	0 \rightarrow \Q_p^c \otimes_{\Z_p} \Lambda(\mathcal{G})
  	\rightarrow \Q_p^c \otimes_{\Z_p} \Lambda(\mathcal{G})
  	\rightarrow \Q_p^c \rightarrow 0,
  	\]
  	where the second arrow is multiplication by 
  	$^{\ast}((1 - \sigma^{-1})e_I)$. The last equality holds as
  	$^{\ast}((1 - \sigma q_K^{-1}) e_I)$ belongs to
  	$\zeta(\Q_p \otimes_{\Z_p} \Lambda(\mathcal{G}))^{\times}$.
  \end{proof}
  
  \begin{remark}
  	Suppose that $\mathcal{G}$ is abelian. Then Proposition 
  	\ref{prop:mc-Galois} asserts that
  	\[
	  	R_{L_{\infty}/K} \in 
	  	K_0(\Lambda(\mathcal{G}), \mathcal{Q}(\mathcal{G})),
  	\]
  	whereas Proposition \ref{prop:mc-Galois+} in fact shows that
  	\[
  	R_{L_{\infty}/K} \in 
  	K_0(\Lambda(\mathcal{G}), \Q_p \otimes_{\Z_p} \Lambda(\mathcal{G})).
  	\]
  \end{remark}
  
  \section{The maximal order case} \label{sec:max-orders}
  
  \subsection{Principal units}
  Let $m \geq 1$ be an integer. As $L_{\infty}$ is the unramified
  $\Z_p$-extension of $L$, we may define
  $U^m(L_{\infty}) := \varprojlim_n U^m_{L_n}$ where the transition maps
  are given by the norm maps. For any integer $m$ we likewise define
  $\mathcal{P}_m := \varprojlim_n \mathfrak p^m_{L_n}$ 
  where the transition maps are given by the trace maps.
  Note that in particular $\mathcal{P}_0 = \varprojlim_n \mathcal{O}_{L_n}$.
  
  \begin{prop} \label{prop:principal-units}
  	Let $L/K$ be at most tamely ramified and let $m$ be an integer. 
  	\begin{enumerate}
  		\item 
  		Then $\mathcal{P}_m$ is a free 
  		$\Lambda^{\mathcal{O}_K}(\mathcal{G})$-module of rank $1$.
  		\item
  		For $m \geq 1$
  		the $\Lambda(\mathcal{G})$-module $U^m(L_{\infty})$ is of
  		projective dimension at most $1$. 
  		\item
  		If $m$ is sufficiently large, then $U^m(L_{\infty})$ is
  		indeed a free $\Lambda(\mathcal{G})$-module of rank $[K: \Q_p]$.
  	\end{enumerate}
  \end{prop}
  
  \begin{proof}
  	For sufficiently large $m$ the $p$-adic logarithm induces isomorphisms
  	of $\Z_p[G_n]$-modules
  	$U^m_{L_n} \simeq \mathfrak p_{L_n}^m$ for all $n \geq 0$. As
  	$L_n/K$ is tamely ramified, the ideal $\mathfrak p_{L_n}^m$ is a free
  	$\mathcal O_K[G_n]$-module of rank $1$ for every integer $m$.
  	 Since the transition maps are surjective, we obtain (i) and (iii).
  	 For $m \geq 1$ we consider the exact sequences
  	 \begin{equation} \label{eqn:principal-units-ses}
	  	 0 \longrightarrow U^{m+1}(L_{\infty}) \longrightarrow
	  	 U^m(L_{\infty}) \longrightarrow \mathcal{P}_m / \mathcal{P}_{m+1}
	  	 \longrightarrow 0.
  	 \end{equation}
  	 Now (i) and (iii) imply (ii) by downwards induction.
  \end{proof}
  
  \begin{remark}
  	Let $I$ be the inertia subgroup of $\mathcal{G}$. If $L/K$ is tamely
  	ramified, then $p$ does not divide $|I|$ by definition.
	Note that $I$ is actually a subgroup of $H$ and that 
	$\mathcal{G} / I \simeq H/I \times \Gamma$ is abelian.
	Since $\Gamma$ is the Galois group of the maximal unramified
	pro-$p$-extension of $K$, we see that $p$ does actually not divide
	$|H|$. Thus $\Lambda(\mathcal{G})$ is a maximal $R$-order in
	$\mathcal{Q}(\mathcal{G})$, where we recall from
	\S \ref{subsec:Iwasawa-algebras} that $R = \Lambda(\Gamma_0)$
	for some central subgroup $\Gamma_0 \simeq \Z_p$ of $\mathcal{G}$
	(this can be deduced from either of \cite[Theorem 3.5]{MR3247796}
	or \cite[Proposition 3.7 or Theorem 3.12]{hybrid-EIMC}).
	If we assume in addition that $\mathcal{G} \simeq H \times \Gamma$,
	then \cite[Theorem 11.2.4(iii) and Proposition 11.2.1]{MR2392026}
	show that $U^1(L_{\infty})$ is a $\Lambda(\mathcal{G})$-module
	of projective dimension at most $1$. This gives an alternative proof
	of Proposition \ref{prop:principal-units} (ii) in a special case.
  \end{remark}
  
  \subsection{Tamely ramified extensions}
  We now prove the main conjecture for tamely ramified extensions.
  
  \begin{theorem} \label{thm:mc-tame}
  	Let $L/K$ be a tamely ramified Galois extension of $p$-adic local fields.
  	Then Conjecture \ref{conj:main-conjecture} holds for $L_{\infty}/K$.
  \end{theorem}
  
  \begin{proof}
  	Let $m \geq 1$ be an integer.
  	We consider the following maps of complexes
  	\[
  	\lambda_m: U^m(L_{\infty}) \oplus H_{L_{\infty}}[-1] \longrightarrow
  	K_{L_{\infty}}^{\bullet}
  	\]
  	which on cohomology induces the natural embeddings
  	$U^m(L_{\infty}) \hookrightarrow U^1(L_{\infty}) \simeq H^0(K_{L_{\infty}}^{\bullet})$
  	and $H_{L_{\infty}} \hookrightarrow \Z_p \oplus H_{L_{\infty}}
  	\simeq H^1(K_{L_{\infty}}^{\bullet})$. If $m$ is sufficiently large,
  	the $p$-adic logarithm induces an isomorphism 
  	of $\Lambda(\mathcal{G})$-modules
  	$U^m(L_{\infty}) \simeq \mathcal{P}_m$. We now apply
  	\eqref{eqn:coh-term-description} with $\mathcal{L} = \mathcal{P}_m$
  	and obtain the first equality in
  \begin{eqnarray*}
	  C_{L_{\infty/K}} & = & \chi_{\Lambda(\mathcal{G}), 
		  \mathcal{Q}(\mathcal{G})}(\Cone(\lambda_m), 0)
		  - [\mathcal{P}_m, \alpha_{\infty}, H_{L_{\infty}}] \\
		  & = & \chi_{\Lambda(\mathcal{G}), 
		  	\mathcal{Q}(\mathcal{G})}(\Cone(\lambda_1), 0)
		  - [\mathcal{P}_1, \alpha_{\infty}, H_{L_{\infty}}] \\
		  & = & -\partial_{\Lambda(\mathcal{G}), \mathcal{Q}(\mathcal{G})}
		  \left(^{\ast}\left((1 - \sigma^{-1})q_K e_I\right)\right)
		  - [\mathcal{P}_0, \alpha_{\infty}, H_{L_{\infty}}]
  \end{eqnarray*}
  The second equality follows from the short exact sequences
  \eqref{eqn:principal-units-ses}. For the last equality we first observe
  that $\Cone(\lambda_1) \simeq \Z_p[-1]$. As $L/K$ is tamely ramified,
  the central idempotent $e_I$ actually belongs to $\Lambda(\mathcal{G})$
  so that
  \[
  0 \longrightarrow \Lambda(\mathcal{G}) 
  \xrightarrow{^{\ast}((1 - \sigma^{-1}) e_I)}
  \Lambda(\mathcal{G}) \longrightarrow \Z_p \longrightarrow 0
  \]
  is a free resolution of $\Z_p$. Thus we have an equality
  \[
	  -\partial_{\Lambda(\mathcal{G}), \mathcal{Q}(\mathcal{G})}
	  \left(^{\ast}\left((1 - \sigma^{-1}) e_I\right)\right)
	  = \chi_{\Lambda(\mathcal{G}), 
	  	\mathcal{Q}(\mathcal{G})}(\Cone(\lambda_1), 0).
  \]
  Moreover, the quotient $\mathcal{P}_0 / \mathcal{P}_1$ identifies
  with the inverse limit of the residue fields of the
  $L_n$, $n \geq 0$. We therefore have isomorphisms of
  $\Lambda(\mathcal{G})$-modules
  \[
	  \mathcal{P}_0 / \mathcal{P}_1 \simeq 
	  \varprojlim_n \overline{K}[G_n/I] = 
	  \overline{K}\llbracket \mathcal{G}/I \rrbracket \simeq
	  \mathbb F_p \llbracket \mathcal{G}/I \rrbracket^{\oplus f_{K/\Q_p}},
  \]
  where we recall that $q_K = |\overline{K}| = p^{f_{K/\Q_p}}$.
  Hence there is a free resolution
  \[
	  0 \longrightarrow \Lambda(\mathcal{G})^{f_{K/\Q_p}} 
	  \xrightarrow{^{\ast}(p e_I)}
	  \Lambda(\mathcal{G})^{f_{K/\Q_p}} \longrightarrow 
	  \mathcal{P}_0 / \mathcal{P}_1
	  \longrightarrow 0.
  \]
   We conclude that
  \[
	  [\mathcal{P}_0, \alpha_{\infty}, H_{L_{\infty}}] - 
	  [\mathcal{P}_1, \alpha_{\infty}, H_{L_{\infty}}] =
	  -\partial_{\Lambda(\mathcal{G}), \mathcal{Q}(\mathcal{G})}
	  \left(^{\ast}(q_K e_I)\right).
  \]
  This shows the last equality. It follows that
  \begin{eqnarray*}
	  M_{L_{\infty}/K} - C_{L_{\infty}/K} & = &
	  [\mathcal{P}_0, \alpha_{\infty}, H_{L_{\infty}}]
	  + \partial_{\Lambda(\mathcal{G}), \mathcal{Q}(\mathcal{G})}
	  \left(^{\ast}\left((q_K - \sigma) e_I\right)\right) \\
	  & = &
	  [\mathcal{P}_0, \alpha_{\infty}, H_{L_{\infty}}]
	  + \partial_{\Lambda(\mathcal{G}), \mathcal{Q}(\mathcal{G})}
	  \left(^{\ast}(-\sigma e_I)\right).
  \end{eqnarray*}
  Here, the second equality holds, since we have
  \[
  \frac{^{\ast}\left((q_K - \sigma) e_I\right)}{^{\ast}(-\sigma e_I)}
  = {}^{\ast}\left((1 - q_K \sigma^{-1}) e_I\right)
  \in \Lambda(\mathcal{G})^{\times}.
  \]
  Let $\xi' \in K_1(\Q_p^c \otimes_{\Z_p} \Lambda(\mathcal{G}))$
  be a pre-image of 
  $U'_{L_{\infty}/K} - [\mathcal{P}_0, \alpha_{\infty}, H_{L_{\infty}}]$
  as in the proof of Theorem \ref{thm:Galois-on-coh-term}.
  By the above considerations we have to show that 
  \[
	  \tau_{L_{\infty/K}}^{(j)} \cdot 
	  \Det(\xi' \cdot {}^{\ast}(-\sigma^{-1} e_I)) \in
	  \Det(K_1(\Lambda(\mathcal{G}))).
  \]
  As we have an isomorphism $K_1(\Lambda(\mathcal{G})) \simeq
  \varprojlim_n K_1(\Z_p[G_n])$ by \cite[Proposition 1.5.1]{MR2276851},
  Lemma \ref{lem:Det-criterion} (see also Remark \ref{rem:Det-square})
  and Proposition \ref{prop:mc-Galois+} 
  imply that it suffices to show the following claim on finite level.
  Let $n \geq 0$ be an integer. Then we have an equality
  \[
	  \tau_{L_n/K}^{(j)} \cdot \Nrd_{\Q_p^c[G_n]}(\xi_n' \cdot {}^{\ast}(-\sigma^{-1} e_I)) \in \Nrd_{\Q_p^c[G_n]}(K_1(\Z_p[G_n])),
  \]
  where we put
  \[
	  \tau_{L_n/K}^{(j)} := \left(j\left((\tau_{\Q_p}(\ind_K^{\Q_p} j^{-1} \circ \chi))\right)\right)_{\chi \in \Irr_{\Q_p^c}(G_n)}
  \]
  and $\xi_n' \in K_1(\Q_p^c[G_n])$ is a pre-image of
  \[
	  U_{L_n/K}' - [(\mathcal{P}_0)_{\Gamma_n}, \alpha_n, H_{L_n}]
	  = U_{L_n/K} - [\mathcal{O}_{L_n}, \rho_n, H_{L_n}].
  \]
  However, this claim actually is a main step in the proof of the local
  epsilon constant conjecture for tamely ramified extensions
  \cite[Theorem 3.6]{MR2078894}; apply \cite[Lemma 2.7 and (3.4)]{MR2078894}
  and Taylor's fixed point theorem \cite{MR608528} 
  (see Theorem \ref{thm:fixed-points}). Note that in the notation of
  \cite{MR2078894} one has $\Nrd_{\Q_p^c[G_n]}({}^{\ast}(-\sigma e_I)) =
  \left(y(K, \chi)\right)_{\chi \in \Irr_{\Q_p^c}(G_n)}$.
  \end{proof}
  
  \begin{remark}
  	If $\mathcal{G}$ is abelian, then Theorem \ref{thm:mc-tame}
  	is the local analogue
  	of Wiles' result \cite{MR1053488} on the main conjecture
  	for totally real fields.
  \end{remark}
  
  \subsection{Maximal orders}
  Choose $x_{L_{\infty}/K} \in K_1(\mathcal{Q}^c (\mathcal{G}))$
  such that
  $\partial_{\Lambda(\mathcal{G}), \mathcal{Q}^c(\mathcal{G})}(x_{L_{\infty}/K}) = 
  -C_{L_{\infty}/K} - U'_{L_{\infty}/K} + M_{L_{\infty}/K}$.
  Then Conjecture \ref{conj:main-conjecture} asserts that
  \[
	  \Det(x_{L_{\infty}/K})^{-1} \cdot \tau_{L_{\infty}/K}^{(j)}
	  \in \Det(K_1(\Lambda(\mathcal{G}))).
  \]
  Recall from \eqref{eqn:Det-of-Iwasawa} that we have an inclusion
  \[
	  \Det(K_1(\Lambda(\mathcal{G}))) \subseteq
	  \Hom^{\ast}(R_p( \mathcal{G}), \Lambda^{c}(\Gamma_K)^{\times}).
  \]
  If $\mathcal{M}(\mathcal{G})$ is a maximal $R$-order in 
  $\mathcal{Q}(\mathcal G)$ containing $\Lambda(\mathcal G)$
  (where $R \simeq \Z_p \llbracket T \rrbracket$ is as in \S
  \ref{subsec:Iwasawa-algebras}), then by \cite[Remark H]{MR2114937}
  the bottom isomorphism in triangle \eqref{eqn:Det_triangle} induces
  an isomorphism
  \[
	  \zeta(\mathcal{M}(\mathcal{G}))^{\times} \simeq
	  \Hom^{\ast}(R_p( \mathcal{G}), \Lambda^{c}(\Gamma_K)^{\times}).
  \]
  The following result may therefore be seen as the main conjecture
  `over the maximal order'.
  
  \begin{theorem} \label{thm:mc-max-order}
  	For every $x_{L_{\infty}/K} \in K_1(\mathcal{Q}^c (\mathcal{G}))$
  	such that
  	\[
  	\partial_{\Lambda(\mathcal{G}), \mathcal{Q}^c(\mathcal{G})}(x_{L_{\infty}/K}) = 
  	-C_{L_{\infty}/K} - U'_{L_{\infty}/K} + M_{L_{\infty}/K}
  	\]
  	we have
  	\[
  	\Det(x_{L_{\infty}/K})^{-1} \cdot \tau^{(j)}_{L_{\infty}/K}
  	\in \Hom^{\ast}(R_p( \mathcal{G}), \Lambda^c(\Gamma_K)^{\times}).
  	\]
  \end{theorem}
  
  \begin{proof}
  	This follows from Theorem \ref{thm:mc-tame} by a reduction 
  	argument which mainly uses the functorial properties of
  	the conjecture (see \cite[Theorem 16]{MR2114937} for the analogue
  	in the case of the main conjecture for totally real fields).
  	We sketch the proof for convenience of the reader.
  	Let us put $f := \Det(x_{L_{\infty}/K})^{-1} \cdot \tau^{(j)}_{L_{\infty}/K}$ for simplicity. 
  	We have to show that for every $\chi \in \Irr_{\Q_p^c}(\mathcal{G})$
  	we have $f(\chi) \in \Lambda^c(\Gamma_K)^{\times}$.
  	By Brauer induction we may assume that $\mathcal{G}$
  	is abelian. In particular, we have a decomposition
  	$\mathcal{G} = H \times \Gamma_K$ with an abelian finite group $H$.
  	As we have $f(\chi \otimes \rho) = \rho^{\sharp}(f(\chi))$ for all
  	characters $\rho$ of type $W$, we may in addition assume that
  	$\chi$ is a character of type $S$, i.e.~$\chi$ actually factors
  	through $H$. Since we already know that $f(\chi) \in
  	(\Q_p^c \otimes_{\Z_p} \Lambda(\Gamma_K))^{\times}$ by
  	Proposition \ref{prop:mc-Galois+}, there is a prime element $\pi$
  	in some finite extension of $\Q_p$ such that 
  	$f(\chi) = \pi^{\mu_{\chi}} g(\chi)$ for some $\mu_{\chi} \in \Z$
  	and $g(\chi) \in \Lambda^c(\Gamma_K)^{\times}$.
  	We have to show
  	that the $\mu$-invariant $\mu_{\chi}$ of $f(\chi)$ vanishes.
  	Let us put $\Q_p(\chi) := \Q_p(\chi(h) \mid h \in H)$
  	and let $U$ be the Galois group of the extension $\Q_p(\chi)/ \Q_p$.
  	Recall that $U$ acts on $\chi$ via 
  	${}^{\sigma} \chi := \sigma \circ \chi$. Since $f$ is invariant
  	under Galois action, we may actually choose $\pi \in \Q_p(\chi)$
  	and the $\mu$-invariants $\mu_{\chi}$
  	and $\mu_{{}^{\sigma} \chi}$ coincide. Let $V$ be the inertia subgroup
  	of $U$. By the claim in the proof of \cite[Proposition 11]{MR1423032}
  	there is an integer $m \not= 0$ such that the character
  	\[
	  	\chi' := m \sum_{\sigma \in V} {}^{\sigma} \chi
  	\]
  	can be written as a sum of characters induced from cyclic subgroups of $H$
  	of order prime to $p$. Since these subgroups correspond to tamely ramified
  	subextensions, Theorem \ref{thm:mc-tame} implies that
  	$\mu_{\chi'}$ vanishes. The equality
  	$\mu_{\chi'} = m |V| \mu_{\chi}$ now gives the result.
  \end{proof}
  
  \begin{remark}
  	Theorems \ref{thm:mc-tame} and \ref{thm:mc-max-order} are the 
  	Iwasawa-theoretic analogues
  	of \cite[Theorem 3.6]{MR2078894} and \cite[Corollary 3.8]{MR2078894},
  	respectively. Theorem \ref{thm:mc-max-order} might also be seen
  	as the local analogue of \cite[Theorem 4.12]{hybrid-EIMC}
  	(see also \cite[Example 2]{MR2205173} if 
  	Iwasawa's $\mu$-invariant vanishes).
  \end{remark}
  
  \subsection{Consequences}
  We now prove a reduction step which also appears in the proof of the
  main conjecture for totally real fields (see \cite{MR2205173}).
  This result has no analogue at finite level.
  We let $\Lambda_{(p)}(\mathcal{G})$ be the Iwasawa algebra 
  $\Lambda(\mathcal{G})$ localized at the height $1$ prime ideal
  $(p)$ of $R \simeq \Z_p \llbracket T \rrbracket$.
  
  \begin{corollary} \label{cor:MC-localizes}
  	Choose $x_{L_{\infty}/K} \in K_1(\mathcal{Q}^c (\mathcal{G}))$
  	such that
  	\[
  	\partial_{\Lambda(\mathcal{G}), \mathcal{Q}^c(\mathcal{G})}(x_{L_{\infty}/K}) = 
  	-C_{L_{\infty}/K} - U'_{L_{\infty}/K} + M_{L_{\infty}/K}.
  	\]
  	Then Conjecture \ref{conj:main-conjecture} holds if and only if
  	we have
  	\[
  	\Det(x_{L_{\infty}/K})^{-1} \cdot \tau^{(j)}_{L_{\infty}/K}
  	\in \Det(K_1(\Lambda_{(p)}(\mathcal{G}))).
  	\]
  \end{corollary}
  
  \begin{proof}
  	By \cite[Theorem B]{MR2205173} we have an inclusion
  	\[
	  	\Hom^{\ast}(R_p( \mathcal{G}), \Lambda^c(\Gamma_K)^{\times}) \cap
	  	\Det(K_1(\Lambda_{(p)}(\mathcal{G}))) \subseteq
	  	\Det(K_1(\Lambda(\mathcal{G}))).
  	\]
  	Now the result follows from Theorem \ref{thm:mc-max-order}.
  \end{proof}
  
  \begin{remark}
  	In fact, \cite[Theorem B]{MR2205173} shows that one may replace
  	$\Lambda_{(p)}(\mathcal{G})$ by its $(p)$-adic completion
  	in the statement of Corollary \ref{cor:MC-localizes}.  	
  	Then this is the local analogue of \cite[Theorem A]{MR2205173}.
  \end{remark}
  
  Working over $\Lambda_{(p)}(\mathcal{G})$ rather than
  $\Lambda(\mathcal{G})$ has the big advantage that 
  the cohomology groups of the complex $\Lambda_{(p)}(\mathcal{G}) 
  \otimes_{\Lambda(\mathcal{G})}^{\mathbb L}
  K_{L_{\infty}}^{\bullet}$ are free
  $\Lambda_{(p)}(\mathcal{G})$-modules.
  
  \begin{prop}
  	The $\Lambda_{(p)}(\mathcal{G})$-modules
  	$H^i(\Lambda_{(p)}(\mathcal{G}) 
  	\otimes_{\Lambda(\mathcal{G})}^{\mathbb L}
  	K_{L_{\infty}}^{\bullet})$ are free of rank $[K:\Q_p]$
  	for $i=0,1$ and vanish otherwise.
  \end{prop}
  
  \begin{proof}
  	The cohomology vanishes outside degrees $0$ and $1$ by
  	Corollary \ref{cor:coh-at-infty-II}. Furthermore, we have
  	$H^1(K_{L_{\infty}}^{\bullet}) \simeq \Z_p \oplus H_{L_{\infty}}$.
  	The $\Lambda(\mathcal{G})$-module $\Z_p$ vanishes after localization
  	at $(p)$, whereas $H_{L_{\infty}}$ already is a free
  	$\Lambda(\mathcal{G})$-module of rank $[K:\Q_p]$. 
  	Finally, the $\Lambda_{(p)}(\mathcal{G})$-module 
  	$H^0(\Lambda_{(p)}(\mathcal{G}) 
  	\otimes_{\Lambda(\mathcal{G})}^{\mathbb L}
  	K_{L_{\infty}}^{\bullet}) \simeq U^1(L_{\infty})_{(p)}$ 
  	is free of the same rank by \cite[Corollary 4.4]{Swan-Iwasawa}
  	(we point out that the results established in \S 
  	\ref{subsec:Galois-coh} at least show 
  	that the projective dimension has to be
  	less or equal to $1$; this shows the slightly weaker result that the 
  	cohomology groups are perfect).
  \end{proof}
  
\nocite*
\bibliography{epsilon-bib}{}
\bibliographystyle{amsalpha}

\end{document}